\documentclass[reqno,12pt]{amsart}
\usepackage{amsfonts}
\usepackage{amsmath}
\usepackage{amsthm}
\usepackage{amscd}
\usepackage{latexsym}
\usepackage{amsbsy}
\usepackage{amscd}
\usepackage{dsfont}
\usepackage{bbm}
\usepackage{} 
\numberwithin{equation}{section}
\usepackage{indentfirst}
 \usepackage{color}
\usepackage{amssymb}
\usepackage{mathrsfs}
\usepackage{xy}
\xyoption{all}
\def\Ext{\mbox{\rm Ext}\,} \def\Hom{\mbox{\rm Hom}} \def\dim{\mbox{\rm dim}\,} 
\def\lr#1{\langle #1\rangle}    \def\mod{\mbox{\rm \textbf{mod}}\,}\def\Mod{\mbox{\rm mod}\,}\def\top{\mbox{\rm top}\,}
\def\Ker{\mbox{\rm Ker}\,}   \def\im{\mbox{\rm Im}\,} \def\Coker{\mbox{\rm Coker}\,}
\def\End{\mbox{\rm End}\,}\def\t{\mbox{\rm t}\,}\def\h{\mbox{\rm h}\,}
 \def\gl.{\mbox{\rm gl.}\,}
\def\diag{\mbox{\rm diag}\,}\def\ooz{\Omega}\def\dz{\delta}\def\vez{\varepsilon}\def\rad{\mbox{\rm rad}\,}\def\ad{\mbox{\rm ad}\,}\def\x{\mbox{\rm \textbf{x}}\,}\def\y{\mbox{\rm \textbf{y}}\,}
\def\Aut{\mbox{\rm Aut}\,}\def\Dim{\mbox{\underline{\dim}}\,}\def\A{\mathcal{A}\,} \def\H{\mathcal{H}\,}
\def\Der{\mbox{\rm Der}\,}
\theoremstyle{plain} 
\newtheorem{theorem}{\bf Theorem}[section]
\newtheorem*{theorem*}{\bf Theorem}
\newtheorem{lemma}[theorem]{\bf Lemma}
\newtheorem{corollary}[theorem]{\bf Corollary}
\newtheorem{proposition}[theorem]{\bf Proposition}
\newtheorem{claim}[theorem]{\bf Claim}
\theoremstyle{definition} 
\newtheorem{definition}[theorem]{\bf Definition}
\newtheorem{remark}[theorem]{\bf Remark}
\newtheorem{example}[theorem]{\bf Example}

\newcommand{\bt}{\begin{theorem}}
\newcommand{\et}{\end{theorem}}
\newcommand{\bl}{\begin{lemma}}
\newcommand{\el}{\end{lemma}}
\newcommand{\bd}{\begin{definition}}
\newcommand{\ed}{\end{definition}}
\newcommand{\bc}{\begin{corollary}}
\newcommand{\ec}{\end{corollary}}
\newcommand{\bp}{\begin{proof}}
\newcommand{\ep}{\end{proof}}
\newcommand{\bx}{\begin{example}}
\newcommand{\ex}{\end{example}}
\newcommand{\br}{\begin{remark}}
\newcommand{\er}{\end{remark}}
\newcommand{\be}{\begin{equation}}
\newcommand{\ee}{\end{equation}}
\newcommand{\ba}{\begin{align}}
\newcommand{\ea}{\end{align}}
\newcommand{\bn}{\begin{enumerate}}
\newcommand{\en}{\end{enumerate}}
\newcommand{\bcs}{\begin{cases}}
\newcommand{\ecs}{\end{cases}}
\makeatletter
\renewcommand{\section}{\@startsection{section}{1}{0mm}
  {-\baselineskip}{0.5\baselineskip}{\bf\leftline}}
\makeatother

\begin{document}

\title[Lie algebras arising from 1-cyclic perfect complexes]{Lie algebras arising from 1-cyclic perfect complexes} 

\author[Shiquan Ruan]{{Shiquan Ruan}} 
\author[Jie Sheng]{{Jie Sheng$^{*}$}}
\author[Haicheng Zhang]{{Haicheng Zhang}}

%
\thanks{$^{*}$ Corresponding author.}
\thanks{S. Ruan was supported by NSF of China (No. 11471269) and J. Sheng was supported by NSF of China (No.11301533).}
\address{
Yau Mathematical Sciences Center, Tsinghua University,
 Beijing 100084, P. R. China.\endgraf
}
\email{shiquanruan@stu.xmu.edu.cn}

\address{
Department of Applied Mathematics, College of Science, China Agricultural University,
 Beijing 100083, P. R. China.\endgraf
}
\email[Corresponding author]{shengjie@amss.ac.cn}

\address{
Yau Mathematical Sciences Center, Tsinghua University,
 Beijing 100084, P. R. China.\endgraf
}
\email{zhanghai14@mails.tsinghua.edu.cn}


\maketitle

\begin{abstract}
Let $A$ be the path algebra of a Dynkin quiver $Q$ over a finite field, and $\mathscr{P}$ be the category of projective $A$-modules. Denote by $C^1(\mathscr{P})$ the category of 1-cyclic complexes over $\mathscr{P}$, and $\tilde{\mathfrak{n}}^+$ the vector space spanned by the isomorphism classes of indecomposable and non-acyclic objects in $C^1(\mathscr{P})$. In this paper, we prove the existence of Hall polynomials in $C^1(\mathscr{P})$, and then establish a relationship between the Hall numbers for indecomposable objects therein and those for $A$-modules. Using Hall polynomials evaluated at $1$, we define a Lie bracket in $\tilde{\mathfrak{n}}^+$ by the commutators of degenerate Hall multiplication.
The resulting Hall Lie algebras provide a broad class of nilpotent Lie algebras. For example, if $Q$ is bipartite, $\tilde{\mathfrak{n}}^+$ is isomorphic to the nilpotent part of the corresponding semisimple Lie algebra; if $Q$ is the linearly oriented quiver of type $\mathbb{A}_{n}$,  $\tilde{\mathfrak{n}}^+$ is isomorphic to the free 2-step nilpotent Lie algebra with $n$-generators. Furthermore, we give a description of the root systems of different $\tilde{\mathfrak{n}}^+$. We also characterize the Lie algebras $\tilde{\mathfrak{n}}^+$ by generators and relations. When $Q$ is of type $\mathbb{A}$, the relations are exactly the defining relations.  As a byproduct, we construct an orthogonal exceptional pair satisfying the minimal Horseshoe lemma for each sincere non-projective indecomposable $A$-module.
\end{abstract}

\section{Introduction}
\subsection{Background}
In his remarkable paper \cite{R90a}, Ringel introduced the Hall algebra $\mathcal {H}(A)$ of a finite dimensional hereditary algebra $A$ over a finite field $\mathbb{F}_q$.  The Hall algebra is an associative algebra, with a basis indexed by the iso-classes (isomorphism classes) of finite dimensional $A$-modules, and  with the multiplication whose structure constants are given by the so-called Hall numbers $$g^{L}_{XY}=|\{M\subseteq L| M\simeq Y , L/M\simeq X\}|.$$ When $A$ is representation-finite, Ringel \cite{R91a} showed that the Hall numbers
are actually integral coefficient polynomials in $q$, thus called Hall polynomials. Using the Hall polynomials evaluated at $q=1$, Ringel \cite{R90,R91b} defined the degenerate Hall algebra $\mathcal {H}_1(A)$ of $A$, and proved that it is isomorphic to the positive part of the universal enveloping algebra associated to $A$. This also gives a realization of the nilpotent part of the
corresponding semisimple Lie algebra. Inspired by Ringel's idea, Peng and Xiao \cite{PX97,PX2000} provided a realization of the whole Kac--Moody Lie algebra via the root category of a hereditary algebra.

Moreover, the Hall algebra approach also provides a nice framework for the realization of quantum enveloping algebras, see \cite{Sch} for details. Actually, Ringel \cite{R90a,R92a,R95} and Green \cite{Gr95,R96} showed that the Hall algebra is a good model for the positive and Borel part of a quantum enveloping algebra. So a natural question is how one can give a Hall algebra realization of the whole quantum enveloping algebra. There have been many attempts to answer such a question (cf. \cite{X1,Xiao,Toen}).


Recently, Bridgeland \cite{Br} considered the Hall algebra of the category $C^2(\mathscr{P})$, where $C^2(\mathscr{P})$ is the category of 2-cyclic complexes of projective modules over a finite dimensional hereditary algebra $A$, and obtained a new algebra called Bridgeland's Hall algebra by taking some localization and reduction. He proved that the quantum enveloping algebra associated to  $A$ can be embedded into Bridgeland's Hall algebra. This provides a realization of the full quantum enveloping algebra. Based on Bridgeland's work, for each positive integer $m\geq2$, Chen and Deng \cite{ChenD} considered the category $C^m(\mathscr{P})$ of $m$-cyclic complexes of projective modules over a representation-finite hereditary algebra $A$. They proved the existence of Hall polynomials in $C^m(\mathscr{P})$, thus defined the generic Bridgeland's Hall algebra, and showed that it contains a subalgebra isomorphic to the integral form of the corresponding quantum enveloping algebra. Note that the methods of proving the existence of Hall polynomials in \cite{ChenD} do not work well for the case of $m=1$.

On the other hand, as the representations of a quiver $Q$ over a field $k$ have been studied thoroughly, it seems worthwhile to consider the representations of a quiver $Q$ over an arbitrary finite-dimensional $k$-algebras $D$. When $D=k[T]/(T^2)$ is the algebra of
dual numbers, Ringel and Zhang \cite{RZ} investigated the AR (Auslander--Reiten) structure of the subcategory of Gorenstein-projective modules over the
path algebra $DQ$, which is precisely the category $C^1(\mathscr{P})$ of 1-cyclic complexes of projective modules over $kQ$.
As pointed out in  \cite{RZ}, the Kac theorem also yields a correspondence between the iso-classes of indecomposable objects in the stable category of $C^1(\mathscr{P})$ and the positive roots of the Kac--Moody algebra associated to $Q$. This is also our motivation to find a new realization of (the nilpotent part of) a Kac--Moody Lie algebra using the Hall algebra of $C^1(\mathscr{P})$. The new realization may be seen as a new categorification
of the positive root system of a Kac--Moody Lie algebra.

\subsection{Main results}
From now on, for the sake of simplicity, a complex of projective modules will be also called \emph{a perfect complex}.
The present paper mainly deals with Hall Lie algebras arising from the category $C^1(\mathscr{P})$ of 1-cyclic perfect complexes over the path algebra $A=\mathbb{F}_q Q$ of a Dynkin quiver $Q$. Let $\mathfrak{P}^1(\Gamma)$ denote the set of iso-classes of objects in $C^1(\mathscr{P})$.
We remind that the existence of Hall polynomials means that Hall numbers are always integral coefficient polynomials in $q$, which are independent of the base field $\mathbb{F}_{q}$.
\begin{theorem*}[A]
For any $\lambda, \mu, \nu\in\mathfrak{P}^1(\Gamma)$, the Hall polynomial $\psi_{\mu \nu}^{\lambda}$  exists.
\end{theorem*}
Using the Hall polynomials evaluated at $1$, we obtain a Hall Lie algebra $\tilde{\mathfrak{n}}$ spanned by the iso-classes of indecomposable objects in $C^1(\mathscr{P})$. Naturally there are two Lie subalgebras of $\tilde{\mathfrak{n}}$, i.e. $\tilde{\mathfrak{n}}^{+}$ (resp. $\hat{\mathfrak{n}}$), which is spanned by the iso-classes of indecomposable and non-acyclic (resp. acyclic) objects in $C^1(\mathscr{P})$.

The structure of $\tilde{\mathfrak{n}}^{+}$ depends not only on the underlying graph but also on the orientation of the quiver $Q$. So
for each Dynkin quiver $Q$ with $n$ vertices, one can attach an $n\times n$ \emph{path matrix} $E=(a_{ij})$ as follows: If there is a path
between $i$ and $j$ in $Q$, say from $i$ to $j$, then $a_{ij}=1$ and $a_{ji}=-1$. Otherwise, $a_{ij}=a_{ji}=0$.

For each indecomposable $A$-module $M$, we fix a minimal projective resolution of $M$
$$0 \longrightarrow \Omega_M \stackrel{\delta_M}{\longrightarrow} P_M \stackrel{\epsilon_M}{\longrightarrow} M \longrightarrow 0$$
and a 1-cyclic perfect complex $$C_{M}=(P_{M}\oplus \Omega_{M}, \begin{pmatrix}0& \delta_M\\0&0\end{pmatrix}).$$
If $M$ is projective, say $M=P_{i}$, then $C_{P_{i}}=(P_{i},0)$. By abuse of language, we also denote by $C_{M}$ the iso-class of the corresponding 1-cyclic perfect complex, then all the $C_{M}$'s constitute the natural basis of the Hall Lie algebra $\tilde{\mathfrak{n}}^{+}$.
\begin{theorem*}[B]
Let $Q$ be a Dynkin quiver and  $E=(a_{ij})$ its path matrix.
The associated Lie algebra $\tilde{\mathfrak{n}}^{+}$ is generated by $\{C_{P_i}~|~1\leq i\leq n\}$, and these generators satisfy the following relations:

(a) If $|a_{ij}|=1$, $(\ad C_{P_i})^2(C_{P_j})=(\ad C_{P_j})^2(C_{P_i})=0$;

(b) If $a_{ij}a_{jk}=1$, $[C_{P_{i}},[C_{P_{j}},C_{P_{k}}]]=[C_{P_{k}},[C_{P_{i}},C_{P_{j}}]]=0$;

(c) If $a_{ij}=0$, $[C_{P_i},C_{P_j}]=0$.
\end{theorem*}
Actually the Hall Lie algebras $\tilde{\mathfrak{n}}^{+}$ provide variant nilpotent Lie algebras.
Let us pause to give a short review of nilpotent Lie algebras.

According to Levi's theorem, in characteristic zero,
the natural problem of determining all finite-dimensional Lie algebras can be broken into two parts:
the classification of semisimple Lie algebras and the classification of solvable Lie algebras.
The first one was settled by Killing and Cartan around 1890, while the second
one was reduced to the classification of nilpotent Lie algebras by Malcev in 1945. It is well known that
nilpotent Lie algebra is very important both in algebra and differential geometry.
However, compared to semisimple Lie algebras, the systematic study of nilpotent Lie algebras is very difficult.
From the examples of Malcev in \cite{M}, we know several interesting results: (1) there are uncountably many non-isomorphic
nilpotent Lie algebras over $\mathbb{C}$; (2) a nilpotent Lie algebra over $\mathbb{R}$ need not have
a rational form.

Since the classification problem of nilpotent Lie algebras is hopeless, it is natural to
restrict to certain special cases. Except for those serving as nilpotent radicals of semisimple Lie algebras,
the simplest nontrivial nilpotent Lie algebras are 2-step nilpotent Lie algebras. A Lie algebra $\mathfrak{N}$
is called \emph{2-step nilpotent} if $\mathfrak{N}^{3}=0$ while $\mathfrak{N}^{2}\neq 0$, where $\mathfrak{N}^{t}$
denotes the $t$-th term of the descending central series of $\mathfrak{N}$. Each 2-step nilpotent Lie algebra is
a homomorphism image of a \emph{free 2-step nilpotent} Lie algebra, which will be recalled in Section \ref{free2}.
\begin{theorem*}[C]
(1) If $Q$ is bipartite, i.e. each vertex is a sink or a source, then $\tilde{\mathfrak{n}}^{+}$ is isomorphic to the nilpotent part $\mathfrak{n}^{+}$ of the semisimple Lie algebra $\mathfrak{g}$ associated to $Q$.

(2) If $Q$ is the linearly oriented quiver of type $\mathbb{A}_{n}$, then $\tilde{\mathfrak{n}}^{+}$
is isomorphic to the free 2-step nilpotent Lie algebra with $n$-generators.
\end{theorem*}


Generally, different orientations of a given graph always result in different iso-classes of nilpotent Lie algebras. They all admit integral forms and share the same dimension. Since $\tilde{\mathfrak{n}}^{+}$ is a nilpotent Lie algebra of maximal rank, we describe its root system in
a unified manner following \cite{Sa}. For $C_{M}$, let $\alpha_{M}$ be the image of $P_{M}\oplus \Omega_{M}$ in the Grothendieck
group of the category of projective $A$-modules.


One can associate a new quiver $Q^{p}$ to the quiver $Q$, which has the same vertices as $Q$, while an arrow $i\rightarrow j$ in $Q^{p}$ indicates a path from $i$ to $j$ in $Q$. Then $Q^{p}$ is an acyclic quiver and let $C_{E}=(c_{ij})$ be the generalized Cartan matrix associated to $Q^{p}$. Let $\mathfrak{LC}(Q^{p})$ be the classical Hall composition Lie algebra generated by all iso-classes $u_{S^{p}_{i}}$ associated to simple $kQ^{p}$-modules $S^{p}_{i}$. Assume that $\tilde{\mathfrak{n}}^{+}$ is of nilpotency $l$ and let $\mathfrak{m}'=\mathfrak{LC}(Q^{p})/C^{l+1}\mathfrak{LC}(Q^{p})$, where $C^{t}\mathfrak{N}:=\mathfrak{N}^{t}$ also denotes the $t$-th term of the descending central series of $\mathfrak{N}$.

\begin{theorem*}[D]
(1) There exists an epimorphism of Lie algebras $\varphi: \mathfrak{m}'\rightarrow\tilde{\mathfrak{n}}^{+}$, $\overline{u_{S^{p}_{i}}}\mapsto C_{P_{i}}$ such that $\Ker \varphi$ is a homogeneous ideal of $\mathfrak{m}'$, $C^{l}\mathfrak{m}'\nsubseteq \Ker \varphi$ and
$(\ad \overline{u_{S^{p}_{i}}})^{-c_{ij}}\overline{u_{S^{p}_{j}}}\not\in \Ker \varphi$.

(2) The nilpotent Lie algebra  $\tilde{\mathfrak{n}}^{+}$ has a root space decomposition relative to its maximal torus $T$:
$$\tilde{\mathfrak{n}}^{+}=\oplus_{\alpha\in R(T)}\tilde{\mathfrak{n}}^{+}_{\alpha},$$
where $R(T)=\{\alpha_{M}|M\in \mathrm{Ind\ } kQ\}$ and $\dim \tilde{\mathfrak{n}}^{+}_{\alpha}=1$, $\forall\alpha \in R(T)$.
\end{theorem*}

Moreover, the defining relations of $\tilde{\mathfrak{n}}^{+}$ are determined for type $\mathbb{A}_{n}$.
\begin{theorem*}[E]
Let $Q$ be a quiver of type $\mathbb{A}_{n}$ and $E=(a_{ij})$ its path matrix, then $\tilde{\mathfrak{n}}^{+}$ is isomorphic to the Lie algebra
defined by generators $\{e_{i}~|~1\leq i\leq n\}$ and relations:

(a) If $|a_{ij}|=1$, $(\ad e_{i})^2(e_{j})=(\ad e_{j})^2(e_{i})=0$;

(b) If $a_{ij}a_{jk}=1$, $[e_{i},[e_{j},e_{k}]]=[e_{k},[e_{i},e_{j}]]=0$;

(c) If $a_{ij}=0$, $[e_{i},e_{j}]=0$;

\end{theorem*}

The paper is organized as follows. Firstly we give a review of known results on 1-cyclic perfect complexes in Section 2. Especially, we recall the AR structure of the category of 1-cyclic perfect complexes over $A=\mathbb{F}_q Q$. In Section 3, by using induction and covering theory, we prove that
the Hall polynomials exist in $C^1(\mathscr{P})$. Also some specific Hall numbers for 1-cyclic perfect complexes are expressed by Hall numbers for $A$-modules. Then several Lie algebras are associated to the category $C^1(\mathscr{P})$ in Section 4, of which $\tilde{\mathfrak{n}}^+$ is our main concern. We exhibit the generators and verify some relations for $\tilde{\mathfrak{n}}^+$. The proof of generators is given in Section 5. To fulfil the aim, we construct an orthogonal exceptional pair satisfying the minimal Horseshoe lemma for any sincere non-projective indecomposable $A$-module. In our view this is interesting in its own right. In Section 6, it is shown that $\tilde{\mathfrak{n}}^+$ of the quiver $\mathbb{A}_{n}^{\rightarrow}$ gives a new model for free 2-step nilpotent Lie algebras. Section 7 is devoted to the root system of $\tilde{\mathfrak{n}}^+$ for any Dynkin quiver and the defining relations of $\tilde{\mathfrak{n}}^+$ for type $\mathbb{A}_{n}$. We also relate $\tilde{\mathfrak{n}}^+$ to Ringel's classical Hall Lie algebras and complete Lie algebras.

\subsection{Notations}
Now we fix some notations used throughout the paper. Let $k=F_q$ be a field with $q$ elements. Given a finite dimensional $k$-algebra $A$, we denote by $\mod A$ the category of finite dimensional (left) $A$-modules. Denote by $D^b(A)$ the bounded derived category of $\mod{A}$. For an $A$-module $M$, we denote by $\Dim M$ the dimension vector of $M$. Let $Q=(Q_0,Q_1)$ be a finite acyclic quiver with a vertex set $Q_0=\{1,\cdots,n\}$ and an arrow set $Q_1=\{\rho~|~\t\rho\xrightarrow{\rho}\h\rho\}$, where $\t\rho$ and $\h\rho$ denote the tail and head of the arrow $\rho$, respectively. The underlying graph of $Q$ is denoted by $\Gamma$.  Let $C$ be the Cartan matrix of $\Gamma$ and $\mathfrak{g}$  be  the associated  Kac--Moody Lie algebra. Denote by $\Phi$ the root system of $\mathfrak{g}$ and $\Phi^{+}$ the set of positive roots. Let $\mathfrak{g}=\mathfrak{n}^{-}\oplus \mathfrak{h}\oplus \mathfrak{n}^{+}$ be the triangular decomposition of $\mathfrak{g}$.
For each vertex $i$ of $Q$, we denote by $S_i$ the corresponding simple $kQ$-module. Denote by $P_i$ and $I_i$ the projective cover and injective envelope of $S_i$ respectively. The \emph{Euler form} $\lr{\cdot,\cdot}:
\mathbb{Z}Q_0\times\mathbb{Z}Q_0\rightarrow\mathbb{Z}$  is defined by $$\lr{\x,\y}=\sum\limits_{i\in Q_0}x_iy_i-\sum\limits_{\rho\in Q_1}x_{\t\rho}y_{\h\rho},~\text{for}~ \x=(x_i), \y=(y_i)\in\mathbb{Z}Q_0.$$ Its symmetrization $$(\x,\y):=\lr{\x,\y}+\lr{\y,\x}$$ is called the \emph{symmetric Euler form} of $Q$. For any $kQ$-modules $M$ and $N$, $$\lr{\Dim M,\Dim N}=\dim\Hom_{kQ}(M,N)-\dim\Ext_{kQ}^1(M,N).$$
For a finite set $S$, we denote by $|S|$ its cardinality. For any object $X$ in an additive category, we denote by $\End(X)$ and $\Aut(X)$ the endomorphism ring and automorphism group of $X$ respectively. We also write $a_X$ for $|\Aut(X)|$. For a complex $M^\cdot=(\cdots\rightarrow M^{i}\xrightarrow{d^{i}} M^{i+1}\rightarrow\cdots)$ over an abelian category $\mathcal{A}$, its homology is denoted by $H_\ast(M^\cdot)$.

For any nilpotent Lie algebra $\mathfrak{N}$, let $\mathfrak{N}^{2}=[\mathfrak{N},\mathfrak{N}]$ be its derived algebra. We use both $\mathfrak{N}^{t}$ and $C^{t}\mathfrak{N}$ to denote the $t$-th term of the descending central series of $\mathfrak{N}$.

\section{Category $C^1(\mathscr{P})$ of $1$-cyclic perfect complexes}
In this section we recall from \cite{RZ,ChenD} the notion of $1$-cyclic complexes. If $A$ is hereditary, we also recall the AR structure of the category of $1$-cyclic perfect complexes.

Given an additive category $\mathcal{A}$,  let $C^b(\mathcal{A})$ be the category of bounded complexes over $\mathcal{A}$ and $K^b(\mathcal{A})$ the corresponding homotopy category. Denote by $[1]$ the shift functor of $C^b(\mathcal{A})$. A \emph{1-cyclic complex} over $\mathcal{A}$ is a pair $M^\cdot=(M,d)$ where $M\in \mathcal{A}$ and $d$ is an endomorphism of $M$ satisfying $d^2=0$. A morphism  $f:(M,d)\rightarrow (M',d')$ between two $1$-cyclic complexes is given by an $\mathcal{A}$-morphism $f:M\rightarrow M'$ such that $d'f=fd$. Two morphisms $f, g:(M,d)\rightarrow (M',d')$ are called \emph{homotopic} if there is an $\mathcal{A}$-morphism $s:M\rightarrow M'$ such that $f-g=sd+d's$. We denote by $C^1(\mathcal{A})$ the category of $1$-cyclic complexes over $\mathcal{A}$, and  ${K}^1(\mathcal{A})$ the corresponding homotopy category.
It is obvious that the shift functor $[1]$ is an automorphism of $\mathcal{C}^1(\mathcal{A})$, which also induces an automorphism of  ${K}^1(\mathcal{A})$.

Consider a finite dimensional algebra $A$ and the category $\mathscr{P}=\mathscr{P}_A$ of projective $A$-modules. For simplicity, we write $C^0(\mathscr{P})$ and $K^0(\mathscr{P})$ for $C^b(\mathscr{P}_A)$ and $K^b(\mathscr{P}_A)$ respectively. Note that $C^0(\mathscr{P})$ (resp. $C^1(\mathscr{P})$) is Frobenius, whose stable category is $K^0(\mathscr{P})$ (resp. $K^1(\mathscr{P})$).
There exists a functor $$\mathcal {F}:C^0(\mathscr{P})\rightarrow C^1(\mathscr{P})$$ taking $X^\cdot=(X^i,d^i)_{i\in\mathbb{Z}}$ to $(X,d)$, where $X=\bigoplus\limits_{i\in\mathbb{Z}}X^i, d=\diag\{d^i~|~i\in\mathbb{Z}\}$. It is easy to know that $\mathcal {F}$ is an exact functor. The following lemma is proved by Gorsky \cite[Lem. 9.4]{Go} for $C^2(\mathscr{P})$. The same proof holds for $m\geq 1$ and we only need it for $C^1(\mathscr{P})$.
\bl\rm{(\cite[Prop. 2.2]{Zhao})}\label{Ext} If $X^\cdot, Y^\cdot\in C^1(\mathscr{P})$, then for any $i\geq 1$ $$\Ext_{C^1(\mathscr{P})}^i(X^\cdot,Y^\cdot)\cong
\Hom_{K^1(\mathscr{P})}(X^\cdot, Y^\cdot).$$
\el
By Lemma \ref{Ext}, we know that $C^1(\mathscr{P})$ has infinite global dimension.
From now on, we assume that $A$ is hereditary. Thus $C^1(\mathscr{P})$
is closed under subobjects and extensions. However, it is not abelian in general. Given a morphism
$f: \Omega\rightarrow P$ in $\mathscr{P}$, we define objects
$$\tilde{C}_f=\cdots\rightarrow 0\rightarrow\Omega\xrightarrow{f} P\rightarrow 0\rightarrow \cdots\in C^0(\mathscr{P}),\quad\quad
C_f=\scriptsize{\left(P\oplus \Omega,\begin{pmatrix}0&f\\0&0\end{pmatrix}\right)}\in C^1(\mathscr{P}),$$
where $P$ and $\Omega$ in $\tilde{C}_f$ are of degree $0$ and $-1$ respectively.
So for each projective $A$-module $P$, we have objects $\tilde{K}_P:=\tilde{C}_{Id_P}\in C^0(\mathscr{P})$ and $K_P:=C_{Id_P}\in C^1(\mathscr{P})$. For each $A$-module $M$, we fix a minimal projective resolution of $M$
\begin{equation}\label{mini proj res}
0 \longrightarrow \Omega_M \stackrel{\delta_M}{\longrightarrow} P_M \stackrel{\epsilon_M}{\longrightarrow} M \longrightarrow 0.
\end{equation}
We set $\tilde{C}_M:=\tilde{C}_{\delta_M}$ and $C_M:=C_{\delta_M}$. Since the minimal projective resolution is unique up to isomorphism, $\tilde{C}_M$ and $C_M$ are well-defined up to isomorphism. By \cite{Br,RZ}, we know that the functor $\mathcal {F}$ defined above is a Galois $G$-covering in the sense of \cite[Def. 2.8]{Cover}, where $G$ is the infinite cyclic group generated by $[1]$.
Using the covering functor $\mathcal{F}$, we have the following lemma.
\bl\rm{(\cite[Th. 1]{RZ})}\label{indec. obj.s}

(1)  $\{C_M, K_P|\ M\ \text{is indecomposable}, P\ \text{is indecomposable projective}\}$ is a complete set of indecomposable objects in $C^1(\mathscr{P})$. Moreover, $K_P$'s are all indecomposable projective-injective objects in $C^1(\mathscr{P})$.

(2) The homotopy category $K^1(\mathscr{P})$ is equivalent to the orbit category $D^b(A)/[1]$.
\el

Now we recall the AR structure of $C^1(\mathscr{P})$ (cf. \cite{RZ,ChenD}). Let $A=kQ$ be a path algebra of a finite acyclic quiver $Q$. Let $M$ be an indecomposable non-projective $A$-module, then we have an almost split sequence in $\mod A$
\be
\eta_M: 0 \longrightarrow \tau M \stackrel{\varphi}{\longrightarrow} E \stackrel{\psi}{\longrightarrow} M \longrightarrow 0,
\ee where $\tau$ is the AR-translation.  By the Horseshoe Lemma (cf. \cite{Weibel}),  the minimal projective resolutions of $M$ and $\tau M$ as in $(\ref{mini proj res})$ yield the following commutative diagram with exact rows and columns:
\begin{equation} \label{horse lemma diagram}
\xymatrix{
&0\ar[d] & 0\ar[d] & 0\ar[d]\\
0\ar[r] & \ooz_{\tau M}\ar[d]^-{\dz_{\tau M}}\ar[r]^{\iota\quad} & \ooz_{\tau M}\oplus
\ooz_M \ar[d]_{\widetilde\dz}\ar[r]^{\quad p} & \ooz_M\ar[d]^-{\dz_M}\ar[r] & 0\\
0\ar[r] & P_{\tau M}\ar[d]^-{\vez_{\tau M}}\ar[r]^{\iota\quad} & P_{\tau M}\oplus P_M\ar[d]_{\widetilde\vez}\ar[r]^{\quad p} & P_M\ar[d]^-{\vez_M}\ar[r] & 0\\
0\ar[r] & {\tau M} \ar[d]\ar[r]^{\varphi} & E\ar[d]\ar[r]^{\psi} & M\ar[d]\ar[r] & 0\\
 & 0 & 0&0\\
 }
\end{equation}
where $\iota$ and $p$ denote the canonical inclusion and projection respectively. Thus we have an exact sequence in $C^1(\mathscr{P})$ \be
\widetilde{\eta}_M: 0 \longrightarrow C_{\tau M} \stackrel{\widetilde{\varphi}}{\longrightarrow} \widetilde{E} \stackrel{\widetilde{\psi}}{\longrightarrow} C_M \longrightarrow 0,
\ee where $\widetilde{E}=C_{\widetilde{\dz}}$. By \cite[Prop. 2.6]{ChenD}, it is an almost split sequence in $C^1(\mathscr{P})$. Moreover, if $\tau M$ is not simple, then $\widetilde{E}\cong C_E$. Otherwise, $\widetilde{E}\cong C_E\oplus K_{P_{\tau M}}$.

Let $i$ be a vertex of $Q$. If $i$ is a source,  we have an almost split sequence of the form
 (cf. \cite[5.2. (a)]{RZ}) \be
0 \longrightarrow C_{I_i} \stackrel{}{\longrightarrow} C_{\rad P_i}\oplus K_{P_i} \stackrel{}{\longrightarrow} C_{P_i} \longrightarrow 0.
\ee

Otherwise, we have an almost split sequence of the form (cf. \cite[5.2. (b)]{RZ})\be
0 \longrightarrow C_{I_i} \stackrel{}{\longrightarrow} C_{\rad P_i}\oplus C_{{I_i}/{S_i}} \stackrel{}{\longrightarrow} C_{P_i} \longrightarrow 0.
\ee Note that each almost split sequence with $K_{P_i}$ a direct summand of the middle term starts with $C_{S_i}$ and ends in $C_{\tau^{-1}S_i}$ .

\begin{example}\label{lizi}
Let $Q$ be the quiver of type $\mathbb{A}_3$
$$1\longrightarrow 2\longrightarrow 3.$$
The AR-quiver of $C^1(\mathscr{P})$ is as follows:
\be\label{AR-quiver}\xymatrix@!=0.8pc{&&&&\ar@{.}[rd]\\
&&& C_{P_1}\ar@{.}[ru]\ar[rd]\ar@{--}[rr] & & C_{P_3}\ar[rd]\\
&& C_{P_2}\ar[ru]\ar[rd]\ar@{--}[rr]&&C_{I_2}\ar[ru]\ar[rd]\ar@{--}[rr]& & C_{P_2}\ar[rd]\\
& C_{P_3}\ar[ru]\ar[rd]\ar@{--}[rr]&&C_{S_2}\ar[ru]\ar[rd]\ar@{--}[rr]&&C_{S_1}\ar[ru]\ar[rd]\ar@{--}[rr]& & C_{P_1}\ar@{.}[rd]\\
\ar@{.}[ru]& & K_{P_3}\ar[ru]& &K_{P_2}\ar[ru]& & K_{P_1}\ar[ru]& &}\ee where the horizontal dashed lines denote different $\tau$-orbits, and the dotted line on the left and that on the right are identified oppositely. Thus the AR-quiver of $C^1(\mathscr{P})$ is like a M\"{o}bius strip. Each $K_{P_i}$ is projective-injective in $C^1(\mathscr{P})$, and occupies its own orbit alone.
\end{example}

\begin{proposition}\label{formula}
Given $A$-modules $M$ and $N$, take their minimal projective resolutions as in $(\ref{mini proj res})$. Let $P$ and $\Omega$ be two projective $A$-modules.
Then

$(1)$~$$\Hom_{C^1(\mathscr{P})}(C_{M},C_{N})\cong\Hom_A(M,N)\oplus\Hom_A(P_{M},\Omega_N)\oplus\Hom_A(\Omega_M,P_N).$$

$(2)$~\begin{equation*}\begin{split}&\Hom_{C^1(\mathscr{P})}(K_{P},C_{M})\cong\Hom_A(P,\Omega_M \oplus P_M),\\
&\Hom_{C^1(\mathscr{P})}(C_{M},K_{P})\cong\Hom_A(\Omega_M\oplus P_M,P),\\
&\Hom_{C^1(\mathscr{P})}(K_{P},K_{\Omega})\cong\Hom_A(P,\Omega)\oplus\Hom_A(P,\Omega).\end{split}\end{equation*}

$(3)$~$$\Ext_{C^1(\mathscr{P})}^1(C_M,C_N)\cong \Hom_A(M,N)\oplus\Ext_A^1(M,N).$$

$(4)$~For any $i\in\mathbb{Z}$,
\begin{equation}\label{ExtZ}\begin{split}\Ext_{C^0(\mathscr{P})}^1(\tilde{C}_M,\tilde{C}_N[i])\cong\begin{cases} \Hom_A(M,N),\;\;&\text{if~$i=-1$};\\
                    \Ext_A^1(M,N),\;\;&\text{if~$i=0$};\\
                     0, &\text{otherwise.}\end{cases}\qquad\end{split}\end{equation}

$(5)$~$a_{K_P}=a_P\cdot|\End_A(P)|$ and $a_{C_M}=a_M\cdot|\Hom_A(P_M,\Omega_M)|\cdot|\Hom_A(\Omega_M,P_M)|.$
\end{proposition}
\bp
$(1)$ Take an arbitrary $f=\left({\begin{smallmatrix}a&b\\c&d\end{smallmatrix}}\right)\in\Hom_{C^1(\mathscr{P})}(C_{M},C_{N})$. Thus, $${\begin{pmatrix}a&b\\c&d\end{pmatrix}}
{\begin{pmatrix}0&\delta_M\\0&0\end{pmatrix}}={\begin{pmatrix}0&\delta_N\\0&0\end{pmatrix}}
{\begin{pmatrix}a&b\\c&d\end{pmatrix}}.$$
Hence $c=0$ and $a\delta_M=\delta_Nd$. Set $S=\{(a,d)~|~a:P_M\to P_N,~d:\Omega_M\to \Omega_N,~a\delta_M=\delta_Nd\}$.
Define a map $$\varphi: S\longrightarrow \Hom_A(M,N),~(a,d)\mapsto h_1,$$ where $h_1$ is uniquely determined by the commutative diagram of short exact sequences:
$$\xymatrix{0\ar[r]&\Omega_M\ar[r]^-{\delta_M}\ar[d]_-d&P_M\ar[r]^-{\epsilon_M}\ar[d]^-a&M\ar[r]\ar@{.>}[d]^-{h_1}&0\\
0\ar[r]&\Omega_N\ar[r]^-{\delta_N}&P_N\ar[r]^-{\epsilon_N}&N\ar[r]&0.}$$
It is obvious that $\varphi$ is surjective.
\begin{equation*}\begin{split}\Ker \varphi&=\{(a,d)~|~a:P_M\to P_N,~d:\Omega_M\to \Omega_N,~a\delta_M=\delta_Nd,~\epsilon_Na=0\}\\
&=\{(a,d)~|~a:P_M\to P_N,~d:\Omega_M\to \Omega_N,~a\delta_M=\delta_Nd,~a=\delta_Nh_2~\textrm{for~some}~ h_2:P_M\to \Omega_N\}\\
&=\{(a,d)~|~a:P_M\to P_N,~d:\Omega_M\to \Omega_N,~a=\delta_Nh_2,~d=h_2\delta_M~\textrm{for~some}~ h_2:P_M\to \Omega_N\}.
\end{split}\end{equation*}
Hence $\Ker \varphi\cong\Hom_A(P_M,\Omega_N)$. So $\Hom_{C^1(\mathscr{P})}(C_{M},C_{N})\cong\Hom_A(M,N)\oplus\Hom_A(P_{M},\Omega_N)\oplus\Hom_A(\Omega_M,P_N).$

Similarly, one can prove $(2)$.

$(3)$\begin{equation*}
\begin{split}
\Ext_{C^1(\mathscr{P})}^1(C_M,C_N)&\cong\Hom_{K^1(\mathscr{P})}(C_M,C_N[1])\\
&\cong\Hom_{D^b(A)/[1]}(M,N[1])\\
&\cong\bigoplus\limits_{t\in\mathbb{Z}}\Hom_{D^b(A)}(M,N[t+1])\\
&\cong\Hom_A(M,N)\oplus\Ext_A^1(M,N),
\end{split}
\end{equation*}since $A$ is hereditary.

$(4)$\begin{equation*}
\begin{split}
\Ext_{C^0(\mathscr{P})}^1(\tilde{C}_M,\tilde{C}_N[i])&\cong\Hom_{K^0(\mathscr{P})}(\tilde{C}_M,\tilde{C}_N[i+1])\\
&\cong\Hom_{D^b(A)}(M,N[i+1])\\
&\cong\Ext_A^{i+1}(M,N).
\end{split}
\end{equation*}Since $A$ is hereditary and $\Ext_A^{i}(M,N)=0$ for any $i<0$, (\ref{ExtZ}) is proved.

$(5)$ It is easy that any $f\in\End(K_P)$ has the form $f=\left({\begin{smallmatrix}a&b\\0&a\end{smallmatrix}}\right)$ for some $a,~b\in\End(P)$. So $a_{K_P}=a_P\cdot|\End(P)|$. We also have
$$\Aut(C_M)=\{\left({\begin{smallmatrix}a&b\\0&d\end{smallmatrix}}\right)~|~a\delta_M=\delta_Md,~a\in\Aut(P_M),~b\in\Hom_A(\Omega_M,P_M),~d\in\Aut(\Omega_M)\}.$$
So $a_{C_M}=|G|\cdot|\Hom_A(\Omega_M,P_M)|$, where $G=\{(a,d)~|~a\delta_M=\delta_Md,~a\in\Aut(P_M),~d\in\Aut(\Omega_M)\}$.
Obviously $G$ is a subgroup of $\Aut(P_M)\times\Aut(\Omega_M)$.

Consider the map $\psi: G\longrightarrow\Aut(M),~(a,d)\mapsto h$,
where $h$ is uniquely determined by the commutative diagram of short exact sequences:
$$\xymatrix{0\ar[r]&\Omega_M\ar[r]^-{\delta_M}\ar[d]_-d&P_M\ar[r]^-{\epsilon_M}\ar[d]^-a&M\ar[r]\ar@{.>}[d]^-{h}&0\\
0\ar[r]&\Omega_M\ar[r]^-{\delta_M}&P_M\ar[r]^-{\epsilon_M}&M\ar[r]&0.}$$ By 5-Lemma in \cite[Ex. 1.3.3]{Weibel}, $h$ is an automorphism.
It is easy to see that $\psi$ is a homomorphism of groups.

For any $h\in\Aut(M)$, there exists $a\in\End(P_M)$ such that $\epsilon_Ma=h\epsilon_M$. Then $a$ is surjective since $\epsilon_M$ is an essential epimorphism. Thus, $a$ is an automorphism and induces an automorphism $d$ making the above diagram commutative. So $\psi$ is surjective.

Note $\Ker \psi=\{(a,d)~|~a\delta_M=\delta_Md,~a\in\Aut(P_M),~d\in\Aut(\Omega_M),~\epsilon_Ma=\epsilon_M\}.$
For any $(a,d)\in\Ker \psi$, $\epsilon_M(a-1)=0$, so there is a unique $h'\in\Hom_A(P_M,\Omega_M)$ such that $a-1=\delta_Mh'$. It is straightforward $\Ker \psi \cong \Hom_A(P_M,\Omega_M)$.
Hence  $a_{C_M}=a_M\cdot|\Hom_A(P_M,\Omega_M)|\cdot|\Hom_A(\Omega_M,P_M)|$.
\ep

Recall Example $\ref{lizi}$. Spread the AR-quiver $(\ref{AR-quiver})$ of $C^1(\mathscr{P})$ and consider its universal covering:
$$
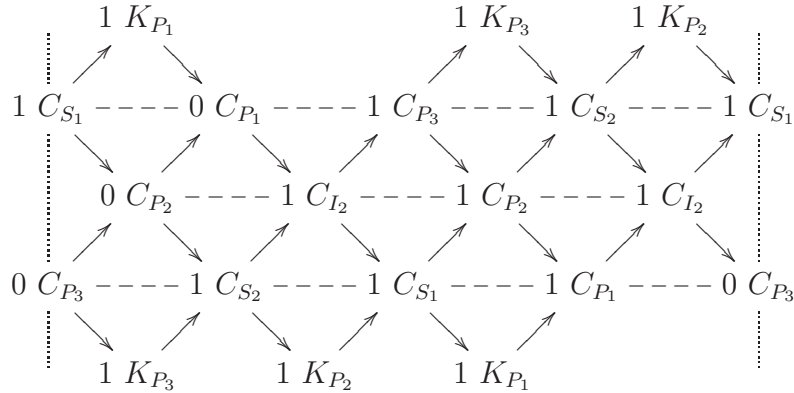
\xymatrix@!=0.8pc{\ar@{.}[d]&{1}~K_{P_1}\ar[rd]&&&&1~K_{P_3}\ar[rd]&&1~K_{P_2}\ar[rd]&\ar@{.}[d]\\
{1}~C_{S_1}\ar@{.}[dd]\ar[ur]\ar[rd]\ar@{--}[rr]&&{0}~C_{P_1}\ar[rd]\ar@{--}[rr]&&1~C_{P_3}\ar[rd]\ar[ru]\ar@{--}[rr]
&&1~C_{S_2}\ar[rd]\ar[ru]\ar@{--}[rr]&&1~C_{S_1}\ar@{.}[dd]\\
&{0}~C_{P_2}\ar[ur]\ar[rd]\ar@{--}[rr]&&1~C_{I_2}\ar[ur]\ar[rd]\ar@{--}[rr]&&1~C_{P_2}\ar[ur]\ar[rd]\ar@{--}[rr]
&&1~C_{I_2}\ar[ur]\ar[rd]&\\
{0}~C_{P_3}\ar@{.}[d]\ar[ru]\ar[rd]\ar@{--}[rr]&&1~C_{S_2}\ar[ur]\ar[rd]\ar@{--}[rr]
&&1~C_{S_1}\ar[ur]\ar[rd]\ar@{--}[rr]&&1~C_{P_1}\ar[ru]\ar@{--}[rr]&&0~C_{P_3}\ar@{.}[d]\\
&1~K_{P_3}\ar[ur]&&1~K_{P_2}\ar[ur]&&1~K_{P_1}\ar[ur]&&&}$$
where $0$ or $1$ indicates the dimension of $\Hom(K_{P_3},X)$ for each indecomposable object $X$ in the quiver.
By this hammock, we know that
\small{\begin{equation*}\begin{split}&\dim\Hom_{C^1(\mathscr{P})}(K_{P_3},C_{S_1})=\dim\Hom_{C^1(\mathscr{P})}(K_{P_3},C_{S_2})=\dim\Hom_{C^1(\mathscr{P})}(K_{P_3},C_{I_2})
=1+1=2,\\
&\dim\Hom_{C^1(\mathscr{P})}(K_{P_3},C_{P_1})=\dim\Hom_{C^1(\mathscr{P})}(K_{P_3},C_{P_2})=\dim\Hom_{C^1(\mathscr{P})}(K_{P_3},C_{P_3})=1,\\
&\dim\Hom_{C^1(\mathscr{P})}(K_{P_3},K_{P_1})=\dim\Hom_{C^1(\mathscr{P})}(K_{P_3},K_{P_2})=\dim\Hom_{C^1(\mathscr{P})}(K_{P_3},K_{P_3})=1+1=2.
\end{split}\end{equation*}}
These dimensions coincide with the ones obtained in Proposition $\ref{formula}$.

\section{Hall polynomials for $1$-cyclic perfect complexes}

In what follows, we assume that $Q$ is a Dynkin quiver, that is, $Q$ is of $\mathbb{ADE}$ type. Let $\Gamma$ be the underlying graph of $Q$. For each prime power $q$ ($\neq 1$ by convention), we  denote by  $A=A(q)$ the path algebra of $Q$ over the finite field $k=\mathbb{F}_q$.

First of all, we recall the notions of Ringel--Hall algebras for exact categories from \cite{Hubery}.
Let $\mathcal{A}$ be a finitary and skeletally small exact $k$-category and $W_{XY}^Z$ the set of all conflations $Y\xrightarrow{\varphi} Z\xrightarrow{\psi} X$. The group $G:=\Aut{X}\times\Aut{Y}$ acts on $W_{XY}^Z$ via
$$\xymatrix{Y\ar[r]^{\varphi}\ar[d]_{g}&Z\ar[r]^{\psi}\ar@{=}[d]&X\ar[d]^{f}\\
Y\ar[r]^{\overline{\varphi}}&Z\ar[r]^{\overline{\psi}}&X.}$$
We denote the orbit space by $V_{XY}^Z$. This action is free, and we define $$F_{X Y}^{Z}:=|V_{X Y}^{Z}|=\frac{|W_{X Y}^{Z}|}{a_Xa_Y}.$$

\begin{definition}\label{def of Hall}
Let $\A$ be a finitary and skeletally small $k$-exact category.
The \emph{Ringel--Hall algebra} $\H(\A)$ of $\A$ is the vector space over $\mathbb{C}$ with basis the iso-classes $[X]$ of objects in $\A$, and with the multiplication given by
\[[X]\cdot[Y] = \sum\limits_{[Z]} F_{X,Y}^{Z} [Z].\]
\end{definition}
Note that the sum is finite and $[0]$ is the unit for the multiplication.
\begin{theorem}{\rm(\cite[Th. 3]{Hubery})}
The Ringel--Hall algebra $\H(\A)$ of a finitary and skeletally small $k$-exact category $\A$ is an associative and unital algebra.
\end{theorem}

By Riedtmann--Peng formula \cite{Riedt,Peng}, for any objects $X,Y,Z$ in $\A$, we have
$$F_{X Y}^{Z}=\frac{|\Ext_{\A}^1(X,Y)_{Z}|}{|\Hom_{\A}(X,Y)|}
\frac{a_{Z}}{a_{X}a_{Y}},$$ where $\Ext_{\A}^1(X,Y)_{Z}$ denotes the subset of $\Ext_{\A}^1(X,Y)$ consisting of equivalence classes of conflations with middle terms isomorphic to $Z$.  The following proposition is an adaptation of \cite[Th. 4]{Hubery}.

\begin{proposition}\label{F-F}
Let $\A$ be a finitary and skeletally small $k$-exact category and $X,Y,Z\in\A.$ If $X, Y$ are indecomposable and $Z$ is decomposable, then $$F_{XY}^Z-F_{YX}^Z\equiv 0~\Mod~(q-1).$$
\end{proposition}

\subsection{}
Now we return to the exact category $C^1(\mathscr{P})$. For any objects $X^\cdot,Y^\cdot$ in $C^1(\mathscr{P})$, $\Hom_{C^1(\mathscr{P})}(X^\cdot,Y^\cdot)$ is a finite dimensional $k$-vector space. Thus $\Ext_{C^1(\mathscr{P})}^1(X^\cdot,Y^\cdot)$ is also finite-dimensional by Lemma $\ref{Ext}$. Let $\mathcal {H}(C^1(\mathscr{P}))$ be the Ringel--Hall algebra of the exact category $C^1(\mathscr{P})$.



By a well-known theorem of Gabriel \cite{Gabriel1,Gabriel2}, the correspondence $M\mapsto\Dim M$ induces a bijection between the set of isoclasses of indecomposable $A$-modules and the set of positive roots $\Phi^{+}$ of the simple Lie algebra $\mathfrak{g}$ associated with $\Gamma$. For each $\alpha\in\Phi^{+}$, let $M_q(\alpha)$ denote the corresponding indecomposable $A$-module. Denote the set of simple roots  by $\{\alpha_i~|~1\leq i\leq n\}$, so $M_q(\alpha_i)\cong S_i, i=1,\cdots,n$. For each $1\leq i\leq n$, let $\beta_i$ be the root in $\Phi^{+}$ such that $M_q(\beta_i)\cong P_i$.

By Lemma \ref{indec. obj.s}, the set \begin{equation}\label{allindec}\{C_{M_q(\alpha)}, K_{M_q(\beta_i)}~|~\alpha\in\Phi^{+}, 1\leq i\leq n\}\end{equation}
is a complete set of indecomposable objects in $C^1(\mathscr{P})$. Set $I=\{1,\cdots,n\}$, $\mathcal {I}^1(\Gamma)=\Phi^{+}\cup I,~\mathcal {I}^0(\Gamma)=(\Phi^{+}\cup I)\times\mathbb{Z}$, and define \begin{equation*}\begin{split}\mathfrak{P}^1(\Gamma)=\{\lambda:\mathcal {I}^1(\Gamma)\rightarrow\mathbb{N}\},~
\mathfrak{P}^0(\Gamma)=\{\lambda:\mathcal {I}^0(\Gamma)\rightarrow\mathbb{N}~|~\mbox{supp} \lambda~ \mbox{is finite}\},\end{split}\end{equation*} where supp$\lambda$ denotes the set of all $x\in\mathcal {I}^0(\Gamma)$ satisfying $\lambda(x)\neq 0$. By the Krull--Schmidt theorem, the correspondence sending $\lambda\in\mathfrak{P}^1(\Gamma)$ to $$C(\lambda)=C_q(\lambda)=(\bigoplus\limits_{\alpha\in\Phi^{+}}\lambda(\alpha){C_{M_q(\alpha)}})\bigoplus(\bigoplus\limits_{1\leq i\leq n}\lambda(i){K_{M_q(\beta_i)}})$$ induces a bijection from $\mathfrak{P}^1(\Gamma)$ to the set of isoclasses of objects in $C^1(\mathscr{P})$, and the correspondence sending $\lambda\in\mathfrak{P}^0(\Gamma)$ to $$\tilde{C}(\lambda)=\tilde{C}_q(\lambda)=(\bigoplus\limits_{(\alpha,t)\in\Phi^{+}\times\mathbb{Z}}\lambda(\alpha,t){\tilde{C}_{M_q(\alpha)}[t]})\bigoplus
(\bigoplus\limits_{(i,t)\in I\times\mathbb{Z}}\lambda(i,t){\tilde{K}_{M_q(\beta_i)}[t]})$$ induces a bijection from $\mathfrak{P}^0(\Gamma)$ to the set of isoclasses of objects in $C^0(\mathscr{P})$.
For any $\lambda_1, \lambda_2\in\mathfrak{P}^1(\Gamma)$ (resp. $\mathfrak{P}^0(\Gamma)$), we define the addition $\lambda_1\oplus\lambda_2$ by setting $$(\lambda_1\oplus\lambda_2)(x)=\lambda_1(x)+\lambda_2(x) ~\mbox {for all}~ x \in\mathfrak{P}^1(\Gamma)~ (\mbox{resp.}~\mathfrak{P}^0(\Gamma)).$$
An element $\lambda\in\mathfrak{P}^1(\Gamma)$ (resp. $\lambda\in\mathfrak{P}^0(\Gamma)$) is said to be \emph{indecomposable} if $C_q(\lambda)$ (resp. $\tilde{C}_q(\lambda)$) is indecomposable and \emph{decomposable} otherwise. We say that $\lambda$ is \emph{acyclic} if $C_q(\lambda)$ (resp. $\tilde{C}_q(\lambda)$) is an acyclic complex. Here a complex is called \emph{acyclic} if its homology groups vanish.

By Lemma \ref{indec. obj.s}(1), the covering functor $\mathcal {F}:C^0(\mathscr{P})\rightarrow C^1(\mathscr{P})$ is dense. So the map $\gamma: \mathfrak{P}^0(\Gamma)\rightarrow\mathfrak{P}^1(\Gamma), \lambda\mapsto\gamma(\lambda)$ defined by $$\gamma(\lambda)(x)=\sum_{t\in\mathbb{Z}}\lambda(x,t),~ \forall x\in\mathfrak{P}^1(\Gamma),$$ is surjective. By definition, $\mathcal {F}(\tilde{C}_q(\lambda))\cong C_q(\gamma(\lambda))$.

\begin{remark}
There is a bijection from the set of functions $\lambda: \Phi^+\rightarrow\mathbb{N}$ to the set of isoclasses of $A$-modules by sending $\lambda\mapsto[M_q(\lambda)]$, where
$$M_q(\lambda)=\bigoplus\limits_{\alpha\in\Phi^+}\lambda(\alpha)M_q(\alpha)\in\mod A.$$
\end{remark}

\begin{definition}
Let $\lambda, \mu, \nu\in\mathfrak{P}^1(\Gamma)$. If there exists a polynomial $\psi_{\mu \nu}^{\lambda}(x)\in\mathbb{Z}[x]$ such that for each prime power $q$,
$$\psi_{\mu \nu}^{\lambda}(q)=F_{C_q(\mu) C_q(\nu)}^{C_q(\lambda)},$$
then we say that the \emph{Hall polynomial} $\psi_{\mu \nu}^{\lambda}$ exists for $\lambda, \mu, \nu$.
\end{definition}

For any $\mu, \nu\in\mathfrak{P}^1(\Gamma)$, it is well-known (cf. \cite[Sect. 2]{R90} and \cite[Lem. 3.5]{ChenD}) that the dimensions $\dim \Hom_{C^1(\mathscr{P})}(C_q(\mu),C_q(\nu))$ and $\dim \Hom_{K^1(\mathscr{P})}(C_q(\mu),C_q(\nu))$ only depend on $\mu$ and $\nu$, but not on $q$. Moreover, there exists a monic polynomial $\mathfrak{a}_\mu(x)\in\mathbb{Z}[x]$ such that for each prime power $q$, $\mathfrak{a}_\mu(q)=a_{C_q(\mu)}$. By Riedtmann--Peng formula, it follows that for any $\lambda, \mu, \nu\in\mathfrak{P}^1(\Gamma)$, the Hall polynomial $\psi_{\mu \nu}^{\lambda}$ exists if and only if there exists a polynomial $\varepsilon_{\mu \nu}^{\lambda}(x)\in\mathbb{Z}[x]$ such that for each prime power $q$,
$\varepsilon_{\mu \nu}^{\lambda}(q)=|\Ext_{C^1(\mathscr{P})}^1(C_q(\mu), C_q(\nu))_{C_q(\lambda)}|.$

\begin{theorem}\label{Hallpoly}
For any $\lambda, \mu, \nu\in\mathfrak{P}^1(\Gamma)$, the Hall polynomial $\psi_{\mu \nu}^{\lambda}$  exists.
\end{theorem}
\bp
Firstly, we claim that the Hall polynomial $\psi_{\mu \nu}^{\lambda}$ exists for $\lambda, \mu, \nu\in\mathfrak{P}^1(\Gamma)$ with $\mu$ or $\nu$ indecomposable.

We only need to prove that there exists a polynomial $\varepsilon_{\mu \nu}^{\lambda}(x)\in\mathbb{Z}[x]$ such that for each prime power $q$,
$$\varepsilon_{\mu \nu}^{\lambda}(q)=|\Ext_{C^1(\mathscr{P})}^1(C_q(\mu), C_q(\nu))_{C_q(\lambda)}|.$$
We just prove the case in which $\mu$ is indecomposable, the case of $\nu$  being indecomposable can be treated in a dual way.

If $\mu$ or $\nu$ is acyclic, $\Ext_{C^1(\mathscr{P})}^1(C_q(\mu), C_q(\nu))=0$.

So we concentrate on the case that $C_q(\mu)$ and  $C_q(\nu)$ have no acyclic summands. By Proposition \ref{formula}$(4)$,
$$\Ext_{C^1(\mathscr{P})}^1(C_q(\mu), C_q(\nu))\cong\Hom_A(M_q(\mu), M_q(\nu))\oplus\Ext_A^1(M_q(\mu), M_q(\nu)).$$
Since $A$ is representation-directed, we can write $M_q(\nu)=M_q(\nu_1)\oplus M_q(\nu_2)$ for some $\nu_1, \nu_2: \Phi^+\rightarrow \mathbb{N}$ such that $$\Hom_A(M_q(\mu), M_q(\nu_1))=0~~\textrm{and}~~\Ext_A^1(M_q(\mu), M_q(\nu_2))=0.$$

Hence,
\begin{equation*}
\begin{split}
\Ext_{C^1(\mathscr{P})}^1(C_q(\mu), C_q(\nu))&\cong\Hom_A(M_q(\mu), M_q(\nu_2))\oplus\Ext_A^1(M_q(\mu), M_q(\nu_1))\\
&\cong\Ext_{C^0(\mathscr{P})}^1(\tilde{C}_{M_q(\mu)}, \tilde{C}_{M_q(\nu_2)}[-1])\oplus\Ext_{C^0(\mathscr{P})}^1(\tilde{C}_{M_q(\mu)}, \tilde{C}_{M_q(\nu_1)})\\
&\cong\Ext_{C^0(\mathscr{P})}^1(\tilde{C}_{M_q(\mu)}, \tilde{C}_{M_q(\nu_1)}\oplus\tilde{C}_{M_q(\nu_2)}[-1]).
\end{split}
\end{equation*}
Obviously, $C_q(\nu)=C_{M_q(\nu_1)}\oplus C_{M_q(\nu_2)}\cong C_{M_q(\nu_1)}\oplus C_{M_q(\nu_2)}[-1]$ and $\mathcal {F}(\tilde{C}_{M_q(\nu_1)}\oplus\tilde{C}_{M_q(\nu_2)}[-1])=C_q(\nu)$.
Choose $\tilde{\mu}, \tilde{\nu}\in\mathfrak{P}^0(\Gamma)$ such that $\tilde{C}_q(\tilde{\mu})=\tilde{C}_{M_q(\mu)}, \tilde{C}_q(\tilde{\nu})=\tilde{C}_{M_q(\nu_1)}\oplus\tilde{C}_{M_q(\nu_2)}[-1]$,
then we have $$\Ext_{C^1(\mathscr{P})}^1(C_q(\mu), C_q(\nu))\cong\Ext_{C^0(\mathscr{P})}^1(\tilde{C}_q(\tilde{\mu}), \tilde{C}_q(\tilde{\nu})).$$

 Thus we have a bijection $$\xymatrix{\bigcup\limits_{\tilde{\lambda}\in\mathfrak{P}^0(\Gamma),\gamma(\tilde{\lambda})=\lambda}\Ext_{C^0(\mathscr{P})}^1(\tilde{C}_q(\tilde{\mu}), \tilde{C}_q(\tilde{\nu}))_{\tilde{C}_q(\tilde{\lambda})}\ar[r]^-{\simeq}&\Ext_{C^1(\mathscr{P})}^1(C_q(\mu), C_q(\nu))_{C_q(\lambda)}.}$$

By \cite[Cor. 2.8 and 3.7]{ChenD},~$C^0(\mathscr{P})$ is directed and  there exists a polynomial
$\varepsilon_{\tilde{\mu} \tilde{\nu}}^{\tilde{\lambda}}(x)\in\mathbb{Z}[x]$ such that for each prime power $q$,
$$\varepsilon_{\tilde{\mu} \tilde{\nu}}^{\tilde{\lambda}}(q)=|\Ext_{C^0(\mathscr{P})}^1(\tilde{C}_q(\tilde{\mu}), \tilde{C}_q(\tilde{\nu}))_{\tilde{C}_q(\tilde{\lambda})}|.$$
Set $$\varepsilon_{\mu \nu}^{\lambda}(x)=\sum\limits_{\tilde{\lambda}\in\mathfrak{P}^0(\Gamma),\gamma(\tilde{\lambda})=\lambda}
\varepsilon_{\tilde{\mu} \tilde{\nu}}^{\tilde{\lambda}}(x).$$
Therefore, $$\varepsilon_{\mu \nu}^{\lambda}(q)=|\Ext_{C^1(\mathscr{P})}^1(C_q(\mu), C_q(\nu))_{C_q(\lambda)}|.$$
Thus the claim is proved.

Now for the general case, the proof is similar to that of \cite[Th. 3.11]{ChenD}.
For each $\mu\in \mathfrak{P}^1(\Gamma)$, set
$d_{\mu}=\sum_{x\in \mathcal {I}^1(\Gamma)}\mu(x)$ and $l_{\mu}=\dim \Ext^{1}_{C^{1}(\mathscr{P})}(C_{q}(\mu),C_{q}(\mu))$.
We only need to proceed by induction on $(l_{\mu},d_{\mu})$ with  the lexicographical order which starts from $(l_{\mu},d_{\mu})=(1,1)$.
So we are done.
\ep

\subsection{}
Our next aim in this section is to express the Hall numbers for objects in $C^{1}(\mathscr{P})$ by the Hall numbers for $A$-modules.

Recall that, for any $M, N\in\mod A$, $\Ext_{C^1(\mathscr{P})}^1(C_M,C_N)\cong \Hom_A(M,N)\oplus\Ext_A^1(M,N).$
Inspired by the ideas in \cite[Sect. 4]{ChenD}, we consider the Hall numbers $F_{C_MC_N}^{L^\cdot}$ associated to $M, N\in\mod A$ such that exactly one of $\Ext_A^1(M,N)$ and $\Hom_A(M,N)$ vanishes. So the two cases are treated as follows.

\textbf{Case I}\ Let $M, N\in\mod A$ such that $\Ext_{C^1(\mathscr{P})}^1(C_M,C_N)\cong\Ext_A^1(M,N)$, that is, each extension of $C_M$ by $C_N$ is induced by an extension of $M$ by $N$. More precisely, we consider the  Horseshoe Lemma diagram
\begin{equation}\label{Horse}\xymatrix{
&0\ar[d] & 0\ar[d] & 0\ar[d]\\
0\ar[r] & \ooz_N\ar[d]^-{\dz_N}\ar[r]^{\iota\quad} & \ooz_N\oplus
\ooz_M \ar[d]_{\widetilde\dz}\ar[r]^{\quad p} & \ooz_M\ar[d]^-{\dz_M}\ar[r] & 0\\
0\ar[r] & P_N\ar[d]^-{\vez_N}\ar[r]^{\iota\quad} & P_N\oplus P_M\ar[d]_{\widetilde\vez}\ar[r]^{\quad p} & P_M\ar[d]^-{\vez_M}\ar[r] & 0\\
0\ar[r] & N \ar[d]\ar[r]^{f} & L\ar[d]\ar[r]^{g} & M\ar[d]\ar[r] & 0.\\
 & 0 & 0&0\\
 }\end{equation}
By \cite[Lem. 4.1]{Br}, there exists a projective $A$-module $X_L$ such that $P_M\oplus P_N\cong P_L\oplus X_L$ and $\Omega_M\oplus \Omega_N\cong \Omega_L\oplus X_L$. Clearly, for fixed $M, N$, $X_L$ is uniquely determined by $L$ up to isomorphism. Thus $C_{\tilde{\delta}}\cong C_L\oplus K_{X_L}$.

\begin{proposition}
Let $L, M, N\in\mod A$ such that  $\Hom_A(M,N)=0$. Then $$F^{C_L\oplus K_{X_L}}_{C_MC_N}=q^{s} a_{X_L}F^L_{MN},$$ where $s=\lr{\Dim P_N,\Dim \Omega_M}+\lr{\Dim \Omega_N,\Dim P_M}-\lr{\Dim X_L,\Dim X_L}.$
\end{proposition}
\begin{proof}
We assume that $L$ is an extension of $M$ by $N$.
Since $C_L$ and $K_{X_L}$ have no direct summands in common, $$|\Aut(C_L\oplus K_{X_L})|
=a_{C_L}\cdot a_{K_{X_L}}\cdot|\Hom_{C^1(\mathscr{P})}(C_L,K_{X_L})|\cdot|\Hom_{C^1(\mathscr{P})}(K_{X_L},C_L)|.$$
By arguments as above, $|\Ext^1_{C^1(\mathscr{P})}(C_M,C_N)_{C_L\oplus K_{X_L}}|=|\Ext^1_A(M,N)_L|$.
By Riedtmann--Peng formula,
$$F^{C_L\oplus K_{X_L}}_{C_MC_N}=\frac{|\Ext^1_{C^1(\mathscr{P})}(C_M,C_N)_{C_L\oplus K_{X_L}}|\cdot|\Aut(C_L\oplus K_{X_L})|}
{|\Hom_{C^1(\mathscr{P})}(C_M,C_N)|\cdot |\Aut(C_M)|\cdot |\Aut(C_N)|}.$$
Using Proposition \ref{formula}, we deduce that $F^{C_L\oplus K_{X_L}}_{C_MC_N}=q^s a_{X_L}F^L_{MN},$
where \begin{equation*}\begin{aligned}
s=&(\Dim P_L,\Dim\Omega_L)+(\Dim X_L,\Dim\Omega_L)+(\Dim P_L,\Dim X_L)+\lr{\Dim X_L,\Dim X_L}\\
&-(\Dim P_M,\Dim \Omega_M)-(\Dim P_N,\Dim \Omega_N)-\lr{\Dim P_M,\Dim \Omega_N}-\lr{\Dim\Omega_M,\Dim P_N}.\\
\end{aligned}
\end{equation*}
Since $P_M\oplus P_N\cong P_L\oplus X_L$ and $\Omega_M\oplus \Omega_N\cong \Omega_L\oplus X_L$, we have $$s=\lr{\Dim P_N,\Dim \Omega_M}+\lr{\Dim \Omega_N,\Dim P_M}-\lr{\Dim X_L,\Dim X_L}.$$
\end{proof}

\textbf{Case II}\ Let $M, N\in\mod A$ such that $\Ext_{C^1(\mathscr{P})}^1(C_M,C_N)\cong\Hom_A(M,N)$, that is, each extension of $C_M$ by $C_N$ is induced by a homomorphism from $M$ to $N$. Take an arbitrary extension of $C_M$ by $C_N$
\begin{equation}\label{ss0}0\longrightarrow C_N\longrightarrow L^\cdot\longrightarrow C_M\longrightarrow 0.\end{equation}
It induces a long exact sequence:
$$\xymatrix{\ar[r]&H_0(C_N)\ar[r]&H_0(L^\cdot)\ar[r]&H_0(C_M)\ar[r]&H_0(C_N)\ar[r]&H_0(L^\cdot)\ar[r]&H_0(C_M)\ar[r]&}.$$
Obviously $H_0(C_X)=X$ for any $X\in\mod A$.
Write $L^\cdot=C_L\oplus K_{Y_L}$ for some $L\in\mod A$ and $Y_L\in\mathscr{P}$, then  we obtain a long exact sequence of $A$-modules
$$\xymatrix{\ar[r]&N\ar[r]&L\ar[r]&M\ar[r]^\delta&N\ar[r]&L\ar[r]&M\ar[r]&}$$ where $\delta$ is uniquely determined by the isomorphism \begin{equation}\label{iso}\Ext_{C^1(\mathscr{P})}^1(C_M,C_N)\cong\Hom_A(M,N).\end{equation}
So we have the short exact sequence $0\longrightarrow C\longrightarrow L\longrightarrow K\longrightarrow0,$ where $C=\Coker \delta,~K=\Ker \delta$.
It follows from the lemma below that $L\cong K\oplus C$. We also note that $Y_L$ is uniquely determined by \begin{equation}\label{determine}P_L\oplus \Omega_L\oplus {Y_L}\oplus {Y_L}\cong P_M\oplus P_N\oplus\Omega_M\oplus\Omega_N.\end{equation}
\begin{lemma}
Let $M, N\in\mod A$ such that $\Ext^1_A(M,N)=0$.
Then for any $\delta\in\Hom_A(M,N)$, $\Ext_A^1(K,C)=0$, where $C=\Coker\delta, K=\Ker\delta$.
\end{lemma}
\begin{proof}
From the long exact sequence $0\rightarrow K\rightarrow M\xrightarrow{\delta}N\rightarrow C\rightarrow 0$, we get two short exact sequences
$0\rightarrow K\rightarrow M\rightarrow G\rightarrow 0$ and $0\rightarrow G\rightarrow N\rightarrow C\rightarrow 0$, where $G= \im \delta$.
Applying the Hom functor $\Hom_{A}(-,N)$ to the first one, we have $\Ext_{A}^{1}(K,N)=0$. Then applying $\Hom_{A}(K,-)$ to the second one, we get
$\Ext_{A}^{1}(K,C)=0$.
\end{proof}

Therefore, for each $\delta\in\Hom_A(M,N)$, we have an extension $L^\cdot=C_L\oplus K_{Y_L}$ of $C_M$ by $C_N$,  where $L=\Ker\delta\oplus\Coker\delta$, and $Y_L$ is uniquely determined by $L$ as in (\ref{determine}).

Let $X, Y, M, N\in\mod A$, we denote by ${}_X\Hom_A(M,N)_Y$ the set $\{f: M\rightarrow N~|~\Ker f\cong X, \Coker f\cong Y \},$ and by $W_{MN}^{XY}$ the set $$\{(f,g,h)~|~\xymatrix{0\ar[r]&X\ar[r]^g&M\ar[r]^f&N\ar[r]^h&Y\ar[r]&0} \text{is exact in } \mod A\}.$$ By \cite[Sect. 8]{Van}, we know that $$|W_{MN}^{XY}|=a_Xa_Y\sum_{[L']}a_{L'}F_{L'X}^MF_{YL'}^N.$$
Hence, it is easy to know that \begin{equation}\label{lHoml}|{}_X\Hom_A(M,N)_Y|=\sum_{[L']}a_{L'}F_{L'X}^MF_{YL'}^N.\end{equation}

For any $M, N\in\mod A$, $|\Ext_{C^1(\mathscr{P})}^1(C_M,C_N)_{C_{M\oplus N}}|=1=|{}_{M}\Hom_A(M,N)_{N}|$. So we focus on  the nontrivial extension of $C_M$ by $C_N$, which is induced by a nonzero $\delta\in\Hom_A(M,N)$.

\begin{lemma}\label{hom-ext}
Let $M, N$ be indecomposable $A$-modules such that $\Ext_A^1(M,N)=0$, and suppose that $L\cong L_1\oplus L_2$, where $L_1$ and $L_2$ are the kernel and cokernel of a nonzero morphism $\delta\in\Hom_A(M,N)$, respectively. Then $$|\Ext_{C^1(\mathscr{P})}^1(C_M,C_N)_{C_L\oplus K_{Y_L}}|=|{}_{L_1}\Hom_A(M,N)_{L_2}|.$$
\end{lemma}
\begin{proof}
If $M\cong N$, then each nonzero $\delta\in\Hom_A(M,N)$ is an isomorphism.
Thus $L=L_1=L_2=0$ and $Y_L\cong P_M\oplus\Omega_M$. It follows that
$$|\Ext_{C^1(\mathscr{P})}^1(C_M,C_N)_{K_{P_M\oplus\Omega_M}}|=|\Hom_A(M,N)|-1=|{}_{0}\Hom_A(M,N)_{0}|.$$

If $M\ncong N$, then
for each nontrivial extension
$$\xi: 0\longrightarrow C_N\longrightarrow C_L\oplus K_{Y_L}\longrightarrow C_M\longrightarrow 0,$$ by the isomorphism (\ref{iso}), there exists a unique nonzero $\delta'\in\Hom_A(M,N)$ such that $L\cong\Ker\delta'\oplus\Coker\delta'$. Since $A$ is representation-directed, we obtain that $\Ker\delta'\cong L_1$ and $\Coker\delta'\cong L_2$. Thus $\delta'\in{}_{L_1}\Hom_A(M,N)_{L_2}$.
Now we have an injective map $\varphi:\Ext_{C^1(\mathscr{P})}^1(C_M,C_N)_{C_L\oplus K_{Y_L}}\longrightarrow{}_{L_1}\Hom_A(M,N)_{L_2}$  induced by the isomorphism (\ref{iso}). Obviously, $\varphi$ is surjective.  So we complete the proof.
\end{proof}

 Using Lemma \ref{hom-ext} and the formula (\ref{lHoml}), we have the following proposition.
\begin{proposition}
Let $M, N$  be indecomposable $A$-modules such that $\Ext_A^1(M,N)=0$, and suppose that $L\cong L_1\oplus L_2$, where $L_1$ and $L_2$ are the kernel and cokernel of a morphism $\delta\in\Hom_A(M,N)$, respectively. Then
$$F_{C_MC_N}^{C_L\oplus K_{Y_L}}=\begin{cases}
q^{s_0}\frac{a_{M\oplus N}}{a_Ma_N},\quad\quad&\text{if}~L\cong M\oplus N;\\

q^{s_1}\frac{a_La_{Y_L}}{a_Ma_N}\sum\limits_{[L']}a_{L'}F_{L'L_1}^MF_{L_2L'}^N,\quad\quad&\text{otherwise},
\end{cases}$$
where $$s_0=\lr{\Dim P_N,\Dim \Omega_M}+\lr{\Dim \Omega_N,\Dim P_M}-\lr{\Dim M,\Dim N},$$ and
\begin{equation*}\begin{split}s_1=&(\Dim P_L,\Dim \Omega_L)+(\Dim Y_L,\Dim \Omega_L)+(\Dim P_L,\Dim Y_L)+\lr{\Dim Y_L,\Dim Y_L}\\&-(\Dim P_M,\Dim \Omega_M)-(\Dim P_N,\Dim \Omega_N)-\lr{\Dim P_M,\Dim P_N}-\lr{\Dim \Omega_M,\Dim \Omega_N}.\end{split}\end{equation*}
\end{proposition}

The rest is devoted to a \textbf{special case}: for indecomposable $L, M, N\in\mod A$, we calculate $F_{C_MC_N}^{C_L}$.

Note that the calculations of Hall numbers in $C^1(\mathscr{P})$ always involve the information of the $\top$ of a module.
For any epimorphism $M\twoheadrightarrow N$ in $\mod A$, it is clear that $\top N\subseteq \top M$. However, it is not easy to compare $\top L$ and $\top M$ for any monomorphism $L\hookrightarrow M$ in $\mod A$.
\begin{lemma}\label{top lem}
Let $0\rightarrow L\rightarrow M\rightarrow N\rightarrow 0$ be a short exact sequence with $L, M, N$ all indecomposable nonzero $A$-modules,
then $\top M\nsubseteq\top L$.
\end{lemma}
\bp
Write $0\rightarrow L\xrightarrow{f} M\rightarrow N\rightarrow 0$, where $f$ induces naturally $\overline{f}=\top f: \top L\rightarrow \top M$.
Assume that $\top M\subseteq \top L$. Applying the Snake Lemma to the following short exact sequences:
$$\xymatrix{0\ar[r]&\rad L\ar[r]\ar[d]&L\ar[r]\ar[d]^{f}& \top L\ar[d]^{\overline{f}}\ar[r] & 0\\
0\ar[r]&\rad M\ar[r]&M\ar[r]& \top M\ar[r]& 0,}$$
we get a long exact sequence \begin{equation}\label{top seq}0\longrightarrow \Ker \overline{f}\longrightarrow \rad M / \rad L \longrightarrow N\longrightarrow \Coker \overline{f}\longrightarrow 0.\end{equation}
Since $\top L$ and $\top M$ are both semisimple, we have $\top M\oplus \Ker \overline{f}=\top L\oplus \Coker \overline{f}$ with $\Ker \overline{f}$ and $\Coker \overline{f}$ also semisimple. Then $\top M\subseteq \top L$ implies $\Coker \overline{f}\subseteq \Ker \overline{f}$.
Assume $\Coker \overline{f}\neq 0$, say there is a simple module $S_{i}\subseteq \Coker \overline{f}$. By (\ref{top seq}), $\Hom_{A}(N,S_{i})\neq 0$, which implies $\Ext^{1}_{A}(N,S_{i})=0$ since $A$ is representation-directed. Then applying the Hom functor $\Hom_{A}(-,S_{i})$ to the short exact sequence $0\rightarrow L\rightarrow M\rightarrow N\rightarrow 0$, we acquire
$$0\longrightarrow \Hom_{A}(N,S_{i})\longrightarrow \Hom_{A}(M,S_{i})\longrightarrow \Hom_{A}(L,S_{i})\longrightarrow 0.$$
But this means the multiplicity of direct summand $S_{i}$ in $\top M$ is greater than that in $\top L$, which contradicts our assumption.
Hence $\Coker \overline{f}= 0$, and thus $\Coker f=0$, which is absurd. So we are done.
\ep

\begin{proposition}\label{jisuan}
Let $P, L, M, N\in\mod A$ be indecomposable and $P$ projective.

$(1)$ If $\Hom_A(M,N)=0$, then there exists a short exact sequence
$0\rightarrow C_N\rightarrow C_L\rightarrow C_M\rightarrow 0$ in $C^1(\mathscr{P})$ if and only if $P_L\cong P_M\oplus P_N$, and there exists a short exact sequence $0\rightarrow N\rightarrow L\rightarrow M\rightarrow 0$ in $\mod A$.
Moreover, $$F^{C_L}_{C_MC_N}=q^{\text{\rm dim\,Hom}_A(P_N,\ooz_M)+\text{\rm dim\,Hom}_A(\ooz_N,P_M)} F^L_{MN}.$$

$(2)$ If~$M\ncong N$~and~$\Ext^1_A(M,N)=0$, then there exists a short exact sequence
$0\rightarrow C_N\rightarrow C_L\rightarrow C_M\rightarrow 0$ in $C^1(\mathscr{P})$ if and only if $P_L\oplus\Omega_L\cong P_M\oplus P_N\oplus\Omega_M\oplus\Omega_N$, and there exists a monomorphism $f: M\rightarrow N$ such that $\Coker f\cong L$. Moreover, $F^{C_L}_{C_MC_N}=
q^{t}\frac{a_L}{a_N}F_{LM}^N$, where \begin{equation*}\begin{split}t=&(\Dim P_L,\Dim \Omega_L)-(\Dim P_M,\Dim \Omega_M)-(\Dim P_N,\Dim \Omega_N)\\&-\lr{\Dim P_M,\Dim P_N}-\lr{\Dim \Omega_M,\Dim \Omega_N}.\end{split}\end{equation*}

$(3)$ There exists a short exact sequence
$0\rightarrow C_M\rightarrow K_P\rightarrow C_M\rightarrow 0$ in $C^1(\mathscr{P})$ if and only if $M\cong P$. Moreover, $F_{C_PC_P}^{K_P}=1$.

\end{proposition}
\begin{proof}
We only prove  $(2)$.

$(\Rightarrow)$ Assume that there exists a short exact sequence $0\rightarrow C_N\rightarrow C_L\rightarrow C_M\rightarrow 0$ in $C^1(\mathscr{P})$, then $P_L\oplus\Omega_L\cong P_M\oplus P_N\oplus\Omega_M\oplus\Omega_N$, and we have a long  exact sequence:
$$\xymatrix{N\ar[r]&L\ar[r]&M\ar[r]^f&N\ar[r]&L\ar[r]&M.}$$
Thus $L\cong K\oplus C$, where $K=\Ker f$ and $C=\Coker f$.
Since $L$ is indecomposable, either $K$ or $C$ vanishes.
If $K$ is zero, then $f$ is a monomorphism with $\Coker f\cong L$. Otherwise, $f$ is an epimorphism with $\Ker f\cong L$. From the short exact sequence $0\rightarrow L\rightarrow M\rightarrow N\rightarrow 0$, we conclude that $P_{L}\oplus P_{N}=P_{M}\oplus X$ and $\Omega_{L}\oplus \Omega_{N}=\Omega_{M}\oplus X$ for some projective module $X$.
Hence $X=P_{N}\oplus \Omega_{N}$ and $P_{L}=P_{M}\oplus \Omega_{N}$. But this is impossible according to Lemma \ref{top lem}.

$(\Leftarrow)$ For any $f\in\Hom_A(M,N)$, by the isomorphism (\ref{iso}), there exists an extension of $C_M$
by $C_N$
\begin{equation*}0\longrightarrow C_N\longrightarrow L^\cdot_f\longrightarrow C_M\longrightarrow 0,\end{equation*} where $L^\cdot_f\cong C_L\oplus K_W$ with $L=\Coker f$ and $P_L\oplus \Omega_L\oplus W\oplus W\cong P_M\oplus P_N\oplus\Omega_M\oplus\Omega_N$.
Clearly we have $L^\cdot_f\cong C_L$, since $P_L\oplus\Omega_L\cong P_M\oplus P_N\oplus\Omega_M\oplus\Omega_N$.
\end{proof}
\begin{remark}
Let $L^\cdot, M^\cdot, N^\cdot\in C^1(\mathscr{P})$ with $M^\cdot, N^\cdot$ indecomposable. If $M^\cdot$ or $N^\cdot$ is acyclic, then $F_{M^\cdot N^\cdot}^{L^\cdot}$ is zero unless $L^\cdot\cong M^\cdot\oplus N^\cdot$.  Ringel \cite{R91a} gave a complete list of all Hall polynomials $\psi_{MN}^L$ for indecomposable $L, M, N\in\mod A$, so Proposition $\ref{jisuan}$ roughly provides all Hall polynomials $\psi_{M^\cdot N^\cdot}^{L^\cdot}$ for indecomposable $L^\cdot, M^\cdot, N^\cdot\in C^1(\mathscr{P})$.
\end{remark}

The following corollary is useful in Section 5.

\begin{corollary}\label{proj first}
Let $0\rightarrow P\rightarrow M\rightarrow N \rightarrow 0$ be a short  exact sequence in $\mod A$ with $P$ projective. If $\top M\cong\top N$, then we have a short exact sequence $0\rightarrow C_M\rightarrow C_N\rightarrow C_P \rightarrow 0$ in $C^1(\mathscr{P})$.
\end{corollary}
\begin{proof}
Suppose we are given a short exact sequence \begin{equation}\label{dzhl}\xymatrix{0\ar[r]& P\ar[r]^f& M\ar[r]& N\ar[r]& 0.}\end{equation}
Since $\Ext_{C^1(\mathscr{P})}^1(C_P,C_M)\cong\Hom_A(P,M)$, $f$ induces an extension of $C_P$ by $C_M$
\begin{equation*}0\longrightarrow C_M\longrightarrow L^\cdot_f\longrightarrow C_P\longrightarrow 0,\end{equation*} where $L^\cdot_f\cong C_N\oplus K_W$ with $P_N\oplus \Omega_N\oplus W\oplus W\cong P\oplus P_M\oplus\Omega_M$. By the Horseshoe Lemma diagram of the sequence (\ref{dzhl}) as in (\ref{Horse}) , there exists a projective $R\in\mod A$ such that $P_M\oplus R\cong P\oplus P_N$ and $\Omega_M\oplus R\cong\Omega_N$. Since $\top M\cong\top N$, we know that $P_M\cong P_N$. Hence $R\cong P$ and  $P\oplus \Omega_M\cong\Omega_N$. So  $L^\cdot_f\cong C_N$.

\end{proof}

\section{Lie algebras associated to $1$-cyclic perfect complexes}\label{sec lie}
In this section, $Q$ is still a Dynkin quiver as above. We use the Hall polynomials introduced in Theorem \ref{Hallpoly} to define  Lie algebras associated to $1$-cyclic perfect complexes.

We denote by $\tilde{\mathfrak{n}}$ the vector space over $\mathbb{C}$ spanned by the iso-classes of indecomposable objects in $C^1(\mathscr{P})$, and by $\tilde{\mathfrak{n}}^{+}$ (resp. $\hat{\mathfrak{n}}$) the subspace of $\tilde{\mathfrak{n}}$ spanned by the iso-classes of indecomposable and non-acycilc (resp. acyclic) objects in $C^1(\mathscr{P})$. Obviously, $\tilde{\mathfrak{n}}=\tilde{\mathfrak{n}}^{+}\oplus \hat{\mathfrak{n}}$ as vector spaces.

\begin{proposition}
$\tilde{\mathfrak{n}}$ is a Lie algebra with Lie bracket given by
$$[[C_q(\mu)], [C_q(\nu)]]=\sum\limits_{\lambda\in\mathfrak{P}^1(\Gamma)}(\psi_{\mu \nu}^{\lambda}(1)-\psi_{\nu \mu}^{\lambda}(1))[C_q(\lambda)]$$ for any indecomposable $\mu, \nu\in\mathfrak{P}^1(\Gamma)$. Moreover, $\tilde{\mathfrak{n}}^{+}$ and $\hat{\mathfrak{n}}$ are both Lie subalgebras (ideals) of~$\tilde{\mathfrak{n}}$.
\end{proposition}
\bp

By the associativity of the Ringel--Hall algebra $\mathcal {H}(C^1(\mathscr{P}))$, the Jacobi identity holds naturally. Then according to Proposition \ref{F-F}, the first statement is proved. We only prove that $\tilde{\mathfrak{n}}^{+}$ is an ideal of~$\tilde{\mathfrak{n}}$.

Let $\mu, \nu\in\mathfrak{P}^1(\Gamma)$ such that $\mu$ and $\nu$ are both indecomposable, and $\mu$ is non-acyclic. Suppose that there exists an indecomposable and acyclic $\lambda\in\mathfrak{P}^1(\Gamma)$ such that we have a short exact sequence $0\longrightarrow C_q(\nu)\longrightarrow C_q(\lambda) \longrightarrow C_q(\mu)\longrightarrow 0$ (or $0\longrightarrow C_q(\mu)\longrightarrow C_q(\lambda) \longrightarrow C_q(\nu)\longrightarrow 0$). It induces a long exact sequence
$H_0(C_q(\nu))\rightarrow 0 \rightarrow H_0(C_q(\mu)) \rightarrow H_0(C_q(\nu)) \rightarrow 0 \rightarrow H_0(C_q(\mu))$
(or $H_0(C_q(\mu))\rightarrow 0 \rightarrow H_0(C_q(\nu)) \rightarrow H_0(C_q(\mu)) \rightarrow 0 \rightarrow H_0(C_q(\nu))$).
So we have $M_q(\mu)\cong M_q(\nu)$. Now that $\mu=\nu$, $[[C_q(\mu)], [C_q(\nu)]] \in \tilde{\mathfrak{n}}^{+}$. Thus $\tilde{\mathfrak{n}}^{+}$ is an ideal of~$\tilde{\mathfrak{n}}$.
\ep

We remark that $\hat{\mathfrak{n}}$ is an abelian Lie algebra. It is also straightforward $\tilde{\mathfrak{n}}\cong \tilde{\mathfrak{n}}^{+}\oplus \hat{\mathfrak{n}}$ as Lie algebras.

Note that the basis of $\tilde{\mathfrak{n}}$, i.e. $[C_q(\mu)]$ with $\mu$ indecomposable, can be rewritten as (cf. (\ref{allindec}))
$$\{[C_{M_q(\alpha)}], [K_{M_q(\beta_i)}]~|~\alpha\in\Phi^{+}, 1\leq i\leq n\}.$$

From now on, for convenience, we will use the symbol $C_{M}$ instead of $[C_{M_q(\alpha)}]$ in $\tilde{\mathfrak{n}}^{+}$ for an indecomposable module $M\cong M_q(\alpha)$.  Thus we use $C_{P_i}$ (resp. $C_{S_i}$) instead of $[C_{M_q(\beta_{i})}]$ (resp. $[C_{M_{q}(\alpha_{i})}]$).
Similarly, $[K_{M_q(\beta_i)}]$ is denoted by $K_{P_{i}}$.

\begin{definition}
Given a Dynkin quiver $Q$ with $n$ vertices, one can attach an $n\times n$ matrix $E=(a_{ij})$ as follows: If there is a path
between $i$ and $j$ in $Q$, say from $i$ to $j$, then $a_{ij}=1$ and $a_{ji}=-1$. Otherwise, $a_{ij}=a_{ji}=0$.
$E$ is  skew symmetric and called the \emph{path matrix} of $Q$.
\end{definition}
The following are immediate:

(1) If $a_{ij}=\pm 1$ and $a_{jk}=\pm 1$, then $a_{ik}=\pm 1$;

(2) $Q$ is uniquely determined by $E$ up to a renumbering of vertices, so we write $E=E_{Q}$.

(3) If $Q$ is bipartite, i.e. each vertex is a sink or a source, then $E$ coincides with the exchange matrix associated to $Q$ (cf. \cite{K}).

The rest of this paper is devoted to the structure of $\tilde{\mathfrak{n}}^{+}$. Since $\tilde{\mathfrak{n}}^{+}$
depends heavily on both the underlying graph and the orientation of the quiver $Q$, we will call $\tilde{\mathfrak{n}}^{+}$
the Lie algebra of the quiver $Q$ and write $\tilde{\mathfrak{n}}^{+}(Q)$ if necessary.
\begin{theorem}\label{main result}
Let $Q$ be a Dynkin quiver and  $E=(a_{ij})$ its path matrix.
The associated Lie algebra $\tilde{\mathfrak{n}}^{+}$ is generated by $\{C_{P_i}~|~1\leq i\leq n\}$, and these generators satisfy the following relations:

(a) If $|a_{ij}|=1$, $(\ad C_{P_i})^2(C_{P_j})=(\ad C_{P_j})^2(C_{P_i})=0$;

(b) If $a_{ij}a_{jk}=1$, $[C_{P_{i}},[C_{P_{j}},C_{P_{k}}]]=[C_{P_{k}},[C_{P_{i}},C_{P_{j}}]]=0$;

(c) If $a_{ij}=0$, $[C_{P_i},C_{P_j}]=0$.
\end{theorem}
\bp
We first give the following claim, which will be proved in the next section.
\begin{claim}\label{generators}
$\tilde{\mathfrak{n}}^{+}$ is generated by $\{C_{P_i}~|~1\leq i\leq n\}$.
\end{claim}
(a) Assume $a_{ij}=1$, then there is a path from $i$ to $j$ in $Q$ and $\Hom_A(P_j, P_i)\cong k$.
Note any nonzero morphism from $P_j$ to $P_i$ is a monomorphism. Choose $0\neq f \in \Hom_A(P_j, P_i)$ and set $M=\Coker f$, then we have a short exact sequence in $\mod A$
\begin{equation}\label{P_j,P_i}\xymatrix{0\ar[r]&{P_j}\ar[r]^f&{P_i}\ar[r]&M\ar[r]&0.}\end{equation}
Clearly, $M$ is indecomposable in $\mod A$.

$\Ext^1(C_{P_i},C_{P_j})=0$~and~$\Ext^1(C_{P_j},C_{P_i})\cong \Hom_A(P_j,P_i)$ is one-dimensional, so $[C_{P_j},C_{P_i}]=C_M$. Thus, $(\ad C_{P_j})^2(C_{P_i})=[C_{P_j},[C_{P_j},C_{P_i}]]=[C_{P_j},C_M]$.

Applying the functor $\Hom_A(P_j,-)$ to the sequence $(\ref{P_j,P_i})$, we get $\Hom_A(P_j,M)=0$.
So $\Ext^1(C_{P_j},C_M)\cong\Hom_A(P_j,M)=0$.

Applying the functor $\Hom_A(-,P_j)$ to the sequence $(\ref{P_j,P_i})$, we get that $\Ext^1(C_M,C_{P_j})\cong\Ext_A^1(M,P_j)$ is one-dimensional. Note that the middle term of the nontrivial extension class in $\Ext^1(C_M,C_{P_j})$ is isomorphic to $C_{P_i}\oplus K_{P_j}$, which is decomposable. So $[C_{P_j},C_M]=0$, hence $(\ad[C_{P_j}])^2([C_{P_i}])=0$.

Similarly, $(\ad C_{P_i})^2(C_{P_j})=[C_{P_i},[C_{P_i},C_{P_j}]]=[C_M,C_{P_i}]=0$.

(b) Assume $a_{ij}=a_{jk}=1$, then there exist a path from $i$ to $j$ and a path from $j$ to $k$.
Hence there is a path from $i$ to $k$. Choose $0\neq g\in \Hom_A(P_k, P_j)\cong k$ and let $N= \Coker g$, then there is a short exact sequence
\begin{equation}\label{P_k,P_j}\xymatrix{0\ar[r]&{P_k}\ar[r]^g&{P_j}\ar[r]&N\ar[r]&0.}\end{equation}
Obviously, $[C_{P_k},C_{P_j}]=C_N$. Thus $[C_{P_{i}},[C_{P_{j}},C_{P_{k}}]]=[C_{N}, C_{P_{i}}]$.

By applying $\Hom_A(-,P_i)$ to the sequence $(\ref{P_k,P_j})$,
we deduce that $\Hom_A(N,P_i)=0$ and $\Ext_A^1(N,P_i)=0$. Hence $\Ext^{1}(C_{N}, C_{P_{i}})= 0$.
Moreover, $\Ext^{1}(C_{P_{i}},C_{N})\cong \Hom_A(P_{i},N)=0$. Thus the middle terms concerning the computations of $[C_{N}, C_{P_{i}}]$
are always decomposable. So $[C_{P_i},[C_{P_j},C_{P_k}]]=0$.

Similarly, choose $0\neq f\in \Hom_A(P_j, P_i)$ as in the proof of (a).
Thus $[C_{P_{k}},[C_{P_{i}},C_{P_{j}}]]=[C_M, C_{P_{k}}]$. Again, applying the functors $\Hom_A(P_k,-)$ and $\Hom_A(-,P_k)$ to
the sequence $(\ref{P_j,P_i})$, we conclude that $\Ext^{1}(C_{P_{k}},C_{M})\cong \Hom_A(P_{k},M)=0$ and
$\Ext^{1}(C_{M}, C_{P_{k}})\cong\Hom_A(M,P_k)\oplus \Ext_A^1(M,P_{k})=0$. Therefore, $[C_{P_{k}},[C_{P_{i}},C_{P_{j}}]]=0$.

(c) Since there is no path between $i$ and $j$, $\Hom_A(P_i,P_j)=\Hom_A(P_j,P_i)=0$. Hence,
$\Ext^1(C_{P_i},C_{P_j})=\Ext^1(C_{P_j},C_{P_i})=0$. So $[C_{P_i},C_{P_j}]=0$.
\ep

\begin{remark}
1. We remark that $\tilde{\mathfrak{n}}^{+}$ cannot be generated by $C_{S_{i}}, i\in Q_{0}$, which is different from the case in Ringel's Lie algebra construction (cf. \cite{R91a}) from $\mod A$. This is easily seen in the examples below.\\
2. For any indecomposable $A$-module $M$, consider the minimal projective resolution of $M$
\begin{equation*}0\longrightarrow \oplus_j m_j P_j\longrightarrow \oplus_i m_i P_i \longrightarrow M\longrightarrow 0,\end{equation*} where $P_i$ and $P_j$ are non-isomorphic indecomposable projective $A$-modules. This corresponds to the following short exact sequence in  $C^1(\mathscr{P})$:
\begin{equation*}0\longrightarrow (\oplus_i m_i C_{P_i})\longrightarrow C_M\longrightarrow(\oplus_j m_j C_{P_j}) \longrightarrow 0\end{equation*}
under the isomorphism $\Ext_{C^1(\mathscr{P})}^1(\oplus_j m_j C_{P_j},\oplus_i m_i C_{P_i})\cong\Hom_A(\oplus_j{m_j P_j},\oplus_i {m_i P_i})$.

 Although it seems that, in $\tilde{\mathfrak{n}}^{+}$, $C_{M}$ could be generated by aforementioned $C_{P_{i}}$, with $C_{P_i}$ appearing $m_i$ times in the final expression, we cannot find a constructive way generally.
\end{remark}

Recall Example \ref{lizi}. Consider the Lie algebra $\tilde{\mathfrak{n}}^{+}$ associated to the quiver $Q$ of type $\mathbb{A}_3$, we have the following identities:
$$C_{S_2}=[C_{P_3},C_{P_2}],~C_{I_2}=[C_{P_3},C_{P_1}],~C_{S_1}=[C_{P_2},C_{P_1}].$$

\begin{corollary}\label{semisimple}
If $Q$ is bipartite, then $\tilde{\mathfrak{n}}^{+}$ is isomorphic to the nilpotent part $\mathfrak{n}^{+}$ of the simple Lie algebra $\mathfrak{g}$ associated to $Q$.
\end{corollary}
\bp
It follows from Theorem $\ref{main result}$ that the map $\psi: \mathfrak{n}^{+}\longrightarrow\tilde{\mathfrak{n}}^{+},~e_i\mapsto C_{P_i}$ is an epimorphism of Lie algebras. By a comparison of dimensions, we obtain that $\psi$ is an isomorphism.
\ep

\begin{example}
Let $Q$ be the quiver of type $\mathbb{D}_4$
$$
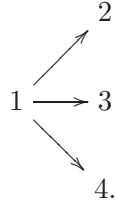
\xymatrix@!=0.8pc{&2\\1\ar[ru]\ar[r]\ar[rd]&3\\&4.}$$
The AR-quiver of $A=kQ$ is
$$
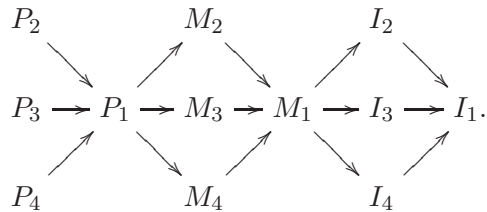
\xymatrix@!=0.8pc{P_2\ar[rd]&&M_2\ar[rd]&&I_2\ar[rd]&\\
P_3\ar[r]&P_1\ar[ru]\ar[r]\ar[rd]&M_3\ar[r]&M_1\ar[ru]\ar[r]\ar[rd]&I_3\ar[r]&I_1.\\
P_4\ar[ru]&&M_4\ar[ru]&&I_4\ar[ru]&}$$
Take the minimal projective resolution of $M_1$
\begin{equation}\label{mini N}0\longrightarrow P_2\oplus P_3\oplus P_4\longrightarrow P_1\oplus P_1\longrightarrow M_1\longrightarrow 0.\end{equation}
Since $Q$ is bipartite, we set $e_1=[C_{P_1}],~e_2=[C_{P_2}],~e_3=[C_{P_3}],~e_4=[C_{P_4}]$.
Then we can get the following pushout-pullback diagrams one by one:
$$\xymatrix@=0.8pc{&  & 0\ar[d] & 0\ar[d]\\
& & P_{1} \ar@{=}[r] \ar[d] & P_{1} \ar[d]\\
0\ar[r] & P_{2}\oplus P_{3}\oplus P_{4}\ar@{=}[d]\ar[r] & P_{1}\oplus P_{1} \ar[r]\ar[d] & M_1\ar[d] \ar[r]& 0\\
0\ar[r] & P_{2}\oplus P_{3}\oplus P_{4}\ar[r] & P_1 \ar[r]\ar[d]& I_1\ar[r]\ar[d] & 0\\
& & 0 & 0\\
}$$
$$\xymatrix@=0.8pc{& 0\ar[d] & 0\ar[d] & & & & & 0\ar[d] & 0\ar[d]\\
& P_{3}\oplus P_{4} \ar@{=}[r] \ar[d] & P_{3}\oplus P_{4}\ar[d] & & & & &  P_{4} \ar@{=}[r] \ar[d] &  P_{4}\ar[d]\\
0\ar[r] & P_{2}\oplus P_{3}\oplus P_{4}\ar[d]\ar[r] & P_1 \ar[r]\ar[d] & I_1\ar@{=}[d] \ar[r]& 0 & &0\ar[r] &  P_{3}\oplus P_{4}\ar[d]\ar[r] & P_1 \ar[r]\ar[d] & I_2\ar@{=}[d] \ar[r]& 0\\
0\ar[r] & P_{2}\ar[r]\ar[d] & I_2 \ar[r]\ar[d]& I_1\ar[r] & 0 & & 0\ar[r] & P_{3}\ar[r]\ar[d] & M_4 \ar[r]\ar[d]& I_2\ar[r] & 0\\
& 0 & 0 & & & & & 0 & 0\\
}$$
It easily follows that $C_{M_4}=[e_4, e_1]$, $C_{I_2}=[e_3, C_{M_4}]=[e_3, [e_4, e_1]]$, $C_{I_1}=[e_2, C_{I_2}]=[e_2, [e_3, [e_4, e_1]]]$.
Finally, we obtain $C_{M_1}=[C_{I_1},e_{1}]=[[e_2, [e_3, [e_4, e_1]]],e_1]$.

Now we identify $\mathfrak{n}^{+}$ with $\tilde{\mathfrak{n}}^{+}$. Thus $e_{i}=C_{P_i}$, $1\leq i\leq 4$ are the chevalley generators, corresponding
to the simple roots $\alpha_{i}$. We have the root space decomposition $\mathfrak{n}^{+}=\oplus_{\alpha\in \Phi^{+}} \mathfrak{n}_{\alpha}$ with $\dim \mathfrak{n}_{\alpha}=1$. It is obvious $C_{M_1}$ corresponds to the longest root $2\alpha_{1}+\alpha_{2}+\alpha_{3}+\alpha_{4}$.

It is well known (cf. \cite{Bo}) that for any $\alpha\in\Phi^{+}$, there exists $i\in I$ such that $\alpha-\alpha_{i}\in \Phi^{+}\cup \{0\}$. Thus our discussion above should be seen as an intuitive way to realize this fact.

We also remark that $C_{M_4}$ corresponds to the root $\alpha_{1}+\alpha_{4}$. This is different from the case of Ringel--Hall Lie algebra generated by indecomposable $kQ$-modules (cf. \cite{R91a,R91b}), where the index module $M_4$ obviously corresponds to the root $\alpha_{1}+\alpha_{2}+\alpha_{3}$.  All in all, the Lie algebra $\tilde{\mathfrak{n}}^{+}$ of a bipartite quiver $Q$ gives an essentially new categorification of the positive root system of a simple Lie algebra.
This will be discussed further at the end of the next section.
\end{example}

\section{Proof of Claim \ref{generators}}
Given an arbitrary element $C_M\in \tilde{\mathfrak{n}}^{+}$, where $M$ is an indecomposable non-projective $A$-module, we have two different methods to generate $C_M$:

(I) find a short exact sequence with all terms indecomposable modules
$$0\longrightarrow M_1\longrightarrow M \longrightarrow M_2\longrightarrow 0$$
such that $\Hom(M_1,M_2)=0$ and $\top M=\top M_1\oplus \top M_2$. Thus we have $C_M=[C_{M_2},C_{M_1}]$.

(II) find an indecomposable module $N$ and an indecomposable projective module $P$ such that
$$0\longrightarrow P\longrightarrow N \longrightarrow M\longrightarrow 0$$
is exact and $\top M=\top N$. Then we get $C_M=[C_P,C_N]$ by Corollary \ref{proj first}.

Our strategy is as follows: We use (I) to proceed by induction on the dimension vectors of $A$-modules which serve as indices. But it
is not enough because (I) doesn't hold generally and also our induction starts from $C_{P_i}$, not $C_{S_i}$.
This will be compensated by (II).

We first list two related results as preparations.

Let $\mathcal{A}$ be a finite-dimensional $k$-category which is a hereditary length category.
An object $M$ in $\mathcal{A}$ is called \emph{exceptional} if it is indecomposable and $\Ext^{1}(M,M)=0$.
Then denote by $s(M)$ the number of iso-classes of composition
factors of $M$.
A pair $(V,U)$ of exceptional objects is called an \emph{orthogonal exceptional pair} if
$$\Hom(U,V)=\Hom(V,U)=\Ext^{1}(U,V)=0.$$
Given an orthogonal exceptional pair $(V,U)$, let $\mathcal{C}(U,V)$ be the full subcategory of all objects of $\mathcal{A}$
which have a filtration with factors of the form $U$ and $V$.
\begin{lemma} {\rm(\cite[Th. 3.1]{R96e})}
Let $\mathcal{A}$ be a finite-dimensional $k$-category which is a hereditary length category. Let $M$ be an exceptional
object in $\mathcal{A}$. Then there are precisely $s(M)-1$ orthogonal exceptional pairs $(V_i,U_i)$ such that $M$ belongs
to $\mathcal{C}(U_i,V_i)$ and is not a simple object in $\mathcal{C}(U_i,V_i)$.
\end{lemma}
In our case, $\mathcal{A}=\mod kQ$ and $Q$ is a Dynkin quiver. Thus every indecomposable $kQ$-module is exceptional and each $\mathcal{C}(U_i,V_i)$,
as a subcategory of $\mod kQ$, is equivalent to the module category of the path algebra of a quiver of type $\mathbb{A}_2$. So for every indecomposable module $M$, there exist exactly $s(M)-1$ short exact sequences
$$0\longrightarrow U_i\longrightarrow M\longrightarrow V_i\longrightarrow 0 \eqno{(**)}$$ with
$U_i, V_i$ indecomposable and $\Hom(U_i, V_i)=0$, which might be suitable candidates for our method (I).

It is well known that Horseshoe Lemma plays an important role in homological algebra. But if we replace projective resolutions by minimal projective resolutions, the so-called minimal Horseshoe Lemma (see \cite{Wangli}) is not true in general.
\begin{lemma}
Let $0\rightarrow X\rightarrow Y\rightarrow Z\rightarrow 0$ be a short exact sequence in $\mod A$, the following are equivalent:

(1) Minimal Horseshoe Lemma holds;

(2) $\top Y=\top X\oplus \top Z$;

(3) $0\rightarrow \Hom(Z, S_{i})\rightarrow \Hom(Y, S_{i})\rightarrow \Hom(X, S_{i})\rightarrow 0$ is exact for each simple module $S_i$.
\end{lemma}
\bp
The equivalence of (1) and (2) is trivial, while the equivalence of (2) and (3) follows from
$\top X \cong D\Hom (X, D(A/r))$ under the identification $\overline{x}\mapsto (f\mapsto f(x)(1))$ for all $\overline{x}\in
\top X, x\in X$, where $r$ is the radical of $A$ and $D$ is the standard $k$-duality.
\ep

So our aim in method (I) is to pick out a short exact sequence $(**)$ satisfying Minimal Horseshoe Lemma for $M$.
For convenience, we will abbreviate this desired property as MHL.

In the study of the Gabriel--Roiter measure for representation-finite hereditary algebras (\cite{Chen}), the
author succeeded in constructing a short exact sequence of the form $(**)$ intuitively. Although it doesn't satisfy MHL generally,
the proof of the next theorem is inspired by the proof in \cite{Chen}.
\begin{theorem}\label{sincere}
For any sincere non-projective indecomposable module $M$, there exists at least one short exact sequence of the form $(**)$ satisfying MHL.
\end{theorem}
Note that the AR-quiver of $kQ$ is a translation quiver and the orbit quiver is a tree with at most 4 end points.
For each indecomposable module $M$, we denote by $[\![M]\!]$ the orbit of $M$, i.e., a vertex in the orbit quiver.  Now consider
$Q$ of type $\mathbb{D}_n$ or $\mathbb{E}_{6,7,8}$, whose orbit graph is a star with 3 branches. $M$ is said to \emph{lie on the center} if $[\![M]\!]$
has exactly 3 neighbors in the orbit quiver; and $M$ is said to \emph{lie on the quasi-center} if $[\![M]\!]$ has two neighbors and one of the neighbors lies on the center. Then we define $sl(M)$ to be the length from the center vertex to $[\![M]\!]$  in the branch containing $[\![M]\!]$.
Thus $sl(M)=0$ if $M$ lies on the center and $sl(M)=1$ if $M$ lies on the quasi-center.

For every indecomposable module $M$, we also define $\alpha(M)$ to be the number of indecomposable summands of $X$ with $X\rightarrow M$ a minimal right almost split map.

Recall from \cite{R84} that the AR-quiver of the linearly oriented quiver of type $\mathbb{A}_n$ is denoted by $\Theta(n)$. Let $w$ be a vertex of
a translation quiver $\Gamma$. A mesh-complete subquiver $\Theta$ of $\Gamma$ is called a \emph{wing} of $w$ if $\Theta$ is of the form $\Theta(n)$ for some $n\geq 2$
and $w$ is the sole projective-injective vertex of $\Theta$. The number $n$ is called the \emph{length of the wing} and we will say $w$ has a wing of length $n$.
By abuse of language, we will not distinguish an indecomposable module $M$ and the corresponding vertex in the AR-quiver of $kQ$.  Also recall that
a path $M_0\rightarrow M_1\rightarrow \cdots \rightarrow M_s$ in the AR-quiver is called a \emph{sectional path} if $\tau M_{i+1}\neq M_{i-1}$ for all $i=1,\cdots, s-1$.
Let  $\Sigma_{\rightarrow}(M)$ be the set of all indecomposable modules that can be reached from $M$ by a sectional path and $\Sigma_{\leftarrow}(M)$ is defined dually (see \cite{S}).
\begin{lemma}\label{wing}
Assume an indecomposable module $M$ has a wing of length $n$, if one of the first $n-1$ modules in the bottom ($M$ is on the top) is not simple, there exists a short exact sequence of the form $(**)$ satisfying MHL.
\end{lemma}
$$
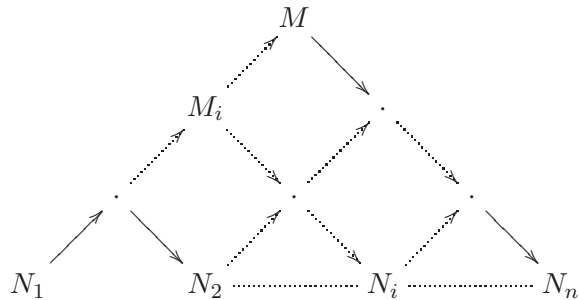
\xymatrix@!=0.8pc{ & & &M\ar[rd]\\
& & M_{i} \ar@{.>}[ru]\ar@{.>}[rd] & & \cdot \ar@{.>}[rd] \\
& \cdot\ar@{.>}[ru] \ar[rd] & & \cdot \ar@{.>}[ru] \ar@{.>}[rd] & & \cdot \ar[rd]\\
N_{1}\ar[ru] & & N_{2}\ar@{.>}[ru]\ar@{.}[rr] & & N_{i}\ar@{.>}[ru]\ar@{.}[rr] & & N_n
}
$$
\bp
Assume $N_i$ is not simple. It is easily seen that $\dim \Hom(M_{i},M)=1$ and
any nonzero map $f$ from $M_i$ to $M$ is a monomorphism. Thus $0 \rightarrow M_i \xrightarrow{f} M \rightarrow \Coker f \rightarrow 0$ is
a short exact sequence of the form $(**)$. Then check the exactness when applying the functor $\Hom(-,  S_i)$
for any simple module $S_i$. If $S_i$ is on the left of $\Sigma_{\leftarrow}(M_i)$, then $\Hom(-,  S_i)=0$ on
all terms of the exact sequence. If $S_i$ is on the right of $\Sigma_{\leftarrow}(\Coker f)$, we have $\Ext^{1}(\Coker f,S_i)=0$.
Thus the left area is in between, which is easily checked.
\ep
We remark that in the following context one only needs to consider the area between $\Sigma_{\rightarrow}( U_i)$ and $\Sigma_{\leftarrow}(V_i)$ when
checking the exactness of $\Hom(-,  S_i)$ on the short exact sequence $(**)$, because the other two areas are the same as in proving Lemma \ref{wing}.
\begin{proposition}
For any quiver of type $\mathbb{D}_n$ or $\mathbb{E}_{6,7,8}$, let $M$ be a sincere non-projective indecomposable module.
If $sl(M)=0$, then there exists a short exact sequence of the form $(**)$ satisfying MHL.
\end{proposition}
\bp
By the work of Happel \cite{Happel}, one can identify the AR-quiver $\Gamma_{D^{b}(kQ)}$ of $D^{b}(kQ)$ with the repetition quiver $\mathbb{Z}Q$.
So we consider the following subquiver of $\Gamma_{D^{b}(kQ)}$, in which $\dashrightarrow$ means a irreducible map that possibly exists, and $\xymatrix@1{\cdot \ar@{.>}[r]& \cdot}$ means the composition of one or more irreducible maps.
$$
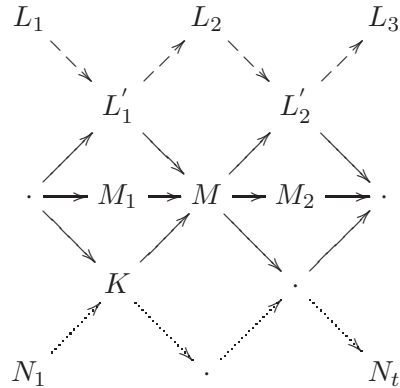
\xymatrix@!=0.8pc{L_1 \ar@{-->}[rd] & & L_2 \ar@{-->}[rd] & & L_3\\
& L_1^{'} \ar@{-->}[ru] \ar[rd] & & L_2^{'} \ar@{-->}[ru] \ar[rd]\\
\cdot \ar[ru] \ar[r] \ar[rd] & M_1 \ar[r] & M \ar[ru] \ar[r] \ar[rd] & M_2 \ar[r]& \cdot\\
& K \ar[ru] \ar@{.>}[rd] & & \cdot \ar[ru] \ar@{.>}[rd]\\
N_1 \ar@{.>}[ru] & & \cdot \ar@{.>}[ru] & & N_t
}
$$
Since $M$ is sincere, $L_i\neq 0$, $M_i\neq 0$ and $N_i\neq 0$. Here an object $N\neq 0$ means $N$ is a $kQ$-module.
Thus $M$ has three wings. If one of them satisfies the condition in Lemma \ref{wing}, then we are done. Otherwise,
$L_1, L_2$ (possibly $L_1^{'}$), $M_1$ and $N_{1},\cdots, N_{t-1}$ are all simple modules. Then we have:
\begin{eqnarray*} & & \underline{\dim} M-(\underline{\dim}L_{1}^{'}+\underline{\dim}M_1+\underline{\dim}K)\\
&=& \underline{\dim} M-(\underline{\dim}L_{1}+\underline{\dim}L_{2}+\underline{\dim}M_{1}+\underline{\dim}N_{1}+\cdots + \underline{\dim}N_{t-1})>0
\end{eqnarray*}
because $M$ is sincere. This is impossible since $M$ is non-projective.
\ep
{\bf{Proof of Theorem \ref{sincere}}}
We will proceed case by case, using $sl(M)$ and $\alpha(M)$ to locate the sincere non-projective module $M$ in the AR-quiver.
\\
\textit{$\mathbb{A}_n$ type}

In this case, $M$ is the unique sincere indecomposable module with $\underline{\dim} M=(1,\cdots,1)$. The type $\mathbb{A}_n$ graph
$\xymatrix@1{1\ar@{-}[r]& 2\ar@{-}[r]& \cdot \ar@{.}[r] &\cdot\ar@{-}[r] & n}$ needs to be given an orientation such that at least one of the vertices excluding both ends is a sink, say $i$, $1<i<n$. Otherwise, $M$ must be projective. With the orientation given, we set indecomposable
modules $M_1$ and $M_2$ such that $\underline{\dim} M_1=(\underbrace{0,\cdots,0}_{i-1},1,\cdots,1)$ and $\underline{\dim} M_2=\underline{\dim} M-\underline{\dim} M_1$. Then $0\rightarrow M_1\rightarrow M\rightarrow M_2\rightarrow 0$ is our desired short exact sequence.

We remark that in any Dynkin type, one can find a short exact sequence $(**)$ satisfying MHL for any sincere non-projective injective indecomposable module. This
is because it has dimension vector $(1,\cdots,1)$, thus the case is nearly the same as above. So from now on, we shall assume:
\begin{itemize}
\item $M$ does not lie on the center;
\item $M$ is a sincere indecomposable module;
\item $M$ is neither projective nor injective.
\end{itemize}
\textit{$\mathbb{D}_{n}$ type}

If $M$ lies on the boundary of the orbit quiver, $l(M)=1$ and $\alpha(M)=1$. Look at the following subquiver
of $\Gamma_{D^{b}(kQ)}$ (note $n=m+2$).
$$
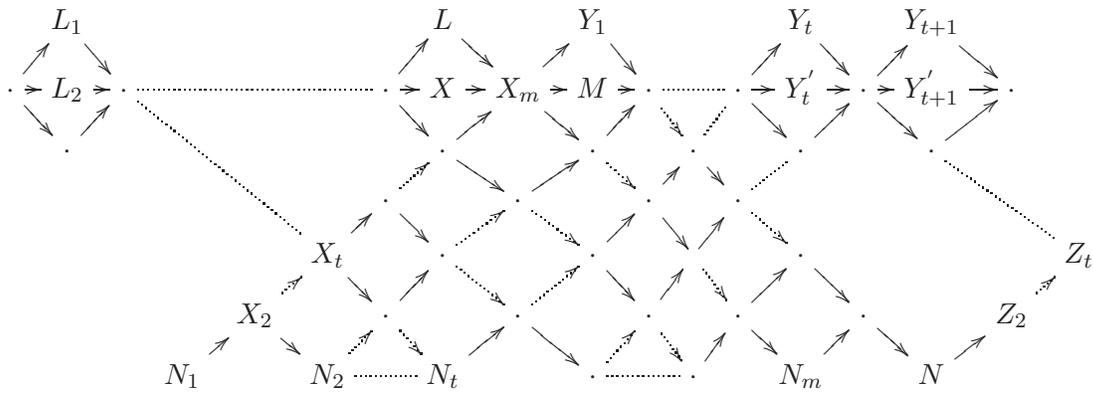
\xymatrix@=0.5pc{& L_1 \ar[rd] & & & & & & L\ar[rd] & & Y_1 \ar[rd] & & & &Y_{t}\ar[rd] & & Y_{t+1}\ar[rd]\\
\cdot\ar[ru]\ar[r]\ar[rd] & L_2 \ar[r] & \cdot \ar@{.}[rrrddd] \ar@{.}[rrrr]& & & & \cdot\ar[ru]\ar[r]\ar[rd]& X \ar[r] & X_m \ar[ru]\ar[r]\ar[rd] & M\ar[r] & \cdot\ar@{.}[rr] \ar@{.>}[rd]&  & \cdot\ar[ru]\ar[r]\ar[rd]  &Y_{t}^{'}\ar[r] &\cdot\ar[ru]\ar[r]\ar[rd] & Y_{t+1}^{'}\ar[r] & \cdot\\
 &\cdot\ar[ru] & & & & & & \cdot \ar[ru] \ar[rd] & & \cdot\ar[ru] \ar@{.>}[rd] & & \cdot\ar@{.}[ru] \ar[rd] & &\cdot\ar[ru] & & \cdot\ar[ru]\ar@{.}[rrdd]\\
& & & & & & \cdot \ar@{.>}[ru] \ar[rd] & & \cdot \ar[ru] \ar@{.>}[rd] & & \cdot \ar[ru] \ar[rd] & & \cdot \ar@{.>}[rd]\ar@{.}[ru]\\
 & & & & & X_t \ar[ru] \ar[rd] & & \cdot \ar@{.>}[ru] \ar@{.>}[rd] & & \cdot \ar[ru] \ar[rd] & & \ar[ru] \ar@{.>}[rd] & & \cdot\ar[rd]& & & & Z_t\\
 & & & &  X_2 \ar@{.>}[ru] \ar[rd]& & \cdot \ar[ru]\ar@{.>}[rd]& & \cdot \ar@{.>}[ru] \ar[rd]& & \cdot \ar[ru] \ar@{.>}[rd]& & \cdot\ar[ru]\ar[rd]& &
 \cdot  \ar[rd] & & Z_2 \ar@{.>}[ru]\\
 & & &  N_1 \ar[ru] & & N_2 \ar@{.>}[ru]\ar@{.}[rr] & &  N_t\ar[ru]& & \cdot  \ar@{.>}[ru] \ar@{.}[rr]& & \cdot \ar[ru]& & N_m\ar[ru]& & N\ar[ru]\\
}
$$

Since $M$ is sincere, $N_1, N\neq 0$. For $\mathbb{D}_n$ type, the number of modules in the $\tau$-orbit of any module not lying on the boundary of the orbit quiver is $n-1=m+1$, which means $N_1$ is projective and $N$ is injective. $M$ is also non-projective, hence $X\ne 0$ and $X_m$ is sincere. Thus $X_m$ has a wing of length $m$, which means there must be non-simple modules at the bottom. Assume $t$ is the minimal integer such that $N_t$ is not simple.

If $t=m$, $X_{m-1}$ has a wing of length $m-1$, whose bottom modules are all simple modules. Since $N_1$ is projective, $X_m$ must be projective. Thus $\underline{\dim}X_m=(1,1,\cdots,1)$. But $X_m\twoheadrightarrow M$, which contradicts the sincereness of $M$. Hence $t<m$. Then we claim $X_t\rightarrow M$ is a monomorphism. Otherwise, $X_t$ is also sincere, so $L_1, L_2\neq 0$. Notice again $N_1$ is a simple projective module and $N_2,\cdots,N_{t-1}$ are all simple modules. Hence $X_t$ has to be projective. The connectedness of the projective slice helps us to conclude that $L_1$ and $L_2$ are both projective. But $\Hom(L_i,M)=0$ for $i=1$ or $i=2$, which contradicts the sincereness of $M$.
Now that $X_t\hookrightarrow M$, the cokernel is $Y_{t+1}\neq 0$ if $t$ is odd or $Y_{t+1}^{'}\neq 0$ if $t$ is even.

Then check MHL for the short exact sequence
$0\rightarrow X_t\rightarrow M\rightarrow Y_{t+1}(Y_{t+1}^{'})\rightarrow 0$. Note that all modules lying on $\Sigma_{\rightarrow}(N_i)$ (excluding $N_i$), $1\leq i\leq t-1$ can not be simple modules. So we only need to inspect the area between $\Sigma_{\rightarrow}(N_t)$
and $\Sigma_{\leftarrow}(Y_{t+1})$ ($\Sigma_{\leftarrow}(Y_{t+1}^{'})$). Assume $t$ is odd. By calculating the hammocks of $\Hom(X_t,-)$ and $\Hom(M,-)$, it is easy to check the exactness of $\Hom(-,S_{i})$ if $S_i$ is in the area excluding the loci of some definitely non-simple modules and $Y_t$. Note that $0\rightarrow X\rightarrow Y_1\rightarrow N\rightarrow 0$ is a short exact sequence. One can also get $0\rightarrow X\rightarrow Y_{2}^{'}\rightarrow Z_2\rightarrow 0$ as long as $Z_2\neq 0$. Now $Z_t\neq 0$ since the number of modules in the orbit of $X_t$ is $m+1=n-1$ and $X_t$ is projective. Hence $0\rightarrow X\rightarrow Y_{t}\rightarrow Z_t\rightarrow 0$ is also a short exact sequence. Thus $Y_t$ can not be a simple module. The case for even $t$ is similar.

If $M$ is not on the boundary or center of the orbit quiver, we use the picture below. Note that $n=m+2$, $M=X_t$ and $k=m-t+2$, $m-t\geq 1$. One could also write $Z_{k}=N_{1}$ for the coordination of symbols.
$$
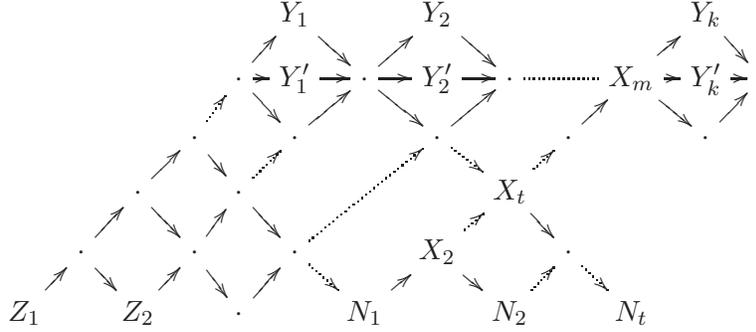
\xymatrix@=0.5pc{
& & & & & Y_1 \ar[rd]& & Y_2\ar[rd] & & & & Y_k\ar[rd]\\
& & & & \cdot \ar[ru]\ar[r]\ar[rd] & Y_{1}' \ar[r] & \cdot \ar[ru]\ar[r]\ar[rd] & Y_{2}' \ar[r]& \cdot\ar@{.}[rr]& & X_{m}\ar[ru]\ar[r]\ar[rd]
& Y_{k}'\ar[r] & \cdot\\
& & & \cdot\ar@{.>}[ru] \ar[rd]& & \cdot \ar[ru] & & \cdot\ar[ru] \ar@{.>}[rd]& & \cdot\ar[ru] & & \cdot\ar[ru]\\
& & \cdot \ar[ru]\ar[rd] & & \cdot \ar@{.>}[ru]\ar[rd] & & & & X_t \ar@{.>}[ru]\ar[rd]\\
& \cdot \ar[ru]\ar[rd] & & \cdot \ar[ru]\ar[rd] & & \cdot \ar@{.>}[rruu]\ar@{.>}[rd] & & X_2 \ar@{.>}[ru] \ar[rd] & & \cdot\ar@{.>}[rd]\\
Z_{1}\ar[ru] & & Z_{2} \ar[ru] & & \cdot\ar[ru] & & N_{1}\ar[ru] & & N_{2} \ar@{.>}[ru] & & N_{t}
}$$

Since $M$ is sincere, $M=X_t$ has a wing of length $t$. If one of $N_i$, $1\leq i\leq t-1$ is not simple, we are done by Lemma \ref{wing}.
Assume now $N_i$, $1\leq i\leq t-1$ are all simple modules. Thus $N_1$ cannot be projective, for otherwise we can deduce that $M$ is projective, which is a contradiction. By drawing the hammock of $\Hom(-,X_t)$, we can easily see none of $Z_2,\cdots,Z_{k-1}$ are projective modules because $X_t$ is sincere. Thus $Z_1\neq 0$. We get a short exact sequence $0\rightarrow Y_1\rightarrow M\rightarrow Y_k(Y_k')\rightarrow 0$, where the third term is $Y_k$ if $k$ is even and $Y_k'$ if $k$ is odd. Assume $k$ is even. There exists a series of short exact sequences $Z_1\hookrightarrow Y_1\twoheadrightarrow Y_2'$, $Z_2\hookrightarrow Y_2'\twoheadrightarrow Y_3$, $\cdots$, until $Z_{k-1}\hookrightarrow Y_{k-1}\twoheadrightarrow Y_{k}'$. Thus none of $Y_1,Y_2',\cdots,Y_{k-2}',Y_{k-1}$ are simple modules. Then it is an easy task to check MHL in the area  between $\Sigma_{\rightarrow}(Y_1)$ and $\Sigma_{\leftarrow}(Y_{k})$. The case for odd $k$ is similar.
\\
\textit{$\mathbb{E}_{6}$ type}

In this case, each sincere indecomposable module lies on the center or the quasi-centers, or on the boundary. So we only need to consider the latter two cases.

Assume $M$ lies on the quasi-centers of the orbit quiver, which means $sl(M)=1$ and $\alpha(M)=2$. Look at the following subquiver:
$$
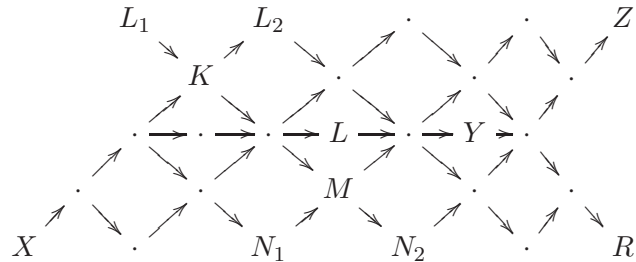
\xymatrix@=0.5pc{
& &  L_1\ar[rd] & & L_2\ar[rd] & & \cdot \ar[rd] & & \cdot\ar[rd] & & Z\\
& & & K\ar[ru]\ar[rd]& & \cdot \ar[ru]\ar[rd]& & \cdot \ar[ru]\ar[rd] & & \cdot \ar[ru]\\
& & \cdot \ar[ru]\ar[rd]\ar[r]& \cdot\ar[r]& \cdot \ar[ru]\ar[rd]\ar[r] & L\ar[r]& \cdot \ar[ru]\ar[rd]\ar[r]& Y\ar[r] & \cdot\ar[ru]\ar[rd]\\
& \cdot \ar[ru]\ar[rd] & & \cdot \ar[ru]\ar[rd] & & M\ar[ru]\ar[rd]& & \cdot \ar[ru]\ar[rd]& & \cdot\ar[rd]\\
X\ar[ru]& & \cdot\ar[ru] & & N_1\ar[ru]& & N_2\ar[ru]& & \cdot\ar[ru]& & R\\
}$$

If $N_1$ is not simple, then $0\rightarrow N_1\rightarrow M\rightarrow N_2\rightarrow 0$ is our desired exact sequence.
Otherwise, we get a short exact sequence $0\rightarrow N_1\rightarrow Y\rightarrow Z\rightarrow 0$, so $Z\neq 0$.
Moreover, $N_1$ can not be projective because $M$ is non-projective. Then by calculating the hammock of $\Hom(-,M)$, we deduce that $X\neq 0$.
Assume $N_2$ is injective, thus all modules along $N_2\rightarrow \cdots\rightarrow Z$ are injective. Then the quiver $Q$ has only two possible orientations, from which we can conclude $N_1$ must be non-simple. This is a contradiction. Hence $N_2$ is not injective, which implies $R\neq 0$.
Thus $0\rightarrow \tau L\rightarrow M\rightarrow \tau Z\rightarrow 0$ is a short exact sequence $(**)$.  Notice that we have a short exact sequence
$0\rightarrow \tau^{-1}X\rightarrow \tau^{-1}L_2\rightarrow R\rightarrow 0$, which means $\tau^{-1}L_2$ is non-simple. Then it is obvious to check MHL.

If $M$ lies on the boundary of the orbit quiver, i.e. $sl(M)=1$ and $\alpha(M)=1$, we consider the following subquiver:
$$
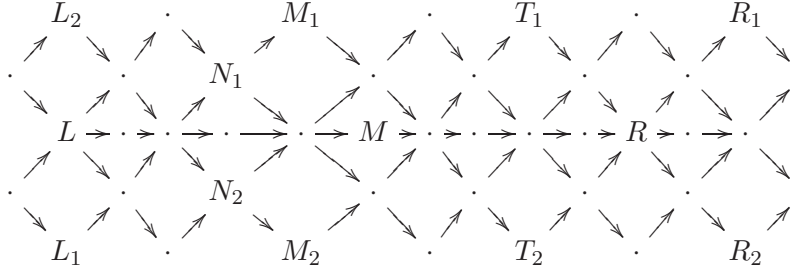
\xymatrix@=0.5pc{
& L_2 \ar[rd]& & \cdot\ar[rd]& & M_1 \ar[rd]& & \cdot \ar[rd] & & T_1 \ar[rd] & & \cdot \ar[rd]& & R_1\ar[rd]\\
\cdot\ar[ru]\ar[rd]& & \cdot \ar[ru]\ar[rd]& & N_1 \ar[ru]\ar[rd]& & \cdot \ar[ru]\ar[rd] & & \cdot\ar[ru]\ar[rd] & & \cdot\ar[ru]\ar[rd] & & \cdot\ar[ru]\ar[rd] & & \cdot\\
& L \ar[ru]\ar[r]\ar[rd]& \cdot\ar[r]& \cdot \ar[ru]\ar[r]\ar[rd]& \cdot\ar[r]& \cdot \ar[ru]\ar[r]\ar[rd]& M\ar[r]& \cdot \ar[ru]\ar[r]\ar[rd]
& \cdot\ar[r] & \cdot \ar[ru]\ar[r]\ar[rd]& \cdot\ar[r]& R \ar[ru]\ar[r]\ar[rd]& \cdot\ar[r]& \cdot \ar[ru]\ar[rd]\\
\cdot\ar[ru]\ar[rd]& & \cdot \ar[ru]\ar[rd]& & N_2 \ar[ru]\ar[rd]& & \cdot \ar[ru]\ar[rd] & & \cdot\ar[ru]\ar[rd] & & \cdot\ar[ru]\ar[rd] & & \cdot\ar[ru]\ar[rd] & & \cdot\\
& L_1\ar[ru] & & \cdot\ar[ru]& & M_2\ar[ru] & & \cdot\ar[ru]& & T_2\ar[ru]& & \cdot\ar[ru]& & R_2\ar[ru]\\
}$$

Recall that $M$ is sincere, non-projective and non-injective, thus $\tau^{2}M\neq 0, \tau^{-2}M\neq 0$.
If $N_1\rightarrow M$ and $N_2\rightarrow M$ are both monomorphisms, $L_1=0, L_2=0$. So $\tau M_1$ and $\tau M_2$ are both projective modules.
Note $R_1, R_2\neq 0$, thus $R\neq 0$. Also $\tau^{-1}L\neq 0$. Recall the numbers of the modules of $\tau$-orbits of $L$ and $M$ are both $6$. Hence
$L$ or $\tau^{-1}L$ must be projective. If $L$ is projective, the quiver has only two possible orientations. If $\tau^{-1}L$ is projective, $\tau^{2}M$ is also projective, then the quiver has four  possible orientations. In each case, it is easy to find the desired short exact sequence.

Otherwise, say $N_1\rightarrow M$ is an epimorphism. Then $R_1=0$. Notice that $\tau^{-1}T_1$ cannot be injective since $M$ is sincere.
So $\tau^{-1}T_1=0$, which indicates $\tau M_1\rightarrow \tau^{-1}M_2$ is an epimorphism. Thus $\tau M_1$ is non-simple. $\tau^{-1}M_2$ is
also non-simple because $\tau^{-1}M_2\rightarrow \tau^{-2}M \rightarrow R_1$ is a triangle. Hence  $0\rightarrow \tau M_1\rightarrow M\rightarrow T_2 \rightarrow 0$ is our desired short exact sequence.
\\
\textit{$\mathbb{E}_{7}$ type}

We first assume $\alpha(M)=2$ and $sl(M)=2$.
$$
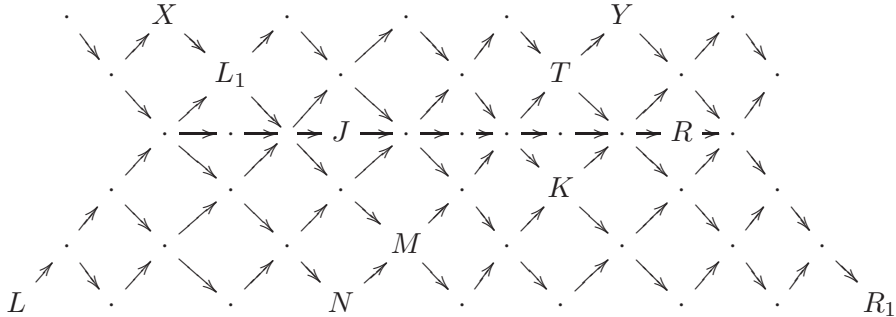
\xymatrix@=0.5pc{
&\cdot\ar[rd] & & X\ar[rd] & & \cdot \ar[rd] & & \cdot \ar[rd] & & \cdot\ar[rd]& & Y\ar[rd] & & \cdot\ar[rd]\\
& & \cdot\ar[ru]\ar[rd] & & L_1\ar[ru]\ar[rd]& & \cdot\ar[ru]\ar[rd] & & \cdot\ar[ru]\ar[rd]& & T\ar[ru]\ar[rd]& &
\cdot\ar[ru]\ar[rd]& & \cdot\\
& & &\cdot\ar[ru]\ar[r]\ar[rd]& \cdot\ar[r]&\ar[ru]\ar[r]\ar[rd]&J\ar[r]& \cdot \ar[ru]\ar[r]\ar[rd]& \cdot\ar[r]& \cdot \ar[ru]\ar[r]\ar[rd]&\cdot\ar[r]&
\cdot \ar[ru]\ar[r]\ar[rd]& R\ar[r]& \cdot\ar[ru]\ar[rd]\\
& & \cdot\ar[ru]\ar[rd]& & \cdot\ar[ru]\ar[rd]& & \cdot\ar[ru]\ar[rd]& & \cdot\ar[ru]\ar[rd]& & K\ar[ru]\ar[rd]& & \cdot\ar[ru]\ar[rd]& &
\cdot\ar[rd]\\
& \cdot\ar[ru]\ar[rd]& & \cdot\ar[ru]\ar[rd]& & \cdot\ar[ru]\ar[rd]& & M\ar[ru]\ar[rd]& & \cdot\ar[ru]\ar[rd]& & \cdot\ar[ru]\ar[rd]& &
\cdot\ar[ru]\ar[rd]& & \cdot\ar[rd]\\
L\ar[ru]& & \cdot\ar[ru] & & \cdot\ar[ru] & & N\ar[ru] & & \cdot\ar[ru]& & \cdot\ar[ru] & & \cdot\ar[ru]& & \cdot\ar[ru]& & R_1\\
}$$

Since $M$ is sincere, $X, Y, N, \tau^{-1}N\neq 0$. If $N$ is non-simple, $0\rightarrow N\rightarrow M\rightarrow \tau^{-1}N\rightarrow 0$ is
our desired short exact sequence. Otherwise, as $N$ is simple, it is obvious $\tau^{-1}M, K\neq 0$. Note we have short exact sequences starting from a simple module $N$: $0\rightarrow N\rightarrow T\rightarrow R\rightarrow 0$, $0\rightarrow N\rightarrow \tau R\rightarrow \tau^{-1}Y\rightarrow 0$, $0\rightarrow N\rightarrow Y\rightarrow R_1\rightarrow 0$, so get $R\neq 0$, $\tau^{-1}Y\neq 0$ and $R_1\neq 0$ in turn.
As before $N$ cannot be a projective module, so $L\neq 0$ since $M$ is sincere. Recall that the number of modules in any $\tau$-orbit is $9$, thus
$L$ is projective and $R_1$ is injective. From the triangle $\tau X\rightarrow N\rightarrow \tau R_1$, we can infer that $\tau X=0$, hence $X$ is projective. If $X$ is simple, $L_1$ must be projective. Then the quiver $Q$ has only two possible orientations. For one of them, $N$ is not simple which contradicts our assumption. The other one is $$\xymatrix{& & \cdot\ar[d]\\
\cdot& \cdot\ar[l]\ar[r]& \cdot\ar[r]& \cdot\ar[r]& \cdot\ar[r]& \cdot\\}$$
and $0\rightarrow \tau J\rightarrow M\rightarrow \tau R_1\rightarrow 0$ is our desired short exact sequence.
If $X$ is not simple, we claim that $0\rightarrow X\rightarrow M\rightarrow Y\rightarrow 0$ is a short exact sequence $(**)$ satisfying MHL. It is
easy to check $\tau Y$ and $J$ are both non-simple modules because of the short exact sequences $0\rightarrow \tau N\rightarrow \tau Y\rightarrow \tau R_1\rightarrow 0$ and $0\rightarrow \tau^{-1} L\rightarrow J\rightarrow \tau Y\rightarrow 0$. Then our claim follows.

The case that $\alpha(M)=2$ and $sl(M)=1$ is similar to the case above, so we just outline the proof. Notice that
$M$ has a wing of length $t$ ($t=2$ or $3$). If one of the first $t-1$ modules at the bottom of the wing is non-simple, we are done.
Otherwise, the simple modules at the bottom would help us to decide all possible orientations of the quiver. Then it is routine to get the desired short exact sequence.

Now we consider the cases $\alpha(M)=1$. First assume $sl(M)=1$.

$$
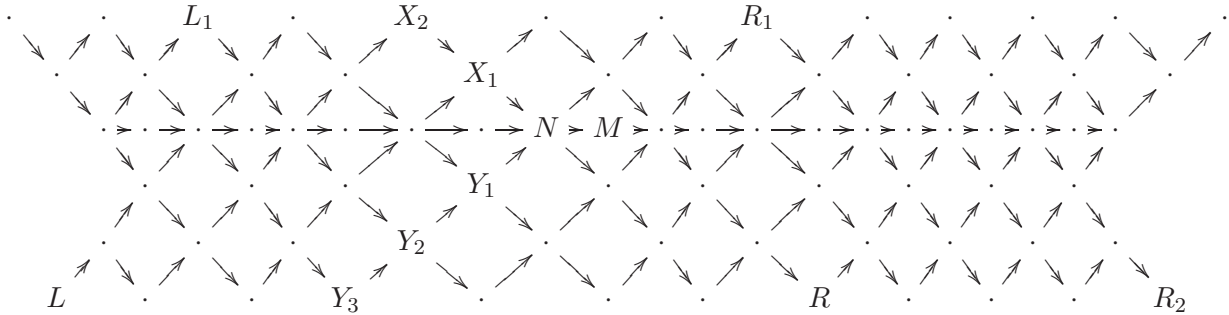
\xymatrix@=0.4pc{
\cdot\ar[rd]& & \cdot\ar[rd] & & L_1\ar[rd] & & \cdot \ar[rd] & & X_2 \ar[rd] & & \cdot\ar[rd]& & \cdot\ar[rd] & & R_1\ar[rd]& &\cdot\ar[rd]& &
\cdot\ar[rd]& & \cdot\ar[rd] & & \cdot\ar[rd]& &\cdot\\
&\cdot\ar[ru]\ar[rd]&  & \cdot\ar[ru]\ar[rd] & &\cdot\ar[ru]\ar[rd] & &\cdot\ar[ru]\ar[rd] & & X_1\ar[ru]\ar[rd]& & \cdot\ar[ru]\ar[rd] & & \cdot\ar[ru]\ar[rd]& & \cdot\ar[ru]\ar[rd]& &\cdot\ar[ru]\ar[rd]& & \cdot\ar[ru]\ar[rd]& & \cdot\ar[ru]\ar[rd]& &\cdot\ar[ru]\\
& &\cdot\ar[ru]\ar[r]\ar[rd]& \cdot\ar[r]& \cdot \ar[ru]\ar[r]\ar[rd]& \cdot\ar[r]&\cdot\ar[ru]\ar[r]\ar[rd]& \cdot\ar[r]&\cdot\ar[ru]\ar[r]\ar[rd]& \cdot\ar[r]& N\ar[ru]\ar[r]\ar[rd]& M\ar[r]& \cdot \ar[ru]\ar[r]\ar[rd]&\cdot\ar[r]&
\cdot\ar[ru]\ar[r]\ar[rd]& \cdot\ar[r]&\cdot\ar[ru]\ar[r]\ar[rd]& \cdot\ar[r]& \cdot\ar[ru]\ar[r]\ar[rd]& \cdot\ar[r]& \cdot\ar[rd]\ar[r]\ar[ru]
& \cdot\ar[r]& \cdot\ar[ru]\\
& & & \cdot\ar[ru]\ar[rd]& & \cdot\ar[ru]\ar[rd]& & \cdot\ar[ru]\ar[rd]& & Y_1\ar[ru]\ar[rd]& &\cdot\ar[ru]\ar[rd]& & \cdot\ar[ru]\ar[rd]& & \cdot\ar[ru]\ar[rd]& & \cdot\ar[ru]\ar[rd]& &\cdot\ar[ru]\ar[rd]& &
\cdot\ar[rd]\ar[ru]\\
& & \cdot\ar[ru]\ar[rd]& &\cdot\ar[ru]\ar[rd]& &  \cdot\ar[ru]\ar[rd]& &Y_2\ar[ru]\ar[rd]& &  \cdot\ar[ru]\ar[rd]& & \cdot\ar[ru]\ar[rd]& &\cdot\ar[ru]\ar[rd]& &  \cdot\ar[ru]\ar[rd]& & \cdot\ar[ru]\ar[rd]& &
\cdot\ar[ru]\ar[rd]& & \cdot\ar[rd]\\
& L\ar[ru]& & \cdot\ar[ru] & & \cdot\ar[ru] & &Y_3\ar[ru] & &\cdot\ar[ru] & &\cdot\ar[ru] & &\cdot\ar[ru] & & R\ar[ru] & & \cdot\ar[ru]& & \cdot\ar[ru] & & \cdot\ar[ru]& & R_2\\
}$$

$M$ is sincere, so $X_2, R_1, Y_3, R\neq 0$. If the composition $Y_2\rightarrow Y_1\rightarrow N\rightarrow M$ is an epimorphism,  $Y_2$ is sincere and $0\rightarrow L\rightarrow Y_2\rightarrow M\rightarrow 0$ is a short exact sequence. Hence $L\neq 0$. Note $R$ is injective because $\Hom(M,\tau^{-1}R)=0$ and $\tau^{-2}R=L[1]$. From the triangle $Y_3\rightarrow \tau R_1\rightarrow \tau^{-1}R$, we get $Y_3$ is non-simple. Notice that $M$ is non-projective and non-injective so $\tau^{2}M, \tau^{-2}M\neq 0$. Then from the triangle $\tau R_1\rightarrow \tau^{-2}M\rightarrow \tau^{-2}R$, we know $\tau R_1$ is also non-simple. Thus
$0\rightarrow Y_3\rightarrow M\rightarrow R_1\rightarrow 0$ is a short exact sequence $(**)$ satisfying MHL.

If $Y_2\rightarrow M$ is a monomorphism, $0\rightarrow Y_2\rightarrow M\rightarrow  \tau^{-2}R\rightarrow 0$ is a short exact sequence, so $\tau^{-2}R\neq 0$. Assume the composition $X_1\rightarrow N\rightarrow M$ is an epimorphism, thus $\tau Y_3\neq 0$. Then $\tau^{2} Y_3\neq 0$, so must be projective. And $\tau^{-2}R$ is injective. If $R_1$ is also injective, then the quiver has only two possible orientations and it is easy to get the desired short exact sequence. Otherwise, $\tau^{-2}R_1\neq 0$. Hence $0\rightarrow X_2\rightarrow M\rightarrow  R\rightarrow 0$ is the desired one. This follows from the fact that $\tau^{2}L_1\rightarrow \tau Y_3\rightarrow \tau R$ is a triangle and  $0\rightarrow \tau^{2}Y_3\rightarrow X_2\rightarrow  \tau R\rightarrow 0$ is a short exact sequence.

Now assume $Y_2\rightarrow M$ and $X_1\rightarrow M$ are both monomorphisms. Since $0\rightarrow X_1\rightarrow M\rightarrow  R_2\rightarrow 0$
is a short exact sequence, $R_2\neq 0$. So $R_2$ is injective and $Y_3$ is projective. Then $\tau^{-2}R_1\neq 0$.  If $Y_1\rightarrow M$ is an epimorphism, $\tau X_2\neq 0$. Thus $L_1\neq 0$. If $\tau L_1\neq 0$, $\tau L_1$ is projective. Then the quiver has only two possible orientations, which can be easily verified. Consider the case $\tau L_1=0$, then $L_1$ is projective. If $Y_3$ is not simple, then $0\rightarrow Y_3\rightarrow M\rightarrow  R_1\rightarrow 0$ is the desired short exact sequence. If $Y_3$ is simple, then $Y_2$ is also projective. Notice that $\tau^{2} M\neq 0$ and $\tau N\neq 0$, thus the quiver also has only two possible orientations. In the end, we assume that $Y_1\rightarrow M$ is also a monomorphism,
then $\tau^{-5} R_1\neq 0$. Hence $X_2$ is projective. If $Y_3$ is not simple,  $0\rightarrow Y_3\rightarrow M\rightarrow  R_1\rightarrow 0$ is the desired one. If $X_2$ is not simple,  $0\rightarrow X_2\rightarrow M\rightarrow  R\rightarrow 0$ is the desired one. If both of $Y_3$ and $X_2$ are
simple modules, then $Y_2$ and $X_1$ are both projective. Then the quiver also has only two possible orientations. It is easy to verify by direct
calculation.

If $sl(M)=2$, we consider the following section of the AR-quiver:

$$
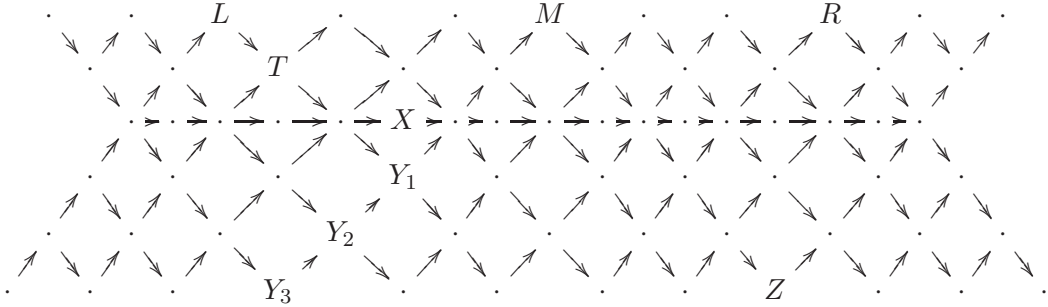
\xymatrix@=0.4pc{
 & \cdot\ar[rd] & & \cdot\ar[rd] & & L \ar[rd] & & \cdot \ar[rd] & & \cdot\ar[rd]& & M\ar[rd] & & \cdot\ar[rd]& &\cdot\ar[rd]& &
R\ar[rd]& & \cdot\ar[rd] & & \cdot\\
 &  & \cdot\ar[ru]\ar[rd] & &\cdot\ar[ru]\ar[rd] & &T\ar[ru]\ar[rd] & & \cdot\ar[ru]\ar[rd]& & \cdot\ar[ru]\ar[rd] & & \cdot\ar[ru]\ar[rd]& & \cdot\ar[ru]\ar[rd]& &\cdot\ar[ru]\ar[rd]& & \cdot\ar[ru]\ar[rd]& & \cdot\ar[ru]\\
 & & & \cdot \ar[ru]\ar[r]\ar[rd]& \cdot\ar[r]&\cdot\ar[ru]\ar[r]\ar[rd]& \cdot\ar[r]&\cdot\ar[ru]\ar[r]\ar[rd]& X\ar[r]& \cdot\ar[ru]\ar[r]\ar[rd]& \cdot\ar[r]& \cdot \ar[ru]\ar[r]\ar[rd]&\cdot\ar[r]&
\cdot\ar[ru]\ar[r]\ar[rd]& \cdot\ar[r]&\cdot\ar[ru]\ar[r]\ar[rd]& \cdot\ar[r]& \cdot\ar[ru]\ar[r]\ar[rd]& \cdot\ar[r]& \cdot\ar[rd]\ar[ru]\\
 & & \cdot\ar[ru]\ar[rd]& & \cdot\ar[ru]\ar[rd]& & \cdot\ar[ru]\ar[rd]& & Y_1\ar[ru]\ar[rd]& &\cdot\ar[ru]\ar[rd]& & \cdot\ar[ru]\ar[rd]& & \cdot\ar[ru]\ar[rd]& & \cdot\ar[ru]\ar[rd]& &\cdot\ar[ru]\ar[rd]& &
\cdot\ar[rd]\\
 & \cdot\ar[ru]\ar[rd]& &\cdot\ar[ru]\ar[rd]& &  \cdot\ar[ru]\ar[rd]& &Y_2\ar[ru]\ar[rd]& &  \cdot\ar[ru]\ar[rd]& & \cdot\ar[ru]\ar[rd]& &\cdot\ar[ru]\ar[rd]& &  \cdot\ar[ru]\ar[rd]& & \cdot\ar[ru]\ar[rd]& &
\cdot\ar[ru]\ar[rd]& & \cdot\ar[rd]\\
 \cdot\ar[ru]& & \cdot\ar[ru] & & \cdot\ar[ru] & &Y_3\ar[ru] & &\cdot\ar[ru] & &\cdot\ar[ru] & &\cdot\ar[ru] & & \cdot\ar[ru] & & Z\ar[ru]& & \cdot\ar[ru] & & \cdot\ar[ru]& & \cdot\\
}$$

Since $M$ is sincere, non-projective and non-injective, we have $Y_3, Z, X, L, R\neq 0$. Assume that $Y_2\rightarrow M$ is an epimorphism,
then $\tau L\neq 0$ because of $0\rightarrow \tau L\rightarrow Y_2\rightarrow  M\rightarrow 0$.  Note $\tau^{-1} R$ cannot be injective
because $M$ is sincere. So $R$ is injective and $\tau^{2} L$ is projective. If $Y_3$ is projective, the quiver has only two possible orientations. Then $0\rightarrow  Y_3\rightarrow M\rightarrow  Z\rightarrow 0$ is the desired short exact sequence. This follows from the fact $0\rightarrow \tau^{2} L\rightarrow Y_3\rightarrow  \tau Z\rightarrow 0$ is a short exact sequence and $\tau Z\rightarrow \tau^{-5}X\rightarrow  \tau^{-2} R$ is a triangle. If $Y_3$ is not projective, then $\tau^{3} Y_3\neq 0$ because $Y_2$ is sincere. Thus $\tau^{3} Y_3$ is projective and $Z$ is injective.
Note $X\rightarrow M$ is an epimorphism, so $X$ is sincere and non-projective, which means $\tau^{2} X\neq 0$. When $\tau^{-5}X\neq 0$, $0\rightarrow  Y_3\rightarrow M\rightarrow  Z\rightarrow 0$ is the desired one. When $\tau^{-5}X= 0$, $\tau^{-4}X= 0$. Notice $\tau^{-2} X\neq 0$, $\tau^{-3}X$ must
be injective. Hence the quiver has three possible orientations.

Now assume $Y_2\rightarrow M$ is a monomorphism. Obviously $L$ is projective and $\tau^{-2} R$ is injective. If $X\rightarrow M$ is an epimorphism,
$\tau Y_3\neq 0$. Thus $\tau^{3} Y_3\neq 0$ since $M$ and $X$ are both sincere. Since $\tau^{3} Y_3$ is projective, the quiver has only two possible orientations. If $X\rightarrow M$ is a monomorphism, $\tau^{-3}Z\neq 0$. Then $Y_3$ is projective and $\tau^{-3}Z$ is injective. If $Y_3$ is non-simple, $0\rightarrow  Y_3\rightarrow M\rightarrow  Z\rightarrow 0$ is the desired short exact sequence. One only needs to notice that $\tau^{2} L\rightarrow Y_3\rightarrow  \tau Z$ is a triangle. If $L$ is non-simple, $0\rightarrow  L\rightarrow M\rightarrow  R\rightarrow 0$ is the desired
one. By calculating the hammock of $\Hom(L,-)$, one only needs to check that $\tau R, X, \tau^{-2}X$ and $\tau^{-2} Y_3$ are all non-simple.  This is easy for $0\rightarrow \tau^{-2} Y_3\rightarrow \tau R\rightarrow  \tau^{-2}Z\rightarrow 0$ and  $0\rightarrow L \rightarrow \tau^{-2} Y_3\rightarrow  \tau^{-1} Z\rightarrow 0$ are both short exact sequences. If $Y_3$ and $L$ are both simple modules, then $Y_2$ and $T$ are both projective.
When $Y_1\rightarrow M$ is an epimorphism, $\tau X\neq 0, \tau Y_1\neq 0$ and the quiver has three possible orientations.
When $Y_1\rightarrow M$ is a monomorphism, $X$ is projective and the quiver has two possible orientations.
The existence of the desired short exact sequence can be shown by calculation.

If $sl(M)=3$, then $\underline{\dim} M=(1,1,\cdots,1)$. Thus the short exact sequence $(**)$ satisfying MHL can be easily found.\\
\textit{$\mathbb{E}_{8}$ type}

The proof is omitted for it is similar to the case of type $\mathbb{E}_{7}$.

We finish the proof of Theorem \ref{sincere}.\qed
\begin{proposition}\label{derived}
Let $(\tilde{\mathfrak{n}}^{+})^{2}=[\tilde{\mathfrak{n}}^{+},\tilde{\mathfrak{n}}^{+}]$ be the derived algebra of the Lie algebra $\tilde{\mathfrak{n}}^{+}$. Then for any indecomposable non-projective $kQ$-module $M$, $C_{M}\in (\tilde{\mathfrak{n}}^{+})^{2}$.
\end{proposition}
\bp
 Let $M$ be an indecomposable non-projective $kQ$-module. If $M$ is sincere, then according to Theorem \ref{sincere}, there exist two indecomposable modules $M_1$ and $M_2$ such that $0\rightarrow M_1 \rightarrow M\rightarrow M_2\rightarrow 0$ is a short exact sequence $(**)$ satisfying MHL.
So we have $C_M=[C_{M_2},C_{M_1}]$ with $\underline{\dim}M_1, \ \underline{\dim}M_2 <\underline{\dim}M$.

Assume $M$ is not sincere as a $kQ$-module. It must be a sincere indecomposable $kQ'$-module with $Q'$ a full subquiver of $Q$.
If $M$ is still non-projective as a $kQ'$-module, we can find two indecomposable $kQ'$-modules $M_1'$ and $M_2'$ such that
$0\rightarrow M_1' \rightarrow M\rightarrow M_2'\rightarrow 0$ is a short exact sequence satisfying $\Hom_{kQ'}(M_1',M_2')=0$
and $0\rightarrow \Hom_{kQ'}(M_2',S_i) \rightarrow \Hom_{kQ'}(M,S_i)\rightarrow \Hom_{kQ'}(M_1',S_i)\rightarrow 0$ is exact for each $i\in Q'_{0}$.
Note that $M_1'$ and $M_2'$ are also indecomposable as $kQ$-modules. And for each $j\in Q_{0}-Q'_{0}$, we have $\Hom_{kQ}(M_t',S_j)=0$, $t=1,2$.
Hence we have $\Hom_{kQ}(M_1',M_2')=0$ and $$0\rightarrow \Hom_{kQ}(M_2',S_i) \rightarrow \Hom_{kQ}(M,S_i)\rightarrow \Hom_{kQ}(M_1',S_i)\rightarrow 0$$ is exact for each $i\in Q_{0}$. So $C_M=[C_{M_2'},C_{M_1'}]$ with  $\underline{\dim}M_1',\ \underline{\dim}M_2' <\underline{\dim}M$.

Suppose $M$ is projective as a $kQ'$-module, say $M$ is the indecomposable projective $kQ'$-module $P'_i$ corresponding to $i\in Q'_0$.
If $P'_i$ is also projective as a $kQ$-module, which means $P'_i=P_i$, we have nothing to prove. Otherwise, there exists an epimorphism $f_i: P_i\twoheadrightarrow P'_i$. Then $\Ker f_i$ must be a projective $kQ$-module. As the generators of the module $P_i$, the paths starting from $i$ have at most three directions because $Q$ is of Dynkin type. This means the number of direct summands of $\Ker f_i$ is at most three. We can write $\Ker f_i=\oplus ^{m}_{i=1}P_{t_i}$, $1\leq m\leq 3$, with $t_i\in Q_0$. If $m=1$, notice that $\top P'_i=\top P_i$, thus we have $C_M=C_{P'_i}=[C_{P_{t_{1}}},C_{P_i}]$. If $m=3$, we construct the pushout of $\oplus ^{3}_{i=1}P_{t_i}\hookrightarrow P_i$ and the natural projection
$\oplus ^{3}_{i=1}P_{t_i}\twoheadrightarrow P_{t_1}$ as follows:
\begin{equation*}\xymatrix{& 0\ar[d] & 0\ar[d]\\
& \oplus ^{3}_{i=2}P_{t_i} \ar@{=}[r] \ar[d] & \oplus ^{3}_{i=2}P_{t_i}\ar[d]\\
0\ar[r] & \oplus ^{3}_{i=1}P_{t_i}\ar[d]\ar[r] & P_i \ar[r]\ar[d] & P'_i\ar@{=}[d] \ar[r]& 0\\
0\ar[r] & P_{t_{1}}\ar[r]\ar[d] & Q_1 \ar[r]\ar[d]& P'_i\ar[r] & 0\\
& 0 & 0\\
}\end{equation*}
Obviously, $Q_1$ shares the same top $S_i$ with $P'_i$, so it is an indecomposable $kQ$-module. According to our previous discussion in method (II),
$C_{P'_{i}}=[C_{P_{t_1}}, C_{Q_1}]$. Note $Q_1$ has a minimal projective resolution as $0\rightarrow \oplus ^{3}_{i=2}P_{t_i}\rightarrow P_i \rightarrow Q_1\rightarrow 0$. So we reduce to the case of $m=2$. Then we proceed once again and finally get $$C_M=[C_{P_{t_1}},[C_{P_{t_2}},[C_{P_{t_3}}, C_{P_i}]]].$$
\ep

Now we finish the proof of Claim \ref{generators}. Let $C_{M}\in \tilde{\mathfrak{n}}^{+}$ with $M$ an indecomposable non-projective $kQ$-module.
If $M$ is a sincere non-projective $kQ'$-module with $Q'$ a full subquiver of $Q$, we always have $C_M=[C_{M_2},C_{M_1}]$ with $\underline{\dim}M_1, \ \underline{\dim}M_2 <\underline{\dim}M$, thus proceed by induction on the dimension vector. Once we reach a sincere projective $kQ'$-module, the proof of Proposition \ref{derived} helps us to conclude that $C_{M}$ is generated by $C_{P_{i}}$, $1\leq i\leq n$.

\begin{remark}
For each non-projective simple module $S_i$,  the element $C_{S_i}$  can be generated by $C_{P_i}, i\in Q_0$, based on the minimal projective resolution of $S_i$. This is a special case of our proof above.
\end{remark}

\begin{remark}
If $Q$ is bipartite, $C_{P_i}$ can be seen as Chevalley generators of $\tilde{\mathfrak{n}}^{+}\cong \mathfrak{n}^{+}$. In the view of root space decomposition, for each $C_M$ with $M$
an indecomposable $kQ$-module, the minimal projective resolution of $M$
\begin{equation}\label{root}0\rightarrow \oplus^{m}_{i=1} P_{t_i} \rightarrow \oplus^{n}_{j=1} P_{r_j}\rightarrow M\rightarrow 0\end{equation}
determines its corresponding root $\Sigma^{m}_{i=1}\alpha_{t_i}+\Sigma^{n}_{j=1}\alpha_{r_j}$.

Consider the root space decomposition $\mathfrak{n}^{+}=\oplus_{\alpha\in \Phi^{+}} \mathfrak{n}_{\alpha}$. Let $C_{M}\in \mathfrak{n}_{\beta}$,
and $C_{N}\in \mathfrak{n}_{\gamma}$. The following are equivalent: (1) $[\mathfrak{n}_{\beta},\mathfrak{n}_{\gamma}]\neq 0$;
(2) $\dim \mathfrak{n}_{\beta+\gamma}=1$; (3) There exists an indecomposable $kQ$-module $L$ such that
the minimal projective resolution of $L$ can be achieved by combing the minimal projective resolutions of $M$ and $N$.
\end{remark}
\begin{remark}\label{root rem}
If $Q$ is not bipartite,  the minimal projective resolution of any indecomposable $kQ$-module also gives
the associated Lie algebra $\tilde{\mathfrak{n}}^{+}$ a ``root'' space decomposition. For any $C_{M}$ in $\tilde{\mathfrak{n}}^{+}$,
by abuse of language, let  $\Sigma^{m}_{i=1}\alpha_{t_i}+\Sigma^{n}_{j=1}\alpha_{r_j}$ be its ``root'' as in (\ref{root}) and $m+n$ be its \emph{height}. Obviously, $\tilde{\mathfrak{n}}^{+}$ is also graded by heights. Recall Example \ref{lizi} again,
 we can write $C_{S_2}\in \tilde{\mathfrak{n}}^{+}_{\alpha_{2}+\alpha_{3}},~C_{I_2}\in \tilde{\mathfrak{n}}^{+}_{\alpha_{1}+\alpha_{3}},~C_{S_1}\in \tilde{\mathfrak{n}}^{+}_{\alpha_{1}+\alpha_{2}}$.
 This will be investigated in the next two sections.
\end{remark}

\section{A new model for free 2-step nilpotent Lie algebras}
\subsection{Nilpotent Lie algebras}

Let $\mathfrak{N}$ be a nilpotent Lie algebra and $\mathfrak{N}^{2}=[\mathfrak{N},\mathfrak{N}]$ be its derived algebra. Similarly, one can define
$\mathfrak{N}^{t+1}=[\mathfrak{N}^{t}, \mathfrak{N}]$, which constitute the descending central series of $\mathfrak{N}$. Clearly, $\mathfrak{N}^{m}=0$ for some $m$. If $\mathfrak{N}^{l}\neq 0$ and $\mathfrak{N}^{l+1}=0$, $l$ is called the \emph{nilpotency} of $\mathfrak{N}$.
For convenience, we also write $C^{t}\mathfrak{N}=\mathfrak{N}^{t}$.

Now recall Section \ref{sec lie}. If $Q$ is a Dykin quiver, one can associate Lie algebras  $\tilde{\mathfrak{n}}^{+}$
and  $\hat{\mathfrak{n}}$ to 1-cyclic perfect complexes over $kQ$. $\hat{\mathfrak{n}}$ is abelian, thus nilpotent.
\begin{proposition}
$\tilde{\mathfrak{n}}^{+}$ is a nilpotent Lie algebra.
\end{proposition}
\begin{proof}
Note that $C^{t}\tilde{\mathfrak{n}}^{+}\neq 0$ means the existence of an indecomposable object with height at least $t$.
Thus the result follows from that $\tilde{\mathfrak{n}}^{+}$ is finite dimensional.
\end{proof}
Thus the Lie algebras $\tilde{\mathfrak{n}}^{+}(Q)$ provide a large class of nontrivial nilpotent Lie algebras.
If $Q$ is bipartite, they are exactly the nilpotent Lie algebras which serve as the nilpotent parts of semisimple Lie algebras.
Otherwise, they contribute more kinds of nilpotent Lie algebras which will be treated below.

The following lemma is well known for nilpotent Lie algebras:
\begin{lemma}\rm{(\cite{Sa})}\label{min gene}
If $\mathfrak{N}$ is a nilpotent Lie algebra, the following two assertions are equivalent:

(1) $(x_{1},\cdots, x_{n})$ is a minimal system of generators;

(2) $(x_{1}+\mathfrak{N}^{2},\cdots, x_{n}+\mathfrak{N}^{2})$ is a basis for the vector space $\mathfrak{N}/\mathfrak{N}^{2}$.
\end{lemma}
Now Claim \ref{generators} can be improved as follows, with a slightly different proof.
\begin{proposition}
$\{C_{P_i}~|~1\leq i\leq n\}$ is a minimal system of generators for the Lie algebra $\tilde{\mathfrak{n}}^{+}$.
\end{proposition}
\begin{proof}
According to Proposition \ref{derived}, $C_{M}\in (\tilde{\mathfrak{n}}^{+})^{2}$
for any indecomposable non-projective $M$. So we are done by Lemma \ref{min gene}.
\end{proof}
\subsection{2-step nilpotent Lie algebras}\label{free2}
A Lie algebra $\mathfrak{N}$ is called \emph{2-step nilpotent} (or \emph{metabelian}) if $\mathfrak{N}^{3}=0$ while $\mathfrak{N}^{2}\neq 0$.
Denote the center of $\mathfrak{N}$ by $Z(\mathfrak{N})$. Obviously $\mathfrak{N}^{2}\subset Z(\mathfrak{N})$.
The study of 2-step nilpotent Lie algebras can be traced back to \cite{S67} in the 1960s, and gained its importance both
in the domain of algebra and the domain of Riemannian geometry (cf. \cite{CGS,E,E2,E3,Gau,Gau2,RM} and references therein).
Although it is still very difficult to classify all 2-step nilpotent Lie algebras, there have been many interesting results and
problems relating them.

Let $K$ be an arbitrary field and consider Lie algebras over $K$. If $\mathfrak{N}$ has a minimal system of $n$ generators, we will say $\mathfrak{N}$ has $n$-generators. Note that $\mathfrak{N}$ can be generated by $n$ generators, but by no fewer than $n$.

Fixing $n$, we propose to study the 2-step nilpotent Lie algebras with $n$-generators. Let $\mathfrak{F}$ be the free Lie algebra over $K$
on $n$-generators $y_{1},\cdots,y_{n}$. Thus $\mathfrak{F}$ is infinite-dimensional. Let $\mathfrak{F}_{t}$ be the subspace of $\mathfrak{F}$
generated by all elements of the type $[y_{i_{1}},\cdots,y_{i_{t}}]$ (short for $[[\cdots[y_{i_{1}},y_{i_{2}}]\cdots],y_{i_{t}}]$)
where $i_{j}\in \{1,2,\cdots,n\}$. Then $\mathfrak{F}$ is graded with $\mathfrak{F}_{t}$ as the homogeneous component of degree $t$.
Set $\mathfrak{F}^{t}=\oplus_{j\geq t} \mathfrak{F}_{j}$. Let $\mathfrak{N}(n)=\mathfrak{F}/\mathfrak{F}^{3}$ and $x_{i}$ denote the
image of $y_{i}$ under the canonical surjection $\mathfrak{F}\twoheadrightarrow\mathfrak{N}(n)$. Obviously, the $x_{i}, 1\leq i\leq n$ generate $\mathfrak{N}(n)$
and $\mathfrak{N}(n)$ inherits a grading from $\mathfrak{F}$: $\mathfrak{N}(n)=\oplus^{2}_{j=1}\mathfrak{N}(n)_{j}$, where $\mathfrak{N}(n)_{1}$
is the subspace spanned by $x_{i}$, $1\leq i\leq n$ and $\mathfrak{N}(n)_{2}$ is the subspace spanned by $[x_{i_{i}},x_{i_{2}}]$, $1\leq i_{1}, i_{2}\leq n$. The resulting Lie algebra  $\mathfrak{N}(n)$ is called \emph{free 2-step nilpotent Lie algebra}, which has the universal property as follows:
\begin{proposition}\rm{(\cite{Gau})}
For any 2-step nilpotent Lie algebra $\mathfrak{N}$, and any $n$-generators $z_{1},\cdots,z_{n}$ of $\mathfrak{N}$,
the correspondence $x_{i}\rightarrow z_{i}$ extends uniquely to a homomorphism $\mathfrak{N}(n)\twoheadrightarrow \mathfrak{N}$.
In other words, every 2-step nilpotent Lie algebra with $n$-generators is a quotient of $\mathfrak{N}(n)$.
\end{proposition}
If the characteristic of $K$ is zero, we have $\dim \mathfrak{N}(n)=\frac{n(n+1)}{2}$ (cf. \cite{Gau}).

Let $\mathfrak{I}(n)$ be the set of all ideals $I$ of  $\mathfrak{N}(n)$ such that $\mathfrak{N}(n)/I$ is 2-step nilpotent with $n$-generators.
If $I, J\in \mathfrak{I}(n)$, we say they are \emph{equivalent} if $\mathfrak{N}(n)/I \cong \mathfrak{N}(n)/J$ as Lie algebras.

\begin{proposition}\rm{(\cite{Gau})}
Suppose $I, J\in \mathfrak{I}(n)$. Then $I$ is equivalent to $J$ if and only if there is a $\theta \in \Aut \mathfrak{N}(n)$ satisfying
$\theta(I)=J$.
\end{proposition}

There are two well known models for free 2-step nilpotent Lie algebras and we will give a review of them, with emphasis on the duality theory
and their connections with nilpotent Lie groups, respectively.

\textbf{Model I} \ Let $K$ be an arbitrary field and $V$ the $K$-vector space spanned by the basis $x_{i}$, $1\leq i\leq n$. Then we can make the vector space $V\oplus \wedge^{2}V$ into a Lie algebra by linearly extending the rule:
$$[x_{i},x_{j}]=x_{i}\wedge x_{j},$$
$$[x_{i},x_{j}\wedge x_{k}]=[x_{i}\wedge x_{j},x_{k}]=0,$$
$$[x_{i}\wedge x_{j},x_{k}\wedge x_{l}]=0.$$
It is easy to show that the unique homomorphism from $\mathfrak{N}(n)$ to $V\oplus \wedge^{2}V$ taking $x_{i}$ to $x_{i}$ is an isomorphism.
Thus we make the identification $\mathfrak{N}(n)=V\oplus \wedge^{2}V$.

For any vector space $W$, there is a natural representation $\wedge^{p}$ of $GL(W)$ on $\wedge^{p} W$ given on decomposable $p$-vectors by
$$\wedge^{p}(\theta)(w_{1}\wedge\cdots\wedge w_{p})= \theta(w_{1})\wedge\cdots\wedge \theta(w_{p}).$$

\begin{proposition}\rm{(\cite{Gau})}
Each 2-step nilpotent Lie algebra with $n$-generators is of the type $\mathfrak{N}(n)/I$ where $I$ ranges over the proper
subspaces (ideals) of $\wedge^{2}V$. If $I$ and $J$ are proper subspaces of $\wedge^{2}V$, then $\mathfrak{N}(n)/I \cong \mathfrak{N}(n)/J$ if and only if there is a $\theta \in GL(V)$ such that $\wedge^{2}(\theta)(I)=J$.
\end{proposition}

Scheunemann \cite{S67} established a duality theory for 2-step nilpotent Lie algebras
using the cohomology ring of Lie algebras, which cuts the classification of 2-step nilpotent Lie algebras by half.
Later Gauger \cite{Gau} reconstructed the duality theory via the natural duality between $\wedge^{2}V$ and  $\wedge^{2}V^{*}$ (here $V^{*}$ is the dual vector space of $V$), then showed in \cite{Gau2} that the two duality theories are identical when $K$ is algebraically closed and of characteristic zero. For each 2-step nilpotent Lie algebra $\mathfrak{N}$, let $\mathfrak{N}^{*}$ be the dual of $\mathfrak{N}$ as in \cite{S67}.
\begin{proposition}\rm{(\cite{S67})}
In the set of 2-step nilpotent Lie algebras (here including abelian Lie algebras), there is a correspondence $\mathfrak{N}\rightarrow \mathfrak{N}^{*}$ with the following properties:

(1) If $\dim \mathfrak{N}=n+r$ and $\dim \mathfrak{N}^{2}=r$, then $\dim \mathfrak{N}^{*}=n+\frac{1}{2}n(n-1)-r$ and
$\dim \mathfrak{{N}^{*}}^{2}=\frac{1}{2}n(n-1)-r$.

(2) $(\mathfrak{N}^{*})^{*}$ is isomorphic to $\mathfrak{N}$.

(3) $\mathfrak{N}_{1}$ and $\mathfrak{N}_{2}$ are isomorphic if and only if ${\mathfrak{N}_{1}}^{*}$ and ${\mathfrak{N}_{2}}^{*}$ are isomorphic.
\end{proposition}
Recall that there is a canonical non-degenerate pairing between $\wedge^{2}V$ and  $\wedge^{2}V^{*}$:
$$(v\wedge w, \lambda\wedge \mu)= \lambda(v)\mu(w)-\lambda(w)\mu(v),$$
where $v\wedge w\in \wedge^{2}V$ and $\lambda\wedge \mu\in \wedge^{2}V^{*}$ are decomposable $2$-vectors.

Thus the duality can be easily exhibited by mapping $(V\oplus \wedge^{2}V)/T$ to $(V^{*}\oplus \wedge^{2}V^{*})/T^{\bot}$, where the subspace $T^{\bot}$ of $V^{*}$ is the orthogonal complement of the subspace $T$ of $V$ with respect to the pairing.

\textbf{Model II}\ Set $K=\mathbb{R}$. Following \cite{E}, we say a 2-step nilpotent Lie algebra $\mathfrak{N}$ is of \emph{type} $(m,n)$ if the
commutator ideal $\mathfrak{N}^{2}$ has dimension $m$ and codimension $n$.

Let $so(n,\mathbb{R})$ denote the Lie algebra of $n\times n$ skew symmetric matrices with real entries. Let $\langle, \rangle$ denote the
positive definite inner product on $so(n,\mathbb{R})$ given by $\langle Z,Z'\rangle = - trace(ZZ')$, which will be called the canonical inner
product on  $so(n,\mathbb{R})$.

Let $W$ be an $m$-dimensional subspace of $so(n,\mathbb{R})$. Then the vector space $\mathfrak{N}=\mathbb{R}^{n}\oplus W$ can be equipped with the inner product $\langle, \rangle^{*}$ such that $\langle \mathbb{R}^{n},W \rangle^{*}=0$, $\langle, \rangle^{*}=\langle, \rangle$ on $W$ and $\langle, \rangle^{*}$ is the standard inner product on $\mathbb{R}^{n}$ for which the natural basis $\{e_{1},\cdots, e_{n}\}$
is orthonormal. There is a unique Lie bracket $[,]$ on  $\mathfrak{N}$ such that $W$ lies in the center of  $\mathfrak{N}$ and $\langle [X,Y],Z \rangle^{*}=\langle Z(X),Y\rangle^{*}$ for $X,Y\in \mathbb{R}^{n}$ and $Z\in W$. Thus $\mathfrak{N}$ is a 2-step nilpotent Lie algebra such
that $\mathfrak{N}^{2}=W$. Let $N$ be the unique simply connected Lie group with Lie algebra  $\mathfrak{N}$.  Then the inner product $\langle, \rangle^{*}$ on $\mathfrak{N}$ determines uniquely a left invariant Riemannian metric on $N$. So $\mathfrak{N}=\mathbb{R}^{n}\oplus W$ is called
a \emph{standard metric 2-step nilpotent Lie algebra of type $(m,n)$}.

Now let $\mathfrak{N}$  be a 2-step nilpotent Lie algebra of type $(m,n)$. One can choose $\{Z_{1},\cdots,Z_{m}\}$ to be a basis of $\mathfrak{N}^{2}$, then extend it to a basis $\mathfrak{B}=\{X_{1},\cdots,X_{n},Z_{1},\cdots,Z_{m}\}$ of $\mathfrak{N}$. Hence $[X_{i},X_{j}]=\sum^{m}_{k=1}C^{k}_{ij}Z_{k}$ for $1\leq i,j\leq n$, $1\leq k\leq m$ and suitable matrices $\{C^{1},\cdots,C^{m}\}$ in $so(n,\mathbb{R})$. Let $W=span\{C^{1},\cdots,C^{m}\}\subset so(n,\mathbb{R})$ and $\mathfrak{N}_{s}=\mathbb{R}^{n}\oplus W$ denote the standard
metric 2-step nilpotent Lie algebra of type $(m,n)$. It is not difficult to show that $\mathfrak{N}$ is isomorphic to $\mathfrak{N}_{s}=\mathbb{R}^{n}\oplus W$ as a Lie algebra (cf. \cite{E2}).

Let us return to the 2-step nilpotent Lie algebra $\mathfrak{N}=\mathbb{R}^{n}\oplus W$. If $W=so(n,\mathbb{R})$, the Lie algebra
$\mathbb{R}^{n}\oplus so(n,\mathbb{R})$ is just the free 2-step nilpotent Lie algebra over $\mathbb{R}$ with $n$-generators. So we identify it with $\mathfrak{N}(n)$. Clearly it has type $(\frac{n(n-1)}{2},n)$. It is also obvious that $\mathfrak{N}(n)$ admits a rational structure
$\mathfrak{N}(n)_{\mathbb{Q}}$, which is exactly the free 2-step nilpotent Lie algebra over $\mathbb{Q}$ with $n$-generators.

\subsection{$\tilde{\mathfrak{n}}^{+}$ of linearly oriented quiver of type $\mathbb{A}_{n}$}
Now we consider the Lie algebra $\tilde{\mathfrak{n}}^{+}$ associated to the quiver $\mathbb{A}_{n}^{\rightarrow}$, which is the linearly
oriented quiver of type $\mathbb{A}_{n}$:
$$\xymatrix{1\ar[r]& 2\ar[r]& \cdot\ar@{.>}[r] & \cdot\ar[r] & n.\\}$$
The corresponding path matrix $E=E_{\mathbb{A}_{n}^{\rightarrow}}$ is as follows:
$$E=(a_{ij})=\begin{pmatrix}
0 & 1 & \cdot & \cdot & 1\\
-1 & 0 & 1 & \cdot & 1\\
\cdot & -1 & \cdot & \cdot & \cdot \\
\cdot & \cdot & \cdot & 0 & 1\\
-1 &\cdot & \cdot & -1 & 0
\end{pmatrix}$$
that is $a_{ii}=0$, $a_{ij}=-a_{ji}=1$ for any $i< j$.
\begin{theorem}\label{free}
Let $Q$ be the linearly oriented quiver $\mathbb{A}_{n}^{\rightarrow}$. The associated Lie algebra $\tilde{\mathfrak{n}}^{+}$
is isomorphic to the free 2-step nilpotent Lie algebra $\mathfrak{N}(n)$.
\end{theorem}
\bp
Citing Theorem \ref{main result}, we know that $\tilde{\mathfrak{n}}^{+}$ is generated by $C_{P_{i}}$, $1\leq i\leq n$ and these
generators satisfy $[[C_{P_{i}},C_{P_{j}}],C_{P_{k}}]=0$ for any $i,j,k$. Through comparing the dimensions, we get the isomorphism from
$\tilde{\mathfrak{n}}^{+}$ to $\mathfrak{N}(n)$ by mapping $C_{P_{i}}$ to $x_{i}$.
\ep
Now the identification between $\tilde{\mathfrak{n}}^{+}$ and the Model I (or II) is straightforward. For any indecomposable non-projective module
$M$, consider its minimal projective resolution \begin{equation}\label{mini proj A}0\longrightarrow P_{i}\longrightarrow P_{j}\longrightarrow M\longrightarrow 0, \ \  i>j.\end{equation}
Thus $C_{M}$ corresponds
to $x_{i}\wedge x_{j}$ in the Model I and $(1/2)(E_{ji}-E_{ij})$ in the Model II, where $E_{ab}$ is the $n\times n$  matrix with $1$ in position $(a,b)$ and zeros elsewhere.

Recall that we can also associate an abelian Lie algebra $\hat{\mathfrak{n}}$, which is generated by the $\mathbb{C}$-basis $K_{P_{i}}$, $1\leq i\leq n$.
\begin{corollary}
Within the Scheunemann duality, we have $(\tilde{\mathfrak{n}}^{+})^{*}\cong \hat{\mathfrak{n}}$.
\end{corollary}

Let $\tilde{\mathfrak{n}}_{\mathbb{R}}^{+}$ be the Lie algebra over $\mathbb{R}$ generated by the same basis $C_{M}$ as in $\tilde{\mathfrak{n}}^{+}$. Then  $\tilde{\mathfrak{n}}_{\mathbb{R}}^{+}$ is a real form of $\tilde{\mathfrak{n}}^{+}$. Clearly, thanks to the existence of Hall polynomials, $\tilde{\mathfrak{n}}_{\mathbb{R}}^{+}$ admits a rational form $\tilde{\mathfrak{n}}_{\mathbb{Q}}^{+}$ and a $\mathbb{Z}$-Lie subalgebra $\tilde{\mathfrak{n}}_{\mathbb{Z}}^{+}$, both of which have the $C_{M}$ as bases. The notations here are actually valid for any quiver of Dynkin type, not just $\mathbb{A}_{n}^{\rightarrow}$.

\subsection{}Before characterizing the root system of $\tilde{\mathfrak{n}}^{+}(\mathbb{A}_{n}^{\rightarrow})$, we firstly give a review of
the theory of root systems of nilpotent Lie algebras of maximal rank, following Santharoubane \cite{Sa}.

Let $\mathfrak{N}$ be a nilpotent complex Lie algebra of finite dimension. Recall that a \emph{derivation} of $\mathfrak{N}$ is a linear map $\delta: \mathfrak{N}\rightarrow \mathfrak{N}$ satisfying Leibnitz rule $\delta([a,b])=[a,\delta(b)]+[\delta(a),b]$. Let $\Der \mathfrak{N}$ be
the derivation algebra of $\mathfrak{N}$, consisting of all derivations of $\mathfrak{N}$. Denote by $\Aut \mathfrak{N}$
the automorphism group of $\mathfrak{N}$.
\begin{definition}
One calls $T$ a \emph{torus} on $\mathfrak{N}$ a commutative subalgebra of $\Der \mathfrak{N}$ which consists of semi-simple endomorphisms.
A torus is called \emph{maximal} if it is not contained strictly in any other torus.
\end{definition}
A torus $T$ defines a representation in $\mathfrak{N}$: $T\times \mathfrak{N}\rightarrow \mathfrak{N}$, $(t,x)\mapsto tx$.
Since $T$ is a commutative family of semi-simple endomorphisms and the ground field $\mathbb{C}$ is algebraically closed, all elements
of $T$ can be diagonalized  simultaneously. This means, $\mathfrak{N}$ can be decomposed into a direct sum of root spaces for
$$T: \mathfrak{N}=\oplus_{\beta \in T^{*}}\mathfrak{N}_{\beta},$$
where $T^{*}$ is the dual of the vector space $T$ and $\mathfrak{N}_{\beta}=\{x\in\mathfrak{N}| tx=\beta(t)x, \forall t\in T\}$.
\begin{definition}
Let $T$ be a maximal torus on $\mathfrak{N}$. The set $R(T)=\{\beta\in T^{*}| \mathfrak{N}_{\beta}\neq 0\}$ is called the
\emph{root system} of $\mathfrak{N}$ associated to $T$. Each element $x\neq 0$ in $\mathfrak{N}_{\beta}$ is called a \emph{root vector} and
$\beta$ is called the \emph{root} of $x$.
\end{definition}
Recall Lemma \ref{min gene}. Let $T$ be a maximal torus on $\mathfrak{N}$. One calls $T$-msg a minimal system of generators which
consists of root vectors for $T$. It is easy that there exists a $T$-msg for any torus $T$.
\begin{lemma}\rm{(\cite{Sa})}
Let  $T$ be a maximal torus on $\mathfrak{N}$, then $\dim T=d$ is an invariant of $\mathfrak{N}$ called the rank of  $\mathfrak{N}$.
Let $(x_{1},\cdots,x_{n})$ be a $T$-msg, $\beta_{i}$ the root of $x_{i}$, then the rank of $\{\beta_{1},\cdots,\beta_{n}\}$ is equal to $d$.
\end{lemma}
This lemma follows from Mostow's theorem, which stated that for any two maximal tori $T$ and $T'$, there exists $\theta\in \Aut \mathfrak{N}$
such that $\theta T\theta^{-1}=T'$. Clearly $d\leq n$. We say $\mathfrak{N}$ is of \emph{maximal rank} if its rank is $n$, which means $\dim T=\dim \mathfrak{N}/\mathfrak{N}^{2}$ for any maximal torus $T$ on  $\mathfrak{N}$.
\begin{proposition}\rm{(\cite{Sa})}
Let $T$ be a maximal torus on $\mathfrak{N}$ that is nilpotent of maximal rank, $R(T)$ the associated root system, $(x_{1},\cdots,x_{n})$ a $T$-msg and $(\beta_{1},\cdots,\beta_{n})$ the corresponding roots.

(1) The set $(\beta_{1},\cdots,\beta_{n})$  is a basis of the vector space $T^{*}$.

(2) For any $\beta\in R(T)$, there exists unique $(d_{1},\cdots,d_{n})\in \mathbb{N}^{n}$ such that $\beta=\Sigma^{n}_{i=1}d_{i}\beta_{i}$.

(3) If we let $|\beta|=\Sigma^{n}_{i=1} d_{i}$ be the height of $\beta=\Sigma^{n}_{i=1}d_{i}\beta_{i}$, then $1\leq |\beta|\leq p$ where $p$ is the nilpotency of $\mathfrak{N}$.
\end{proposition}
In the following, we give a description of the root system of the Lie algebra $\tilde{\mathfrak{n}}^{+}$ of the quiver  $\mathbb{A}_{n}^{\rightarrow}$. Let $t_{i}\in \Der \tilde{\mathfrak{n}}^{+}$, $1\leq i \leq n$, defined by:
$$t_{i}(C_{P_{j}})=\delta_{ij}C_{P_{j}}.$$
Then $T=\oplus^{n}_{i=1}\mathbb{C}t_{i}$ is a maximal torus on $\tilde{\mathfrak{n}}^{+}$ and the nilpotent Lie algebra $\tilde{\mathfrak{n}}^{+}$
is of maximal rank. Obviously, $(C_{P_{1}},\cdots,C_{P_{n}})$ is a $T$-msg with roots $(\alpha_{1},\cdots,\alpha_{n})$. Here $\alpha_{i}$ satisfying:
$$\alpha_{i}(t_{j})=\delta_{ij}$$
is a simple root. For indecomposable non-projective $kQ$-module $M$ admitting a minimal projective resolution as in (\ref{mini proj A}), $$t(C_{M})=(\alpha_{i}+\alpha_{j})(t)C_{M}, \ \forall t\in T.$$
So the associated root system is $R(T)=\{\alpha_{1},\cdots,\alpha_{n}\}\cup\{\alpha_{i}+\alpha_{j}|1\leq j<i\leq n\}$. Thus Remark \ref{root rem} is
explained for $\mathbb{A}_{n}^{\rightarrow}$.

\section{$\tilde{\mathfrak{n}}^{+}$ of Dynkin quivers}
\subsection{Root systems of $\tilde{\mathfrak{n}}^{+}$}
Let $Q$ be a Dynkin quiver and $E=(a_{ij})$ its path matrix. We have a nilpotent Lie algebra $\tilde{\mathfrak{n}}^{+}$  of the quiver $Q$ as in Theorem
\ref{main result}. For any indecomposable $kQ$-module $M$, we define the \emph{dimension vector} of $C_{M}\in C^1(\mathscr{P})$ (note that without risk of confusion, $C_{M}$ also denotes the corresponding element in $\tilde{\mathfrak{n}}^{+}$) as:
$$\Dim C_{M}=\Dim P_{M}+\Dim \Omega_{M},$$
where $P_{M}$ and $\Omega_{M}$ are defined in (\ref{mini proj res}). Obviously $\Dim C_{M}$ is a positive integer combination of $\Dim P_{i}$, $1\leq i\leq n$. For convenience, we write $\alpha_{M}:=\Dim C_{M}$.

Here we give an interesting lemma, in which $2\nmid (a_{1},\cdots,a_{n})$, $(a_{1},\cdots,a_{n})\in\mathbb{N}^{n}$, means there exists a $j$ such that $a_{j}$ is odd.
\begin{lemma}
Let $Q$ be a Dynkin quiver, $M$ and $N$ be two non-isomorphic indecomposable $kQ$-modules, then $2\nmid \Dim M+\Dim N$.
\end{lemma}
\bp
Assume $\Dim M+\Dim N=2d$, $d\in \mathbb{N}^{n}$.
According to Gabriel's Theorem, $\Dim M$ and $\Dim N$ correspond to different positive roots $\beta$ and $\gamma$ respectively, of the root system of the semisimple Lie algebra associated to $Q$. For $\beta$, there exists $\omega$ belonging to the Weyl group $W$ such that $\omega(\beta)$ is a simple root $\alpha_{i}$. Hence we have $\alpha_{i}+\omega (\gamma)=2e$, $e\in \mathbb{Z}^{n}$. Back to the category of $kQ$-modules, we get $$\Dim S_{i}\pm \Dim L= 2e,$$ where $S_{i}$ is the simple module attached to the vertex $i$ and $L$ is an indecomposable module different from $S_{i}$. By inspecting the dimension vectors of all indecomposable modules (cf. \cite{S}), we know this is impossible.
\ep
We remark that one can also prove the preceding lemma using Euler form.
\begin{corollary}
Let $Q$ be a Dynkin quiver, $M$ and $N$ be two non-isomorphic indecomposable $kQ$-modules, then $\alpha_{M}\neq \alpha_{N}$.
\end{corollary}
\bp
Note that $\alpha_{M}=\Dim C_{M}=\Dim P_{M}+\Dim \Omega_{M}= \Dim M+ 2\Dim \Omega_{M}$, and $\alpha_{N}= \Dim N+ 2\Dim \Omega_{N}$.
If $\alpha_{M}=\alpha_{N}$, then $2\alpha_{M}=\Dim M+ \Dim N+ 2(\Dim \Omega_{M}+\Dim \Omega_{N})$. This is impossible.
\ep
Recall Remark \ref{root rem} and consider the minimal projective resolution of $M$ as in (\ref{root}).
Denote $\Dim C_{P_{i}}=\Dim P_{i}$ by $\alpha_{i}$. Then $\alpha_{M}=\Sigma^{m}_{i=1}\alpha_{t_i}+\Sigma^{n}_{j=1}\alpha_{r_j}$.
It is not harmful to consider $\alpha_{i}$, $1\leq i\leq n$ as the natural basis of $\mathbb{Z}^{n}$.

Let us try to define $t_{i}\in \Der \tilde{\mathfrak{n}}^{+}$, $1\leq i \leq n$ by:
$$t_{i}(C_{P_{j}})=\delta_{ij}C_{P_{j}}.$$
Since different $C_{M}$'s have different dimension vectors $\alpha_{M}$, $t_{i}$ is well defined on $\tilde{\mathfrak{n}}^{+}$.
Set $T=\oplus^{n}_{i=1}\mathbb{C}t_{i}$, thus $T$ is a maximal torus on $\tilde{\mathfrak{n}}^{+}$ and
\begin{lemma}
$\tilde{\mathfrak{n}}^{+}$ is of maximal rank.
\end{lemma}
Let $\alpha'_{i}\in T^{*}$ be dual to $t_{i}$, then $t(C_{P_{i}})=(\alpha'_{i})(t)C_{P_{i}}, \ \forall t\in T.$
We identify $\alpha'_{i}$ with $\alpha_{i}$, then $\alpha_{M}$ can be written as a positive integer combination of $\alpha'_{i}$.
Furthermore, $$t(C_{M})=(\alpha_{M})(t)C_{M}, \ \forall t\in T.$$
It is reasonable to view $\{\alpha_{M}|M\in \mathrm{Ind\ } kQ\}$ as the root system of $\tilde{\mathfrak{n}}^{+}$ with each root space being of dimension one. Actually, $\tilde{\mathfrak{n}}^{+}$ can be treated uniformly in the following way (cf. \cite{Sa,Sa2}).

Since for any $i\neq j$ in the quiver $Q$ the corresponding $C_{P_{i}}$ and $C_{P_{j}}$ satisfy Serre relations,
one can associate a generalized Cartan matrix (denoted G.C.M.) $C_{E}=(c_{ij})$ as follows:

(1) $c_{ii}=2$ for each i;

(2) $c_{ij}=-1$ if $|a_{ij}|=1$;

(3) $c_{ij}=0$ if $a_{ij}=0$.

If $Q$ is bipartite, $C_{E}$ is the same as the Cartan matrix $C$ associated to the underlying graph $\Gamma$ of $Q$. Otherwise,
$C_{E}$ is of affine or wild type because it possesses a submatrix of type $\widetilde{\mathbb{A}_{3}}$ due to the existence of a path of length $2$.

We will call Kac--Moody Lie algebra associated to the G.C.M. $C_{E}$ the Lie algebra $L(C_{E})$ generated by a set
$\{f_{1},\cdots,f_{n},h_{1},\cdots,h_{n},e_{1},\cdots,e_{n}\}$ satisfying relations:
$$\forall i,j=1,\cdots,n, [h_{i},h_{j}]=0;\ [e_{i},f_{j}]=\delta_{ij}h_{i};$$
$$[h_{i},e_{j}]=c_{ij}e_{j},\ [h_{i},f_{j}]=-c_{ij}f_{j};$$
$$\forall i,j=1,\cdots,n, i\neq j, \ (\ad e_{i})^{-c_{ij}+1}e_{j}=0,\ (\ad f_{i})^{-c_{ij}+1}f_{j}=0.$$

Let $\{\alpha_{1},\cdots,\alpha_{n}\}$ be the canonical (natural) basis of $\mathbb{Z}^{n}$. For $\alpha=\Sigma d_{i}\alpha_{i}\in \mathbb{N}^{n}\backslash\{0\}$, denote by $L_{\alpha}$ (resp. $L_{-\alpha}$) the subspace of $L(C_{E})$ generated by the elements
$[e_{i_{1}},\cdots,e_{i_{t}}]$ (resp. $[f_{i_{1}},\cdots,f_{i_{t}}]$) where $e_{i}$ (resp. $f_{i}$) appears $d_{i}$ times.
If $\alpha=\Sigma d_{i}\alpha_{i}\in \mathbb{Z}^{n}$ is such that all the $d_{i}$'s are not of the same sign, let $L_{\alpha}=0$.
Denote $$L_{0}=H=\mathbb{C}h_{1}\oplus\cdots\oplus\mathbb{C}h_{n}.$$
One calls the \emph{root system} of $L(C_{E})$ the set
$$\Delta=\{\alpha\in\mathbb{Z}^{n}| \alpha\neq 0, L_{\alpha}\neq 0\}.$$
The Lie algebra $L(C_{E})$  is graded by
$$\Delta\cup \{0\}:  L(C_{E})=\oplus_{\alpha\in \Delta\cup \{0\}}L_{\alpha}, \ [L_{\alpha},L_{\beta}]\subset L_{\alpha+\beta}, \ \forall \alpha,\beta\in \Delta\cup \{0\}.$$
The set $\Delta_{+}=\{\alpha\in\mathbb{N}^{n}| \alpha\neq 0, L_{\alpha}\neq 0\}$ is called the \emph{positive root system} of $L(C_{E})$. It is
well known that $\Delta=\Delta_{-}\cup \{0\}\cup \Delta_{+}$, where $\Delta_{-}=-\Delta_{+}$. Furthermore, $L(C_{E})$ has a triangular decomposition
$$L(C_{E})=L_{-}(C_{E})\oplus H\oplus L_{+}(C_{E}),$$
where $L_{+}(C_{E})=\oplus_{\alpha\in \Delta_{+}}L_{\alpha}$ is called the \emph{positive part} and $L_{-}(C_{E})=\oplus_{\alpha\in \Delta_{-}}L_{\alpha}$  the \emph{negative part}.

Thus $L_{+}(C_{E})$ is a Lie algebra generated by $\{e_{1},\cdots,e_{n}\}$ satisfying only the Serre relations
$$(\ad e_{i})^{-c_{ij}+1}e_{j}=0,\ \forall i\neq j.$$
If $\alpha=\Sigma d_{i}\alpha_{i}\in \Delta_{+}$, let $|\alpha|=\Sigma d_{i}$ and call $|\alpha|$ the \emph{height} of $\alpha$.
Then we have $$C^{m}L_{+}(C_{E})=\oplus_{|\alpha|\geq m}L_{\alpha},$$
where $C^{m}L_{+}(C_{E})$ is the $m$-th term of the descending central series of $L_{+}(C_{E})$.

Assume that $\tilde{\mathfrak{n}}^{+}$ is of nilpotency $l$, which means $C^{l}\tilde{\mathfrak{n}}^{+}\neq 0$ and $C^{l+1}\tilde{\mathfrak{n}}^{+}=0$.

Let $\mathfrak{m}=\mathfrak{m}_{l}(C_{E})=L_{+}(C_{E})/C^{l+1}L_{+}(C_{E})$ and $\mu: L_{+}(C_{E})\rightarrow \mathfrak{m}_{l}(C_{E}),\ x\mapsto\overline{x}$ the canonical map. Then the Lie algebra $\mathfrak{m}_{l}(C_{E})$ is of nilpotency $l$ and has a minimal system of generators $\{\overline{e_{1}},\cdots,\overline{e_{n}}\}$. Let $t'_{i}\in \Der \mathfrak{m}_{l}(C_{E})$ such that $t'_{i}\overline{e_{j}}=\delta_{ij}\overline{e_{j}}$, then $T'=\oplus^{n}_{i=1}\mathbb{C}t'_{i}$ is a maximal torus on $\mathfrak{m}_{l}(C_{E})$
and the nilpotent Lie algebra $\mathfrak{m}_{l}(C_{E})$ is of maximal rank; furthermore $\{\overline{e_{1}},\cdots,\overline{e_{n}}\}$ is a $T'$-msg.
Let $(t^{*}_{1},\cdots,t^{*}_{n})$ be the dual basis of $(t'_{1},\cdots,t'_{n})$. One can identify $t^{*}_{i}$ and $\alpha_{i}$, then the root space
decomposition relative to $T'$ is identical to the decomposition
$$\mathfrak{m}_{l}(C_{E})=\oplus_{\alpha\in \Delta_{+},|\alpha|\leq l}\overline{L_{\alpha}},$$
where the restriction of $\mu$ to $L_{\alpha}$ such that $|\alpha|\leq l$ is  an isomorphism from  $L_{\alpha}$ to $\overline{L_{\alpha}}$.
Thus the root system of $\mathfrak{m}$ is $R(T')=\{\alpha\in \Delta_{+}; |\alpha|\leq l\}$.

We remark that $\mathfrak{m}$ can be constructed as a homomorphism image of the free Lie algebra generated by $\{e_{1},\cdots,e_{n}\}$  and
$\dim \overline{L_{\alpha}}$ may be greater than one if $|\alpha|\geq 3$.

Let $\mathfrak{J}=\mathfrak{J}_{l}(C_{E})=\{\mathfrak{a}: \text{homogeneous ideal of }\ \mathfrak{m}; C^{l}\mathfrak{m}\nsubseteq \mathfrak{a},
(\ad \overline{e_{i}})^{-c_{ij}}\overline{e_{j}}\not\in \mathfrak{a}\}$.
\begin{proposition}\rm{(\cite{Sa})}\label{nil-ideal}
There exists $\mathfrak{a}\in\mathfrak{J}$ such that $\tilde{\mathfrak{n}}^{+}=\mathfrak{m}/\mathfrak{a}$.
\end{proposition}
Since $\mathfrak{a}$ is homogeneous, $\mathfrak{a}=\oplus_{\alpha\in\Delta_{+},|\alpha|\leq l}\mathfrak{a}_{\alpha}$, where $\mathfrak{a}_{\alpha}=\mathfrak{a}\cap\overline{L_{\alpha}}$.
Let $S=\{\alpha|0\neq \mathfrak{a}_{\alpha}=\overline{L_{\alpha}}\}$, then $\alpha\in S$ is called a root of  $\mathfrak{m}$ \emph{killed by} $\mathfrak{a}$. Let $\pi: \mathfrak{m}\rightarrow \tilde{\mathfrak{n}}^{+}$ be the canonical map by sending $\overline{e_{i}}$ to $C_{P_{i}}$.
For any $t\in T'$ one can define $\tilde{\pi}(t)\in \Der \tilde{\mathfrak{n}}^{+}$ such that $\pi\circ t=\tilde{\pi}(t)\circ \pi$.
Then $\tilde{\pi}(T')$ is a maximal torus on $\tilde{\mathfrak{n}}^{+}$. By $\pi\circ t'_{i}=t_{i}\circ\pi$,  we can identify $\tilde{\pi}(T')$
and $T$. Thus $\alpha'_{i}\in T^{*}$ is identified with $\alpha_{i}$. As a conclusion,
\begin{theorem}
The nilpotent Lie algebra $\tilde{\mathfrak{n}}^{+}$ is of maximal rank with $T$ as a maximal torus. $(C_{P_{1}},\cdots,C_{P_{n}})$ is
a $T$-msg with roots $(\alpha_{1},\cdots,\alpha_{n})$. The root space decomposition relative to $T$ is identical to the decomposition
$$\tilde{\mathfrak{n}}^{+}=\oplus_{\alpha\in R(T)}\tilde{\mathfrak{n}}^{+}_{\alpha},$$
where $R(T)=\{\alpha\in \Delta_{+}; |\alpha|\leq l, \alpha\not\in S\}=\{\alpha_{M}|M\in \mathrm{Ind\ } kQ\}$ is the associated root system and $\dim \tilde{\mathfrak{n}}^{+}_{\alpha}=1$ for all $\alpha\in R(T)$. Moreover, for any $\alpha\in R(T)-\{\alpha_{1},\cdots,\alpha_{n}\}$, there exists $i\in \{1,\cdots,n\}$ such that $\alpha-\alpha_{i}\in R(T)$.
\end{theorem}
\begin{example}
If $Q$ is bipartite, then $\tilde{\mathfrak{n}}^{+}=L_{+}(C_{E})$, where $C_{E}=C$ is the Cartan matrix associated to the underlying graph $\Gamma$ of $Q$. The root system is just $\Delta_{+}$.
\end{example}
\begin{example}
If $Q=\mathbb{A}_{n}^{\rightarrow}$, $\tilde{\mathfrak{n}}^{+}=\mathfrak{m}_{2}(C_{E})$. The root system is
$$\{\alpha\in \Delta_{+};|\alpha|\leq 2\}=\{\alpha_{1},\cdots,\alpha_{n}\}\cup\{\alpha_{i}+\alpha_{j}|j<i\}.$$
\end{example}
\begin{example}
If $Q$ is of type $\mathbb{A}$, then without loss of generality, we assume $Q$ is as follows:
$$\xymatrix@=0.8pc{1\ar[r]& \cdot\ar@{.>}[r]& \cdot\ar[r] & n_{1}+1 & \cdot\ar[l]& \cdot\ar@{.>}[l]& \sum^{2}_{i=1}n_{i}+1\ar[l]\ar@{.}[rr]& &\cdot\ar[r]& \cdot\ar@{.>}[r]& \cdot\ar[r]&\sum^{t}_{i=1}n_{i}+1.  \\}$$
Here $Q$ has $\sum^{t}_{i=1}n_{i}+1$ vertices and $t$ paths of maximal length. We call a vertex of $Q$ a \emph{knot} if it is the source or the target of exactly two arrows. Thus a knot must be a source or a sink.
Note that $t$ is an odd number and $Q$ has $t-1$ knots. The total number of sinks and sources is $t+1$, which is also the maximal height of roots. The root system is as follows ($n_{0}=0$ by convention): the simple roots are $\alpha_{i}$, $1\leq i\leq \sum^{t}_{i=1}n_{i}+1$; the roots of height $2$ are $\alpha_{i}+\alpha_{j}$, $\sum^{s}_{r=0}n_{r}+1\leq i<j\leq \sum^{s+1}_{r=0}n_{r}+1$, $0\leq s\leq t-1$; the roots of height $m$ ($3\leq m\leq t+1$) are $$\alpha_{i}+\sum^{m+s-3}_{p=s} \alpha_{\sum^{p}_{r=0}n_{r}+1}+ \alpha_{j},$$ where $\sum^{s-1}_{r=0}n_{r}+1\leq i<\sum^{s}_{r=0}n_{r}+1\leq \sum^{m+s-3}_{r=0}n_{r}+1<j\leq \sum^{m+s-2}_{r=0}n_{r}+1$, $1\leq s\leq t+2-m$.
\end{example}
\begin{example}
If $Q$ is of type $\mathbb{D}_{n}$ and oriented as follows with $n\geq 5$:
$$\xymatrix{ & 1\ar[d]  \\
2 & 3\ar[l]& 4\ar[l]& \cdot\ar[l]& \cdot \ar@{.>}[l] & n.\ar[l]}$$
The root system is as follows: the simple roots are $\alpha_{i}$, $1\leq i\leq n$; the roots of height $2$ are $\alpha_{1}+\alpha_{2}$, $\alpha_{1}+\alpha_{3}$ and $\alpha_{i}+\alpha_{j}$, $2\leq i<j\leq n$; the roots of height $3$ are $\alpha_{1}+\alpha_{t}+\alpha_{m}$, $t=2, 3$,
$4\leq m\leq n$; the roots of height $4$ are $\alpha_{1}+\alpha_{2}+\alpha_{3}+\alpha_{m}$, $4\leq m\leq n$; the roots of height $5$ are
$\alpha_{1}+\alpha_{2}+\alpha_{3}+\alpha_{m}+\alpha_{t}$, $4\leq m<t\leq n$.
\end{example}
\begin{example}
If $Q$ is of type $\mathbb{E}_{6}$ and oriented as follows:
$$
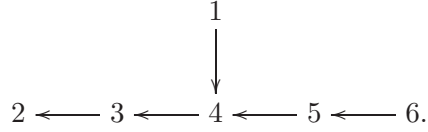
\xymatrix{ & & 1\ar[d]  \\
2 & 3\ar[l]& 4\ar[l]& 5 \ar[l]& 6. \ar[l]}$$
The root system is as follows: the simple roots are $\alpha_{i}$, $1\leq i\leq 6$; the roots of height $2$ are $\alpha_{1}+\alpha_{t}$, $t=2,3,4$ and $\alpha_{i}+\alpha_{j}$, $2\leq i<j\leq 6$;  the roots of height $3$ are $\alpha_{1}+\alpha_{m}+\alpha_{t}$, $m=2,3,4$, $t=5,6$;
the roots of height $4$ are $\alpha_{1}+\alpha_{i}+\alpha_{j}+\alpha_{t}$, $2\leq i<j\leq 4$, $t=5,6$; the roots of height $5$ are
$\alpha_{1}+\alpha_{i}+\alpha_{j}+\alpha_{5}+\alpha_{6}$, $2\leq i<j\leq 4$; the remaining roots are $\Sigma^{6}_{i=1}\alpha_{i}$ of height $6$ and
$2\alpha_{1}+\Sigma^{6}_{i=2}\alpha_{i}$ of height $7$.
\end{example}

\subsection{Comparing two kinds of Hall Lie algebras}
Let $\tilde{\mathfrak{n}}^{+}$ be the Hall Lie algebra of the Dynkin quiver $Q$ as above. Let $Q^{p}$ be the \textit{path quiver} of $Q$ defined as follows: $Q^{p}$ has the same vertex set as $Q$, while each arrow $i\rightarrow j$ of $Q^{p}$ corresponds to exactly a path from $i$ to $j$ in $Q$.

Thus $Q^{p}$ is an acyclic quiver. Obviously $Q^{p}=Q$ if and only if $Q$ is bipartite. We define the classical Hall Lie algebra of $kQ^{p}$-modules as in \cite{R91a} and \cite{PX2000}. Let $S^{p}_{1},\cdots, S^{p}_{n}$ be all non-isomorphic simple $kQ^{p}$-modules. Let $\mathfrak{LC}(Q^{p})$ be the (generic) Hall composition Lie algebra generated by $u_{S^{p}_{1}},\cdots, u_{S^{p}_{n}}$, with the Lie bracket induced by the degenerate Hall multiplication. It is well known that $\mathfrak{LC}(Q^{p})$ gives a realization of the positive part of the corresponding Kac--Moody Lie algebra. Precisely,
\begin{theorem}\rm{(Ringel--Green, cf. \cite{PX2000})}
Let $E$ and $Q^{p}$ be respectively the path matrix and path quiver of the Dynkin quiver $Q$, then we have the isomorphism of complex Lie algebras:
$L_{+}(C_{E})\cong \mathfrak{LC}(Q^{p})$, $e_{i}\mapsto u_{S^{p}_{i}}$.
\end{theorem}

Assume that $\tilde{\mathfrak{n}}^{+}$ is of nilpotency $l$ and let $\mathfrak{m}'=\mathfrak{LC}(Q^{p})/C^{l+1}\mathfrak{LC}(Q^{p})$. Of course we
have an isomorphism of Lie algebras $\mathfrak{m}\cong \mathfrak{m}'$, $\overline{e_{i}}\mapsto \overline{u_{S^{p}_{i}}}$, where $\overline{u_{S^{p}_{i}}}$ is the image of $u_{S^{p}_{i}}$ under the natural epimorphism $\mathfrak{LC}(Q^{p})\twoheadrightarrow \mathfrak{LC}(Q^{p})/C^{l+1}\mathfrak{LC}(Q^{p})$.

Then by Proposition \ref{nil-ideal}, we relate the Hall Lie algebras of 1-cyclic perfect complexes to classical Hall Lie algebras as follows.
\begin{proposition}
There is an epimorphism of Lie algebras $\varphi: \mathfrak{m}'\twoheadrightarrow \tilde{\mathfrak{n}}^{+}$, $\overline{u_{S^{p}_{i}}}\mapsto C_{P_{i}}$ such that $\Ker \varphi$ is a homogeneous ideal of $\mathfrak{m}'$,  $C^{l}\mathfrak{m}'\nsubseteq \Ker \varphi$ and
$(\ad \overline{u_{S^{p}_{i}}})^{-c_{ij}}\overline{u_{S^{p}_{j}}}\not\in \Ker \varphi$.

\end{proposition}
\begin{example}
Let $Q$ be the quiver $$\xymatrix{ & 1  \\
2\ar[r] & 3\ar[r]\ar[u]& 4.}$$
The path matrix $E$ is $\begin{pmatrix}
0 & -1 & -1 & 0\\
1 & 0 & 1 &  1\\
1 & -1 & 0 & 1 \\
0 & -1 & -1 & 0
\end{pmatrix}$ and the path quiver $Q^{p}$ is
$
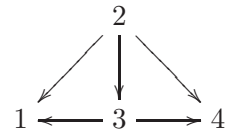
\xymatrix{ & 2\ar[ld]\ar[rd]\ar[d]  \\
1 &\ar[l] 3\ar[r]& 4}.$

Since $\tilde{\mathfrak{n}}^{+}(Q)$ is of nilpotency $4$, let $\mathfrak{m}'=\mathfrak{LC}(Q^{p})/C^{5}\mathfrak{LC}(Q^{p})$.
We have epimorphisms of Lie algebras $\mu: \mathfrak{LC}(Q^{p})\twoheadrightarrow \mathfrak{m}'$ and $\varphi: \mathfrak{m}'\twoheadrightarrow \tilde{\mathfrak{n}}^{+}(Q)$. So there is an epimorphism from the classical Hall Lie algebra $\mathfrak{LC}(Q^{p})$ to $\tilde{\mathfrak{n}}^{+}(Q)$
and the element $u_{S^{p}_{i}}$ is mapped to $C_{P_{i}}$.

\end{example}

\subsection{Defining relations of $\tilde{\mathfrak{n}}^{+}$}
\begin{definition}\label{Lie relation}
Let $Q$ be a Dynkin quiver and  $E=(a_{ij})$ its path matrix.
Let $\tilde{\mathfrak{N}}$ be the Lie algebra defined by generators $\{e_{i}~|~1\leq i\leq n\}$ and relations:

(a) If $|a_{ij}|=1$, $(\ad e_{i})^2(e_{j})=(\ad e_{j})^2(e_{i})=0$;

(b) If $a_{ij}a_{jk}=1$, $[e_{i},[e_{j},e_{k}]]=[e_{k},[e_{i},e_{j}]]=0$;

(c) If $a_{ij}=0$, $[e_{i},e_{j}]=0$.
\end{definition}
Thus we have an epimorphism $\psi: \tilde{\mathfrak{N}}\twoheadrightarrow \tilde{\mathfrak{n}}^{+}$, $e_{i}\mapsto C_{P_{i}}$.
By comparing dimensions, we already know that $\psi$ is an isomorphism when $Q$ is bipartite or $Q=\mathbb{A}_{n}^{\rightarrow}$.

For a quiver $Q$ of Dynkin type, we call a \emph{branch} of $Q$ a path of maximal length. Thus a branch is a subquiver $\mathbb{A}_{m}^{\rightarrow}$
of $Q$ with $m\geq 2$.

\begin{proposition}
If $Q'=\mathbb{A}_{m}^{\rightarrow}$ is a subquiver of $Q$, then $\tilde{\mathfrak{n}}^{+}(Q')$ is a Lie subalgebra of $\tilde{\mathfrak{n}}^{+}(Q)$.
\end{proposition}
\bp
For any vertex $i$ of $Q'$, we have an indecomposable projective $kQ'$-module $P'_{i}$ and an indecomposable projective $kQ$-module $P_{i}$.
Consider the morphism $\eta: \tilde{\mathfrak{n}}^{+}(Q')\rightarrow \tilde{\mathfrak{n}}^{+}(Q)$ defined by $\eta(C_{P'_{i}})=C_{P_{i}}$.

Renumber the vertices of $Q'$ as $1,\cdots, m$, then $Q'$ is as follows:
$$\xymatrix{1\ar[r]& 2\ar[r]& \cdot\ar@{.>}[r] & \cdot\ar[r] & m.\\}$$
For any $i<j$, we get a short exact sequence in $\mod kQ'$:
$$0\longrightarrow P'_{j}\longrightarrow P'_{i}\longrightarrow M'_{ij}\longrightarrow 0,$$
where $M'_{ij}$ is an indecomposable $kQ'$-module. Actually it is a minimal projective resolution of $M'_{ij}$, and  $[C_{P'_{j}},C_{P'_{i}}]=C_{M'_{ij}}$ in $\tilde{\mathfrak{n}}^{+}(Q')$.

Similarly, we have a short exact sequence in $\mod kQ$:
$$0\longrightarrow P_{j}\longrightarrow P_{i}\longrightarrow M_{ij}\longrightarrow 0,$$
where $M_{ij}$ is an indecomposable $kQ$-module. Thus $[C_{P_{j}},C_{P_{i}}]=C_{M_{ij}}$ in $\tilde{\mathfrak{n}}^{+}(Q)$.

By setting $\eta(C_{M'_{ij}})=C_{M_{ij}}$ for all $1\leq i<j\leq m$, we conclude that $\eta$ is a monomorphism. Note that $\eta$
is a homomorphism by Theorem \ref{main result}, and is injective by comparing bases of two Lie algebras.
\ep
Now consider $\eta$ as a natural inclusion. If $Q$ has $t$ branches, namely $\mathbb{A}_{r_{1}}^{\rightarrow}, \cdots, \mathbb{A}_{r_{t}}^{\rightarrow}$, $r_{j}\geq 2, 1\leq j\leq t$, there are exactly $t$ largest free 2-step nilpotent Lie algebras $\mathfrak{N}(r_{j})$, $1\leq j\leq t$ sitting inside $\tilde{\mathfrak{n}}^{+}(Q)$.
Here a free 2-step nilpotent Lie subalgebra is called the ``largest'' if it is not included in a larger free 2-step nilpotent Lie subalgebra inside $\tilde{\mathfrak{n}}^{+}(Q)$.

Generally, it is not easy to determine whether $\psi$ is an isomorphism. We firstly tackle some easy examples.

Since $\tilde{\mathfrak{n}}^{+}(Q)$ has a root space decomposition, it is also graded with heights of roots. Assume $\tilde{\mathfrak{n}}^{+}(Q)$ is of nilpotency $l$, then
 $$\tilde{\mathfrak{n}}^{+}(Q)=\oplus^{l}_{j=1}\tilde{\mathfrak{n}}^{+}(Q)_{j},$$
 where $\tilde{\mathfrak{n}}^{+}(Q)_{j}=\oplus_{\alpha\in R(T), |\alpha|=j}\tilde{\mathfrak{n}}^{+}(Q)_{\alpha}$.
For  $\tilde{\mathfrak{N}}$, we say  $[e_{i_{1}},\cdots,e_{i_{t}}]\neq 0$ is of height $t$. Then $\tilde{\mathfrak{N}}(Q)$ is also graded with heights.
Then we compare the dimensions of $\tilde{\mathfrak{n}}^{+}(Q)_{j}$ and $\tilde{\mathfrak{N}}(Q)_{j}$ for each $j$.
\begin{example}\label{a4}
Let $Q$ be $1\rightarrow 2 \rightarrow 3 \leftarrow 4$. We have $\dim \tilde{\mathfrak{n}}^{+}_{1}=4$, $\dim \tilde{\mathfrak{n}}^{+}_{2}=4$
and $\dim \tilde{\mathfrak{n}}^{+}_{3}=2$. It is easy that $\tilde{\mathfrak{N}}_{1}=\tilde{\mathfrak{N}}_{2}=4$.
For $\tilde{\mathfrak{N}}_{3}$, due to the defining relations and Jacobi identity, there are two linearly independent elements $[[e_{2},e_{3}],e_{4}]$ and $[[e_{1},e_{3}],e_{4}]$. Note that $\tilde{\mathfrak{N}}_{4}=0$ because of defining relations. Obviously $\tilde{\mathfrak{N}}_{j}=0$ for $j\geq 5$. So $\psi$ is an isomorphism.
\end{example}
\begin{example}
Assume that $Q$ is $$\xymatrix{ & 1  \\
2\ar[r] & 3\ar[r]\ar[u]& 4.}$$
Then $\dim \tilde{\mathfrak{n}}^{+}_{1}=4$, $\dim \tilde{\mathfrak{n}}^{+}_{2}=5$, $\dim \tilde{\mathfrak{n}}^{+}_{3}=2$ and
$\dim \tilde{\mathfrak{n}}^{+}_{4}=1$. For $\tilde{\mathfrak{N}}_{3}$, one can find two linearly independent elements $[[e_{1},e_{2}],e_{4}]$ and $[[e_{1},e_{3}],e_{4}]$. For $\tilde{\mathfrak{N}}_{4}$, the only basis element is $[[[e_{1},e_{3}],e_{4}],e_{2}]$. $\tilde{\mathfrak{N}}_{j}=0$ for $j\geq 5$. So $\psi$ is an isomorphism.
\end{example}
The rest of this subsection is devoted to the proof of the following theorem.
\begin{theorem}\label{Atype}
Let $Q$ be a quiver of type $\mathbb{A}_{n}$, then $\psi: \tilde{\mathfrak{N}}\rightarrow \tilde{\mathfrak{n}}^{+}$, $e_{i}\mapsto C_{P_{i}}$ is an isomorphism.
\end{theorem}
Let $Q$ be a quiver whose underlying graph is as follows:
$$\xymatrix{1\ar@{-}[r]& 2\ar@{-}[r]& \cdot\ar@{.}[r] & \cdot\ar@{-}[r] & n.\\}$$

Thus the vertices of $Q$ are given a natural ordering: $1<2<\cdots <n$. We still use the symbols as in Definition \ref{Lie relation}.
Let $\mathfrak{F}$ be the free Lie algebra over $\mathbb{C}$ generated by $e_{i}$, $1\leq i\leq n$. Then $\tilde{\mathfrak{N}}$ is the quotient algebra of $\mathfrak{F}$
modulo a homogeneous ideal, which is determined by $Q$.

Recall that $[e_{i_{1}},e_{i_{2}},\cdots,e_{i_{t}}]:=[[[\cdots[e_{i_{1}},e_{i_{2}}],\cdots],e_{i_{t-1}}],e_{i_{t}}]$, which will be called a \emph{left normed word}. Similarly, we introduce $\lfloor e_{i_{1}},e_{i_{2}},\cdots,e_{i_{t}}\rfloor :=[e_{i_{1}},[e_{i_{2}},[\cdots,[e_{i_{t-1}},e_{i_{t}}]\cdots]]]$, which is called a \emph{right normed word}  in \cite{Chi}.

The following lemma is well known (cf. \cite{Chi}).
\begin{lemma}\label{span}
$[e_{i_{1}},e_{i_{2}},\cdots,e_{i_{t}}]$, $1\leq i_{j}\leq n$, $1\leq j\leq t$, $t\in \mathbb{N}$ generate the linear space $\mathfrak{F}$.
So they also generate the linear space $\tilde{\mathfrak{N}}$.
\end{lemma}
\begin{lemma}\label{word1}
If $1\leq i_{1}\leq i_{2}\leq i_{3}\leq n$, then $[e_{i_{1}},e_{i_{2}},e_{i_{3}}]=\lfloor e_{i_{1}},e_{i_{2}},e_{i_{3}}\rfloor$.
\end{lemma}
\bp
By Jacobi identity, we have
$$[e_{i_{1}},e_{i_{2}},e_{i_{3}}]=[[e_{i_{1}},e_{i_{3}}],e_{i_{2}}]+[e_{i_{1}},[e_{i_{2}},e_{i_{3}}]].$$
Consider the path matrix $E_{Q}=(a_{ij})$ of the quiver $Q$. If $i_{1}=i_{3}$, then $i_{1}=i_{2}=i_{3}$ and we are done. If $i_{1}\neq i_{3}$ and $a_{i_{1}i_{3}}= 0$, then $[[e_{i_{1}},e_{i_{3}}],e_{i_{2}}]=0$ and we are done. If $a_{i_{1}i_{3}}\neq 0$,
there is a path between $i_{1}$ and $i_{3}$. Assume there is a path from  $i_{1}$ to $i_{3}$, then we have
$i_{1}\rightarrow i_{2}\rightarrow i_{3}$, $i_{1}=i_{2}\rightarrow i_{3}$ or $i_{1}\rightarrow i_{2}= i_{3}$, where $\rightarrow$ stands for a path in the quiver $Q$. For all cases, it is easy to see that both terms in the lemma are equal to zero.
\ep
We call $[e_{i_{1}},e_{i_{2}},\cdots, e_{i_{m}}]$ (resp. $\lfloor e_{i_{1}},e_{i_{2}},\cdots, e_{i_{m}}\rfloor$) a left normed (resp. right normed) word \emph{of length} $m$.
\begin{lemma}\label{word2}
If $1\leq i_{1}\leq i_{2}\leq \cdots\leq i_{m}\leq n$, $m\geq 3$, then $[e_{i_{1}},e_{i_{2}},\cdots, e_{i_{m}}]=\lfloor e_{i_{1}},e_{i_{2}},\cdots, e_{i_{m}}\rfloor$.
\end{lemma}
\bp
We proceed by induction on the length $m$. For $m=3$, it is just Lemma \ref{word1}. Assume the lemma holds for all words of length less or equal to $m-1$, then
\begin{eqnarray*}\lfloor e_{i_{1}},e_{i_{2}},\cdots,e_{i_{m-1}},e_{i_{m}}\rfloor &=& [e_{i_{1}},\lfloor e_{i_{2}},\cdots,e_{i_{m-1}},e_{i_{m}}\rfloor]=[e_{i_{1}},[ e_{i_{2}},\cdots,e_{i_{m-1}},e_{i_{m}}]] \\
&=& [[e_{i_{1}},[e_{i_{2}},\cdots, e_{i_{m-1}}]],e_{i_{m}}]+[[ e_{i_{2}},\cdots, e_{i_{m-1}}],[e_{i_{1}},e_{i_{m}}]] \\
&=& [e_{i_{1}},e_{i_{2}},\cdots,e_{i_{m-1}},e_{i_{m}}] + [[ e_{i_{2}},\cdots, e_{i_{m-1}}],[e_{i_{1}},e_{i_{m}}]].
\end{eqnarray*}
If $i_{1}=i_{m}$, then we are done. If $i_{1}\neq i_{m}$ and $a_{i_{1}i_{m}}= 0$, then $[[ e_{i_{2}},\cdots, e_{i_{m-1}}],[e_{i_{1}},e_{i_{m}}]]=0$ and we are done. If $a_{i_{1}i_{m}}\neq 0$, there is a path between $i_{i}$ and $i_{m}$. Obviously, there is a path between $i_{1}$ and $i_{3}$ (resp. $i_{m-2}$ and $i_{m}$) or $i_{1}=i_{3}$ (resp. $i_{m-2}=i_{m}$). By the proof of Lemma \ref{word1}, we have
$[e_{i_{1}},e_{i_{2}},e_{i_{3}}]=0$ and $\lfloor e_{i_{m-2}},e_{i_{m-1}},e_{i_{m}}\rfloor =0$. So both terms in the lemma are equal to zero.
\ep
\begin{lemma}\label{word3}
If $1\leq i_{1}\leq i_{2}\leq \cdots\leq i_{m}\leq n$ , $m\geq 3$ and $1\leq r\leq m-1$, then
$$[e_{i_{1}},e_{i_{2}},\cdots, e_{i_{m}}]=[[e_{i_{1}},\cdots, e_{i_{r}}],[e_{i_{r+1}},\cdots, e_{i_{m}}]].$$
\end{lemma}
\bp
For convenience, we write $L=[e_{i_{1}},e_{i_{2}},\cdots, e_{i_{m}}]$ and $R=[[e_{i_{1}},\cdots, e_{i_{r}}],[e_{i_{r+1}},\cdots, e_{i_{m}}]]$.
Firstly we proceed by induction on the length $m$ of the left normed word $L$. For $m=3$, it is direct from the discussion above. Assume that
the lemma holds for length less or equal to $m-1$. Let $L$ be of length $m$. Then we proceed by induction on $r$. If $r=1$, then the lemma is correct by Lemma \ref{word2}. Now assume the equality holds for $r-1$. Then
\begin{eqnarray*}R&=&[[[e_{i_{1}},\cdots,e_{i_{r-1}}],e_{i_{r}}],[e_{i_{r+1}},\cdots, e_{i_{m}}]]\\
&=& [[e_{i_{1}},\cdots,e_{i_{r-1}}],[e_{i_{r}},e_{i_{r+1}},\cdots, e_{i_{m}}]]+[[e_{i_{1}},\cdots,e_{i_{r-1}},e_{i_{r+1}},\cdots, e_{i_{m}}],e_{i_{r}}]\\
&=& L+[[e_{i_{1}},\cdots,e_{i_{r-1}},e_{i_{r+1}},\cdots, e_{i_{m}}],e_{i_{r}}].
\end{eqnarray*}
If $[e_{i_{1}},\cdots,e_{i_{r-1}},e_{i_{r+1}},\cdots, e_{i_{m}}]=0$, we have nothing to prove. If $[e_{i_{1}},\cdots,e_{i_{r-1}},e_{i_{r+1}},\cdots, e_{i_{m}}]\neq 0$, by Lemma \ref{word2}, we have $[e_{i_{r-1}},e_{i_{r+1}},e_{i_{r+2}}]\neq 0$. There are two possibilities, $i_{r-1}\rightarrow i_{r+1}\leftarrow i_{r+2}$
or $i_{r-1}\leftarrow i_{r+1}\rightarrow i_{r+2}$, where the arrows stand for paths in $Q$. Immediately, we get $[e_{i_{r-1}},e_{i_{r}},e_{i_{r+1}}]= 0$ and so $L=0$. If $i_{r}\neq i_{r+1}$, then $[e_{i_{r}},e_{i_{r+2}}]=0$. Thus
\begin{eqnarray*}R &=&[[e_{i_{1}},\cdots,e_{i_{r}}],[e_{i_{r+1}},[e_{i_{r+2}},\cdots, e_{i_{m}}]]]\\
&=&[[e_{i_{1}},\cdots,e_{i_{r}},e_{i_{r+1}}],[e_{i_{r+2}},\cdots, e_{i_{m}}]]+[e_{i_{r+1}},[e_{i_{1}},\cdots,e_{i_{r}},e_{i_{r+2}},\cdots, e_{i_{m}}]]\\
&=& 0+0=0.
\end{eqnarray*}
If $i_{r}=i_{r+1}$, we have $R=[[a,b],[b,c]]$, where $a=[e_{i_{1}},\cdots,e_{i_{r-1}}]$, $b=e_{i_{r}}$ and $c=[e_{i_{r+2}},\cdots,e_{i_{m}}]$.
Then we deduce both $R=[[a,[b,c]],b]+[a,[b,[b,c]]]=[[a,[b,c]],b]=[[[a,b],c],b]$ and $R=[[[a,b],b],c]+[b,[[a,b],c]]=[b,[[a,b],c]]=-[[[a,b],c],b]$.
So $R=0$.
\ep
\begin{lemma}\label{word4}
If  $1\leq i_{1}\leq i_{2}\leq \cdots\leq i_{m}\leq n$, $m\geq 3$, $2\leq r\leq m-1$, then $$[[e_{i_{1}},\cdots,e_{i_{r-1}},e_{i_{r+1}},\cdots, e_{i_{m}}],e_{i_{r}}]=0.$$
\end{lemma}
\bp
\begin{eqnarray*}& & [[e_{i_{1}},\cdots,e_{i_{r-1}},e_{i_{r+1}},\cdots, e_{i_{m}}],e_{i_{r}}]=[[[e_{i_{1}},\cdots,e_{i_{r-1}}],[e_{i_{r+1}},\cdots, e_{i_{m}}]],e_{i_{r}}] \\
&=& [[e_{i_{1}},\cdots,e_{i_{r-1}}],[[e_{i_{r+1}},\cdots, e_{i_{m}}],e_{i_{r}}]]+[[[e_{i_{1}},\cdots,e_{i_{r-1}}],e_{i_{r}}],[e_{i_{r+1}},\cdots, e_{i_{m}}]]\\
&=& -[[e_{i_{1}},\cdots,e_{i_{r-1}}],[e_{i_{r}},e_{i_{r+1}},\cdots, e_{i_{m}}]]+[[e_{i_{1}},\cdots,e_{i_{r-1}},e_{i_{r}}],[e_{i_{r+1}},\cdots, e_{i_{m}}]]\\
&=& -[e_{i_{1}},\cdots,e_{i_{m}}]+[e_{i_{1}},\cdots,e_{i_{m}}]=0.
\end{eqnarray*}
\ep
We call $\omega=[e_{i_{1}},e_{i_{2}},\cdots,e_{i_{m}}]$ a\emph{ standard left normed word} when it satisfies $1\leq i_{1}<i_{2}<\cdots < i_{m}\leq n$, and associate an $\mathbb{A}_{m}$-type quiver $Q_{\omega}$ to it as follows: the set of vertices is just $\{i_{1},i_{2},\cdots,i_{m}\}$; for each pair $(i_{j},i_{j+1})$, $1\leq j\leq m-1$, if there is a path between $i_{j}$ and $i_{j+1}$, we draw an arrow between them, with its direction indicating the direction of the path.
\begin{proposition}\label{span1}
The linear space $\tilde{\mathfrak{N}}$ is spanned by standard left normed words whose associated quivers are connected and bipartite.
\end{proposition}
\bp
For each left normed word $\theta=[e_{j_{1}},e_{j_{2}},\cdots,e_{j_{r}}]$, we can give a reordering $(i_{1},i_{2},\cdots,i_{r})$ of $(j_{1},j_{2},\cdots,j_{r})$ such that $i_{1}\leq i_{2}\leq\cdots\leq i_{r}$. Now we claim that, if $\theta\neq 0$, then $\theta=[e_{i_{1}},e_{i_{2}},\cdots,e_{i_{r}}]$ or $\theta=-[e_{i_{1}},e_{i_{2}},\cdots,e_{i_{r}}]$.

We prove the claim by induction on the length of $\theta$. When $r=2$, it is trivial. Assume the claim holds for length less or equal to $r$.
Consider $\theta'=[e_{j_{1}},e_{j_{2}},\cdots,e_{j_{r}},e_{j_{r+1}}]=[\theta, e_{j_{r+1}}]$. Since $(i_{1},i_{2},\cdots,i_{r})$ is a reordering
of $(j_{1},j_{2},\cdots,j_{r})$, there are three possibilities for $j_{r+1}$. Firstly, $j_{r+1}\geq i_{r}$. The reordering of $(j_{1},j_{2},\cdots,j_{r},j_{r+1})$ is $(i'_{1},i'_{2},\cdots,i'_{r},i'_{r+1})$ with $i'_{s}=i_{s}, 1\leq s\leq r$ and $i'_{r+1}=j_{r+1}$. If $\theta'\neq 0$, then $\theta\neq 0$. Thus $\theta=\pm[e_{i_{1}},e_{i_{2}},\cdots,e_{i_{r}}]$. So it is easy that $\theta'=[\theta, e_{j_{r+1}}]=\pm[e_{i'_{1}},e_{i'_{2}},\cdots,e_{i'_{r+1}}]$. Secondly, $j_{r+1}\leq i_{1}$. Then the reordering of $(j_{1},j_{2},\cdots,j_{r},j_{r+1})$ is $(i'_{1},i'_{2},\cdots,i'_{r},i'_{r+1})$ with $i'_{1}=j_{r+1}$ and $i'_{s}=i_{s-1}, 2\leq s\leq r+1$.  Again, if $\theta'\neq 0$, then $\theta'=[\theta, e_{j_{r+1}}]=-[e_{i'_{1}},\theta]=\pm[e_{i'_{1}},e_{i'_{2}},\cdots,e_{i'_{r+1}}]$. Lastly, $i_{1}<j_{r+1}< i_{r}$. Then the reordering of $(j_{1},j_{2},\cdots,j_{r},j_{r+1})$ is $(i'_{1},i'_{2},\cdots,i'_{r},i'_{r+1})$ with $j_{r+1}=i'_{p}$ satisfying $2\leq p\leq r$. Thus by Lemma \ref{word4}, $\theta'=[\theta, e_{j_{r+1}}]=\pm [[e_{i_{1}},e_{i_{2}},\cdots,e_{i_{r}}],e_{j_{r+1}}]=\pm [[e_{i'_{1}},\cdots,e_{i'_{p-1}},e_{i'_{p+1}},\cdots,e_{i'_{r+1}}],e_{i'_{p}}]=0$.

Note that for each left normed word $\theta=[e_{i_{1}},e_{i_{2}},\cdots,e_{i_{r}}]$ such that $i_{1}\leq i_{2}\leq\cdots\leq i_{r}$, if there exists an $l$ such that $i_{l}=i_{l+1}$, then $\theta=0$ by Lemma \ref{word2}. According to Lemma \ref{span}, the claim above amounts to saying that the linear space $\tilde{\mathfrak{N}}$ is spanned by standard left normed words. For each standard left normed word $\omega=[e_{i_{1}},e_{i_{2}},\cdots,e_{i_{m}}]$,
if $a_{i_{l}i_{l+1}}=0$ or $a_{i_{l-1}i_{l}}a_{i_{l}i_{l+1}}=1$, then $\omega=0$ by Definition \ref{Lie relation} and Lemma \ref{word2}. So the quiver $Q_{\omega}$ associated to any nonzero standard left normed word $\omega$ must be connected and bipartite.
\ep
\begin{lemma}\label{span2}
The number of standard left normed words whose associated quivers are connected and bipartite is $\frac{n(n+1)}{2}$.
\end{lemma}
\bp
Without loss of generality, we assume $Q$ is as follows:
$$\xymatrix@=0.8pc{1\ar[r]& \cdot\ar@{.>}[r]& \cdot\ar[r] & n_{1}+1 & \cdot\ar[l]& \cdot\ar@{.>}[l]& \sum^{2}_{i=1}n_{i}+1\ar[l]\ar@{.}[rr]& &\cdot\ar[r]& \cdot\ar@{.>}[r]& \cdot\ar[r]&\sum^{t}_{i=1}n_{i}+1.  \\}$$
Here $Q$ has $n=\sum^{t}_{i=1}n_{i}+1$ vertices and $t$ paths of maximal length. Let $W$ be the set of standard left normed words whose associated quivers are connected and bipartite. Let $W_{r}$ be the subset of $W$ whose words are of length $r$. Obviously $W=\dot{\cup}^{t+1}_{r=1}W_{r}$.

By calculation, $|W_{1}|=n$, $|W_{2}|=\sum^{t}_{k=1}\frac{n_{k}(n_{k}+1)}{2}$, $W_{r}=\sum^{t-r+2}_{k=1}n_{k}n_{r+k-2}$ for $3\leq r\leq t+1$. Therefore,
\begin{eqnarray*}|W|&=&n+\sum^{t}_{k=1}\frac{n_{k}(n_{k}+1)}{2}+\sum^{t+1}_{r=3}\sum^{t-r+2}_{k=1}n_{k}n_{r+k-2}\\
&=& n+\sum^{t}_{k=1}\frac{n_{k}}{2}+\frac{1}{2}\sum^{t}_{k=1}n_{k}^{2}+\sum^{t-1}_{p=1}\sum^{t-p}_{k=1}n_{k}n_{k+p}\\
&=& \frac{3n-1}{2}+\frac{1}{2}(\sum^{t}_{k=1}n_{k})^{2}=\frac{n(n+1)}{2}.
\end{eqnarray*}
\ep
Now we finish the proof of Theorem \ref{Atype}. Since $\psi: \tilde{\mathfrak{N}}\rightarrow \tilde{\mathfrak{n}}^{+}$ is an epimorphism and
$\dim \tilde{\mathfrak{n}}^{+}= \frac{n(n+1)}{2}$, by combing Proposition \ref{span1} and Lemma \ref{span2}, we conclude that $\psi$ is an isomorphism.
Recall Proposition \ref{nil-ideal}, then for any $\mathbb{A}_{n}$-type quiver $Q$, we deduce that
$$\mathfrak{a}=\{[e_{i},[e_{j},e_{k}]], [e_{k},[e_{i},e_{j}]]|a_{ij}a_{jk}=1\}.$$

\subsection{Connection with complete Lie algebras}
A Lie algebra is called \emph{complete} if its center is zero and all its derivations are inner. The corresponding Lie groups of complete Lie algebras are complete groups, whose center is unit and all of whose automorphisms are inner. The conception of complete Lie algebras appeared
in the 1940s, shortly after the appearance of complete groups in the 1930s. From the 1940s to 1950s, some important results were obtained concerning complete Lie algebras and complete groups. For example, the semisimple Lie algebras over a field of characteristic $0$ are complete.
Over arbitrary fields, the Lie algebras with non-degenerate Killing form are complete. Then there were few results on complete Lie algebras until the 1980s.

Since the late 1980s, Meng and his collaborators developed a general theory on complete Lie algebras in a series of papers (cf. \cite{MZ,RM} and references therein). Since then more and more complete Lie algebras were known. Following \cite{MZ}, we construct a kind of solvable complete Lie algebras as follows.

Let $T$ be a maximal torus on the nilpotent Lie algebra $\mathfrak{N}$. Then the direct sum $\mathfrak{L}=T\oplus \mathfrak{N}$ is a solvable Lie algebra by setting $$[t_{1}+y_{1}, t_{2}+y_{2}]= t_{1}(y_{2})-t_{2}(y_{1})+[y_{1},y_{2}],$$
where $t_{1}, t_{2}\in T$ and $y_{1}, y_{2}\in \mathfrak{N}$.
\begin{proposition}\rm{(\cite{MZ})}
Let $\mathfrak{N}$ be a nilpotent Lie algebra of maximal rank and $T$ be a maximal torus on $\mathfrak{N}$. Then $\mathfrak{L}=T\oplus \mathfrak{N}$ is complete.
\end{proposition}
Recall that $\tilde{\mathfrak{n}}^{+}$ is a nilpotent Lie algebra of maximal rank and is generated by $C_{P_{i}}$, $1\leq i\leq n$.
$t_{i}$ is defined as $t_{i}(C_{P_{j}})=\delta_{ij}C_{P_{j}}$, $1\leq i\leq n$. Then $T=\oplus^{n}_{i=1}\mathbb{C}t_{i}$ is a maximal torus on $\tilde{\mathfrak{n}}^{+}$. Thus we get a complete Lie algebra $\tilde{\mathfrak{L}}=T\oplus \tilde{\mathfrak{n}}^{+}$. $T$ is also a maximal
torus subalgebra of $\tilde{\mathfrak{L}}$.

Moreover, $\tilde{\mathfrak{L}}$ has a root space decomposition with respect to $T$:
$$\tilde{\mathfrak{L}}=T\oplus \bigoplus_{\alpha \in R(T)}\tilde{\mathfrak{L}}_{\alpha},$$
where $\tilde{\mathfrak{L}}_{\alpha}=\{x|[t,x]=\alpha(t)x, \forall x\in T\}=\tilde{\mathfrak{n}}^{+}_{\alpha}$ and $\dim \tilde{\mathfrak{L}}_{\alpha}=1$.

We remark that, the above procedure could be seen as a generalization of obtaining the Borel subalgebras from the nilpotent radicals of semisimple Lie algebras.

\section*{Acknowledgments}
The authors would like to thank Bangming Deng, Jie Xiao and Fan Xu for stimulating discussions and great help.
The research was initiated while the second author was visiting the Mathematical Institute at the University
of Bonn, and he is grateful to Jan Schr\"{o}er for bringing the paper \cite{RZ} to his attention. The second author also
owns a lot to Yong Jiang for continuous encouragement and valuable suggestions. Furthermore,
it is a pleasure for the second author to thank the Program of Visiting Scholars at Chern Institute of Mathematics,
where part of this paper was written. Lastly, he wants to thank Ming Ding for his invitation and kind hospitality.

\end{document}